\documentclass[12pt, twoside]{article}
\usepackage{amsmath,amsthm,amssymb}
\usepackage{times}
\usepackage{enumerate}

\def\titlerunning#1{\gdef\titrun{#1}}
\makeatletter
\def\author#1{\gdef\autrun{\def\and{\unskip, }#1}\gdef\@author{#1}}
\def\address#1{{\def\and{\\\hspace*{18pt}}\renewcommand{\thefootnote}{}%
\footnote {#1}}%
\markboth{\autrun}{\titrun}}
\makeatother
\def\email#1{e-mail: #1}
\def\subjclass#1{{\renewcommand{\thefootnote}{}%
\footnote{\emph{Mathematics Subject Classification (2010):} #1}}}
\def\keywords#1{\par\medskip
\noindent\textbf{Keywords.} #1}

\numberwithin{equation}{section}

\makeatletter
\def\@currentlabel{2.1}\label{e:dispaa}
\def\@currentlabel{2.21}\label{e:dispau}
\def\@currentlabel{2.22}\label{e:dispav}
\def\@currentlabel{2.23}\label{e:dispaw}
\def\@currentlabel{2.24}\label{e:dispax}
\def\theequation{\thesection.\@arabic\c@equation}
\makeatother

\newcommand{\D}{\Delta}

\newcommand{\inn}{{\quad\hbox{in } }}
\newcommand{\ass}{{\quad\hbox{as } }}
\newcommand{\onn}{{\quad\hbox{on } }}
\newcommand{\ttt}{\tilde }

\newcommand{\LL}{{\tt L}  }

\newcommand{\nn}{ {\nabla}  }

\newcommand{\pp}{ {\partial} }

\newcommand{\R} {\mathbb R}
\newcommand{\cuad}{{\sqcap\kern-.68em\sqcup}}

\newcommand{\DD}{{\mathcal D}}

\newcommand{\dist}{{\rm dist}\, }
\newcommand{\foral}{\quad\mbox{for all}\quad}
\newcommand{\ve}{\varepsilon}

\newcommand{\be}{\begin{equation}}
\newcommand{\ee}{\end{equation}}

\newcommand{\la}{\lambda}

\newcommand{\equ}[1]{(\ref{#1})}

\renewcommand{\theequation}{\thesection.\arabic{equation}}
 
 \newtheorem{lemma}{Lemma}[section]

\newtheorem{theorem}{Theorem}

\newtheorem{remark}{Remark}[section]
\newcommand{\bremark}{\begin{remark} \em}
\newcommand{\eremark}{\end{remark} }

\newcommand{\ov}{\overline}

\renewcommand{\O }{\Omega }
\newcommand{\CC}{\mathcal{C}}
\newcommand{\N}{\mathbb{N}}

\newcommand{\e}{\varepsilon}
\newcommand{\del}{\partial}
\newcommand{\pa}{\partial}
\newcommand{\G}{\Gamma}

\newcommand{\s }{\sigma }
\renewcommand{\a }{\alpha }
\renewcommand{\b }{\beta }
\renewcommand{\d}{\delta }
\newcommand{\zn }{{x_{N}}}

\newcommand{\g }{\gamma}
\newcommand{\U}{\Upsilon}
\renewcommand{\theequation}{\thesection.\arabic{equation}}

\begin{document}


\baselineskip=17pt


\titlerunning{Bubbling on boundary submanifolds }

\title{Bubbling on boundary submanifolds for the Lin-Ni-Takagi problem at higher critical exponents}

\author{Manuel del Pino \and
Fethi Mahmoudi\and
Monica Musso}

\date{}

\maketitle

\address{M. del Pino: Departamento de Ingenier\'{\i}a Matem\'atica and
CMM, Universidad de Chile, Casilla 170 Correo 3, Santiago,
Chile; \email{delpino@dim.uchile.cl}
\and
F. Mahmoudi: Departamento de Ingenier\'{\i}a Matem\'atica and
CMM, Universidad de Chile, Casilla 170 Correo 3, Santiago,
Chile; \email{fmahmoudi@dim.uchile.cl}
\and
M. Musso: Departamento de Matem\'atica, Pontificia Universidad Catolica de
Chile, Avda. Vicu\~na Mackenna 4860, Macul, Chile; \email{mmusso@mat.puc.cl}
}

\



\begin{abstract}
Let $\Omega$ be a bounded domain in $\R^n $ with smooth boundary $\pp\Omega$. We consider the equation $d^2\Delta u - u+
u^{\frac{n-k+2}{n-k-2}} =0\,\hbox{ in }\,\O $, under
zero Neumann boundary conditions, where $\O$ is open, smooth and
bounded and  $d$ is a small positive  parameter.  We assume that there is  a
$k$-dimensional closed, embedded minimal submanifold $K$ of  $\partial\O$, which is
non-degenerate, and certain weighted average of sectional curvatures of $\pp\Omega$ is positive along $K$. Then
we prove the existence of a sequence $d=d_j\to 0$ and a positive solution $u_d$  such that
$$
d^2 |\nabla u_{d} |^2\,\rightharpoonup \, S\,\delta_K
\ass  d \to 0
$$
 in the sense of measures, where $\delta_K     $
stands for the Dirac measure supported on  $K     $ and $S$ is a positive constant.

\

\end{abstract}






\section{Introduction and statement of  main results}

 Let $\Omega $ be a bounded, smooth domain in $\R^n$, $\nu$ the outer unit normal to $\pp\Omega$  and $q>1$.
  The semilinear Neumann elliptic problem
\be d^2\Delta u  -  u  + u^q = 0 \inn \Omega, \quad \frac{\pp
u}{\pp\nu} = 0 \onn \pp\Omega \label{1}\ee has been  widely
considered in the literature for more than 20 years.  In 1988 Lin, Ni
and Takagi \cite{lnt} initiated the  study of this problem for small
values of $d$, motivated by the {\em shadow system} of the
Gierer-Meinhardt model of biological pattern formation \cite{gm}.
 In that context $u$  roughly represents the (steady) concentration of an activating chemical of the process, which is thought to diffuse
 slowly in the region $\Omega$, leaving patterns of high concentration such as small spots or narrow stripes.

\medskip

\medskip
When $n=2$ or $q< \frac{n+2}{n-2}$ the problem is {\em subcritical},
and a positive least energy solution $u_d$ exists
 by a standard compactness argument. This solution corresponds to a minimizer for the
the Raleigh quotient \be Q_d(u) = \frac {  d^2 \int_\Omega |\nn u|^2
+  \int_\Omega |u|^2}{ \left ( \int_\Omega  |u|^{q+1}\right )^{\frac
2{q+1}}}. \label{e0}\ee
 In the papers \cite{lin,lnt,nt1,nt2} the authors
described accurately the asymptotic behavior of $u_d$ as $d\to 0$.
This function maximizes at exactly one point $p_d$ which lies on
$\pp\Omega$. The asymptotic location of the point $p_d$ gets further
characterized as
$$
H_{\pp\Omega}(p_d) \to  \max_{p\in \pp\Omega} H_{\pp\Omega}(p)
$$
where $H_{\pp\Omega} $ denotes the mean curvature of $\pp\Omega$.
Moreover, the asymptotic shape of $u_d$ is indeed highly
concentrated around $p_d$ in the form \be u_d ( x) \approx w\left (
\frac{|x-p_d|} d \right ) \label{ss}\ee where $w(|x|)$ is the unique
positive, radially symmetric solution to the problem \be \Delta w -
w + w^p =0\inn \R^n,\quad \lim_{|x|\to\infty} w(x) = 0,
\label{ww}\ee which decays exponentially. See also \cite{df0} for a
short proof.

\medskip
Construction of single and multiple {\em spike-layer patterns}  for
this problem in the subcritical case has been the object of many
studies, see for instance  \cite{ao1,ao2,b,cao,dy,df0,dfw,gp,g,gw,l,lnw,wei}
and the surveys \cite{ni0,ni}. In particular, in \cite{wei}  it was
found that whenever one has a non-degenerate critical point $p_0$ of the
mean curvature $H_{\pp\Omega}(p)$, a solution with a profile of the
form \equ{ss} can be found with $p_d \to p_0$.

\medskip
It is natural to look for solutions to Problem \equ{1} that exhibit
concentration phenomena as $d\to 0$ not just at points but on higher
dimensional sets.

\medskip
 Given a $k$-dimensional submanifold $\Gamma$ of $\pp\Omega$ and assuming that either $k\ge n-2$ or  $q< \frac{ n-k+2}{n-k-2}$,
  the question is whether there  exists a solution $u_d$
which near $\Gamma$ looks like

\be u_d(x) \approx   w\left ( \frac { \dist (x,\Gamma)} d   \right )
\label{pp0}\ee where now $w(|y|)$ denotes the unique positive,
radially symmetric solution to the problem
$$
\Delta w - w + w^p =0\inn \R^{n-k},\quad \lim_{|y|\to\infty} w(|y|)
= 0.
$$
 In \cite{mmah,mm1,mm2,m},  the authors have established the existence of a solution with the profile \equ{pp0}
  when either $\Gamma=\pp\Omega$ or $\Gamma$ is an  {\em embedded closed minimal submanifold} of $\pp\Omega$,
  which is in addition {\em non-degenerate} in the sense that
its Jacobi operator is non-singular (we recall the exact definitions
in the next section).  This phenomenon is actually quite subtle
compared with concentration at points: existence can only be
achieved along a sequence of values $d \to 0$. The parameter $d$ must actually
remain suitably away from certain values  where resonance
occurs, and the topological type of the solution changes: unlike the
point concentration case, the Morse index of these solutions is very
large and grows as $d\to 0$.

\medskip
It is natural to analyze the critical case $q=\frac{n+2}{n-2}$,
namely the problem \be d^2 \Delta u  -  u  + u^{\frac{n+2}{n-2}} = 0
\inn \Omega, \quad \frac{\pp u}{\pp\nu} = 0 \onn \pp\Omega.
\label{20}\ee
 The lack of compactness of Sobolev's embedding makes it harder to apply variational arguments. On the other hand,  in \cite{am0,wang} it was
proven that a non-constant least energy solution $u_d$ of \equ{20},
minimizer of \equ{e0} exists, provided that $d$ be sufficiently
small. The behavior of $u_d$ as $d\to 0$ has been clarified in the
subsequent works \cite{apy,npt,rey1}: as in the  subcritical case,
$u_d$ concentrates, having a unique maximum point $p_d$ which lies
on $\pp\Omega$ with
$$
H_{\pp\Omega} (p_d ) \to  \max_{p\in \pp\Omega} H_{\pp\Omega}(p).
$$
Pohozaev's identity \cite{Po}  yields nonexistence of positive
solutions to Problem \equ{ww} when $q=\frac{n+2}{n-2}$, and thus the
concentration phenomenon  must necessarily be different. Unlike the
subcritical case,  $u_d(p_d)\to +\infty$ the profile of $u_d$ near
$p_d$ is given, for suitable $\mu_d\to 0$, by \be     u_d(x) \approx
d^{\frac {n-2}2} w_{\mu_d} (|x-p_d|) \label{pp}\ee
 where  $w_\mu(|x|)  $  corresponds to the family of  radial positive solutions of
\be \Delta w  + w^{\frac{n+2}{n-2}} = 0 \inn \R^n \label{w}\ee
namely
\be
w_\mu(|x|)  =   \alpha_n \left (   \frac{ \mu  } { \mu^2 + |x|^2}
\right )^{\frac{n-2}2}, \quad  \alpha_n= (n(n-2))^{\frac{n-2}4},
\label{wmu}\ee
which up to translations, correspond to all positive solutions of
\equ{w}, see \cite{cgs}. The precise concentration rates $\mu_d$ are
dimension dependent, and found in the works \cite{apy1,gl,rey1}. In
particular $\mu_d \sim    d^2 $ for $n\ge 5$, so that $u_d(p_d) \sim
d^{-\frac{n-2}2}$.

\medskip
As in the subcritical case, construction and estimates for bubbling
solutions to Problem \equ{20} have been
 subjects broadly treated. In addition to the above references we refer the reader to
  \cite{am1,amy,dmp,gg,ggz,lww,msw,rey0,rey2,rey,reywei,zqwang1,wwy,weixu}.

\medskip
In particular, in \cite{amy} it was found  that for $n\ge 6$ and a
non-degenerate critical point $p_0$ of the mean curvature
 with $H_{\pp\Omega}(p_0) >0$, there exists a solution whose profile near $p_0$ is given by
\be     u_d(x) \approx  d^{\frac {n-2}2} w_{\mu_d} (|x-p_0|), \qquad\quad
\mu_d  =  a_n\,H_{\pp\Omega}(p_0) ^{\frac 1 {n-2}} \,d^{2}
\label{pp1}\ee
 for certain explicit constant $a_n>0$. See also \cite{rey0,rey1} for the lower dimensional case.  The condition of critical point for $H_{\pp\Omega}$ with
 $H_{\pp\Omega}(p_0)  >0$ turns out to be necessary for the boundary bubbling phenomenon  to take place, see \cite{apy1,gl}.

\medskip
The concentration phenomenon in the critical scenario is more
degenerate than that in the subcritical case, and its features
harder to be detected because of the rather subtle role of the
scaling parameter $\mu$. The purpose of this paper is to unveil the
corresponding analog of a  solution like \equ{pp1} for the
$k$-dimensional concentration question, in  the so far open
critical case of the {\em $k$-th critical exponent} $q=
\frac{n-k+2}{n-k-2}$, namely for the problem

\be d^2 \Delta u  -  u  + u^{\frac{n-k+2}{n-k-2}} = 0 \inn \Omega,
\quad \frac{\pp u}{\pp\nu} = 0 \onn \pp\Omega.
 \label{3}\ee

\medskip

Notice that for the Dirichlet problem solution concentrating  along boundary geodesics near the second critical exponent have been considered by del Pino, Musso and Pacard in \cite{dmpa}.

\

Let $K     $ be  $k$-dimensional embedded submanifold  of
$\pp\Omega$.  Under suitable assumptions we shall find a solution
$u_d(x)$ which for points $x\in \R^n$ near $K     $,  can be
described as
$$
x= p + z  , \quad p\in K     , \quad |z| = \dist (x,K     ),
$$
we have \be     u_d(x) \approx  d^{\frac {n-k-2}2} w_{\mu_d} ( |z|
), \quad \mu_d(p) =  a_{n-k} \bar H(p) ^{\frac 1 {n-k-2}} d^{2}
\label{pp2}\ee where  $w_\mu$ now denotes
$$
w_\mu(|z|)  =    \alpha_{n-k} \left (  \frac \mu   { \mu^2 + |z|^2}
\right )^{\frac{n-k-2}2}.
$$
The form of the quantity $\bar H(p)$  is of course not obvious.  It
turns out to correspond to
 a weighted average of sectional curvatures of $\pp\Omega$ along $K     $, which we shall need to assume positive.
To explain what it is we need some notation.

\medskip
We denote as customary by $T_p\pp\Omega$ the tangent space to
$\pp\Omega$ at the point $p$. We consider the {\em shape operator} $
\LL :  T_p\pp\Omega \to T_p\pp\Omega$ defined as
$$
\LL[e] := -\nn_e \nu(p)
$$
where $\nn_e \nu(p) $ is the directional derivative of the vector
field $\nu$ in the direction $e$. Let us consider the orthogonal
decomposition
 $$T_p\pp\Omega = T_pK\oplus N_pK $$
where $N_pK$ stands for the normal bundle of $K$. We choose
orthonormal bases
 $(e_a)_{a=1,\ldots, k}$ of $T_pK $ and  $(e_i)_{i=k+1,\ldots, n-1}$ of $N_pK$.

Let us consider the $(n-1)\times (n-1)$ matrix ${ H}(p)$
representative of $\LL$ in these bases, namely
$$
H_{\a\b} (p)=  e_\a \cdot \LL[e_\b] .
$$
This matrix also represents the second fundamental form of
$\pp\Omega$ at $p$ in this basis. $H_{\a\a}(p)$ corresponds to the
curvature of $\pp\Omega$ in the direction $e_\a$. By definition, the
mean curvature of $\pp\Omega$ at $p$ is given by the trace of this
matrix, namely

$$
H_{\pp\Omega}(p) =  \sum_{j=1}^{n-1} H_{jj}(p).
$$

In order to state our result we need to consider the mean of the
curvatures in  the directions of   $T_pK$ and $N_pK$, namely the
numbers  $\sum_{i=1}^k H_{ii}(p)$ and $\sum_{j=k+1}^{n-1}
H_{jj}(p)$.

\begin{theorem} \label{teo1} Assume that  $\partial\Omega$
contains a  closed embedded, non-degenerate  minimal submanifold $K$
of dimension $k\ge 1$  with  $n-k\ge 7$, such that
\begin{equation}
\label{condizionesumu} \bar H(p):=   2 \sum_{a=1}^kH_{aa}(p) +
\sum_{j=k+1}^{n-1}H_{jj}(p) >0 \foral p\in K     .
\end{equation}
 Then, for a sequence
$d= d_j\longrightarrow \,0$, Problem \equ{3} has a positive solution
$u_{d}$ concentrating along $K     $ in the sense that expansion
\equ{pp2} holds as $d\to 0$ and besides
$$
d^2 |\nabla u_{d} |^2\,\rightharpoonup \, S_{n-k} \,\delta_K
\ass  d \to 0
$$
 where $\delta_K     $
stands for the Dirac measure supported on  $K     $ and
 $S_{n-k}$ is an explicit positive constant.
\end{theorem}

Condition \equ{condizionesumu}  is new and  unexpected. On the other
hand, it is worth noticing that  \equ{condizionesumu} can be
rewritten as
$$
2 H_{\pp\Omega} (p) -  \sum_{j=k+1}^{n-1}H_{jj}(p) > 0 \foral p\in
K     .
$$
Formally in the case of point concentration, namely $k=0$, this
reduces precisely to $H_{\pp\Omega} (p)>0$, that is exactly the
condition known to be necessary for point concentration. We suspect
that this condition is essential for the phenomenon to  take place.
On the other hand, while the high codimension assumption $n-k\ge 7$
is important in our proof, we expect that a similar phenomenon holds
just provided that $n-k\ge 5$, and with a suitable change in the bubbling scales
for $n-k\ge 3$ (the difference of rates is formally due to the fact that $\int_{\R^{n}} w_\mu^2$ is finite if and only if $n\ge 5$).

\medskip

It will be convenient to rewrite Problem \equ{3} in an equivalent form: Let us
set  $N= n-k$ and  $d^2 = \ve $. We
define
 $$ u(x) = \ve^{-\frac{N-2} 4}
v(\ve^{-1} x ). $$ Then, setting $\O_\ve := \ve^{-1}\O$,  Problem \equ{3} becomes
\begin{equation}\label{eq:pe}
  \begin{cases}
    \Delta v -\ve v +v^{\frac{N+2}{N-2}} =0 & \text{ in } \O_\ve, \\
    \frac{\del v}{\pa \nu} = 0 & \text{ on } \partial \O_\ve.
  \end{cases}
\end{equation}

The proof of the theorem has as a main ingredient the construction
of an approximate solution with arbitrary degree of accuracy in
powers of $\ve$, in a neighborhood of the manifold $K_\ve=\ve^{-1}K$. Later we
built the desired solution by linearizing the equation \equ{eq:pe}
around this approximation. The associated linear operator turns out
to be invertible with inverse controlled in a suitable norm by
certain large negative power of $\ve$, provided that $\ve$ remains
away from certain critical values where resonance occurs. The
interplay of the size of the error and that of the inverse of the
linearization then makes it possible a fixed point scheme.

\medskip
The accurate approximate solution to \equ{eq:pe} is built by using an iterative
scheme of Picard's type which we describe in general next.

Observe that the desired asymptotic behavior \equ{pp2} translates in terms of $v$ as

\be v(x) \approx  \mu_0( \ve z)^{-\frac {N-2}2} w_0\left (
\mu_0^{-1}(\ve z) |\zeta| \right),  \quad x= z + \zeta, \quad z\in
K_\ve, \quad |\zeta|=  \dist (x,K_\ve), \label{op}\ee where
 $$ \mu_0(y) = a_{N} \bar H(y) ^{\frac 1 {N-2}} \quad y= \ve z \in K.$$
Here and  in what follows  $w_0$ designates  the standard
bubble,
\be
w_0(\xi)  = w_0(|\xi|)=  \alpha_N \left (   \frac{ 1  } { 1 + |\xi|^2}
\right )^{\frac{N-2}2}, \quad  \alpha_N= (N(N-2))^{\frac{N-2}4}.
\label{w0}\ee

 We
introduce  the so-called Fermi  coordinates on a neighborhood of $K_\ve:=\ve^{-1}K$, as a suitable tool to describe the approximation \equ{op}.
They are defined  as follows (we refer to
Subsection \ref{ss:fc} for further details):  we parameterize a neighborhood of
$K_\ve$   using the exponential map in $\pp\O_\ve$
\begin{eqnarray*}
&&(z,\bar X,X_N)\in K_\ve \times \R^{N-1}\times \R_+ \mapsto
\U (z,\ov X, X_N)\ :=\
\\
&&\qquad\exp^{\pa \O_\ve}_{
z}( \sum\limits_{i=1}^{N-1}\,X_i\,E_i)- X_N \nu\left( \exp^{\pa
\O_\ve}_{ z}( \sum\limits_{i=1}^{N-1}\,X_i\,E_i)\right).
\end{eqnarray*}
Here the vector fields $E_i(z)$ represent an orthonormal basis of $N_z K_\ve$.
Thus, \equ{op} corresponds to the statement that after expressing $v$ in these coordinates we get
$$
v(z,\bar X,X_N) \approx  \mu_0(\ve z)^{-\frac {N-2}2} w_0\left ( \mu_0^{-1}(\ve z) ( \bar X,X_N )\right )
$$
where suitable corrections need to be introduced if we want further accuracy on the induced error:
we consider a positive smooth function $\mu_\ve
= \mu_\ve (p) $  defined on $K$, a smooth  function  $\Phi_\e\,:K\longrightarrow\, \R^{N-1}$,
 and the change of variables (with some abuse of notation)
\begin{equation}
  v(z,\ov X, X_N)=\mu_{\varepsilon}^{-\frac{N-2}{2}}(\ve z) \, W(\ve^{-1}y,\mu_{\varepsilon}^{-1}(\ve z)(\ov X-\Phi_\e (\ve z)),\,\mu_{\varepsilon}^{-1}(\ve z)X_N ),
\end{equation} \label{b2}
with the new $W$ being a function
\begin{equation*}
  W(z,\xi), \quad z= {y \over \ve}, \quad \ov\xi=\frac{\ov X-\Phi_\e}{\mu_\ve}, \quad \xi_N=\frac{X_N}{\mu_\ve}.
\end{equation*}
We will formally expand $W(z,\xi)$ in powers of $\ve$ starting with
$w_0(\xi)$, with the functions
 $\Phi_\e(y)$ and $\mu_\ve(y)$ correspondingly expanded.

 \medskip
 Substituting into the equation,  we will arrive
formally to linear equations satisfied by the successive remainders of $w_0(\xi)$ (as functions of $\xi$). These linear equations involve the basic linearized operator
$\mathcal{L}:=
-\Delta-pw_0^{p-1}$.

The bounded solvability of the linear equations at
each step of the iteration, is guaranteed by imposing orthogonality conditions of their right-hand
sides,  with respect to $\ker(\mathcal{L})$ in $L^\infty(\R^N)$.  These  orthogonality conditions, amount to choices of the coefficients of the expansions of  $\mu_\ve$
and $\Phi_\ve$: for the latter, the equations involve the Jacobi operator of $K$ and it is where the nondegeneracy
assumption is used. The coefficients for the expansion of $\mu_\ve(\ve z)$ come from algebraic relations, in particular an orthogonality condition
in the first iteration yields
\begin{equation*}
\mu_0 (y):=\,   a_N \left[2 \,\sum_{j=1}^k H_{jj}(y)+ \sum_{i=k+1}^{N+k-1} H_{ii}(y)
\right] \quad y \in K.
\end{equation*}
 This is exactly where the sign condition \equ{condizionesumu} in the
 theorem appears.

\medskip
The rest of the paper is organized as follows. We first introduce some
notations and conventions. Next, we collect some notions in
differential geometry, like the Fermi coordinates (geodesic normal
coordinates) near a minimal submanifold and we expand the
coefficients of the metric near these Fermi coordinates. In Section
3 we expand the Laplace-Beltrami operator. Section 4 will be mainly
devoted to the construction of the approximate solution to our
problem using the local coordinates around the submanifold $K$
introduced before. In Section 5 we define globally the approximation and we write the solution to our problem as the sum
 of the global approximation plus a remaining term. Thus we express our original problem as a non linear problem in the remaining term.
 To solve such problem, we need to understand the invertibility properties of a linear operator. To do so we start expanding a quadratic functional
 associated to the linear problem.
 In Section 6 we develop a linear theory to study
our problem. Then, we turn to the proof of our main theorem in Section 7. Sections 8 and 9 are Appendices, where we postponed the proof of some technical facts
to facilitate the reading of the paper.

\

\setcounter{equation}{0}
\section{Geometric setting}\label{sec4}

\bigskip
In this section we first introduce Fermi coordinates near a
$k$-dimensional submanifold of $\pa \O \subset \R^{n}$ (with
$n=N+k$) and we expand the coefficients of the metric in these
coordinates. Then, we recall some basic notions about minimal and
non-degenerate submanifold.

\medskip
\subsection{ Notation and conventions}
\noindent  Dealing with coordinates, Greek letters like $\a, \b,
\dots$, will denote indices varying between $1$ and $n-1$, while
capital letters like $A, B, \dots$ will vary between $1$ and $n$;
Roman letters like $a$ or $b$ will run from $1$ to $k$, while
indices like $i, j, \dots$ will run between $1$ and $N-1:= n - k -
1$.

\medskip\noindent  $\xi_{1}, \dots, \xi_{N-1}, \xi_{N}$ will denote coordinates
in $\R^{N}=\R^{n-k}$, and they will also be written as $\bar
\xi=(\xi_{1}, \dots, \xi_{N-1})$, $\xi=(\bar \xi,\xi_N)$.

\medskip\noindent  The manifold $K$ will be parameterized with coordinates
$y = (y_1, \dots, y_k)$. Its dilation $K_\e := \frac 1 \e K$ will be
parameterized by coordinates $z=(z_1, \dots, z_k)$ related to the
$y$'s simply by $y = \e z$.

\medskip\noindent  Derivatives with respect to the variables $y$, $z$ or
$\xi$ will be denoted by $\pa_{y}$, $\pa_z$, $\pa_\xi$, and for
brevity sometimes we might use the symbols $\pa_{a}$, $\pa_{\ov a}$
and $\pa_i$ for $\pa_{y_a}$, $\pa_{z_a}$ and $\pa_{\xi_i}$
respectively.

\medskip\noindent In a local system of coordinates, $(\ov{g}_{\a \b})_{\a
\b}$ are the components of the metric on $\pa \O$ naturally induced
by $\R^n$. Similarly, $(\ov{g}_{AB})_{AB}$ are the entries of the
metric on $\O$ in a neighborhood of the boundary. $(H_{\a \b})_{\a
\b}$ will denote the components of the mean curvature operator of
$\partial \O$ into $\R^n$.

\medskip
\subsection{Fermi coordinates on $\pa\O$ near $K$ and expansion of the metric}\label{ss:fc}
Let $K$ be a $k$-dimensional submanifold of $(\partial\O,\ov g)$
($1\le k\le N-1$). We choose
along $K$ a local orthonormal frame field $((E_a)_{a=1,\cdots
k},(E_i)_{i=1,\cdots, N-1})$ which is oriented. At points of $K$, we have the natural splitting
$$T\pa \O=T K \oplus N K$$ where $T K$ is the
tangent space to $K$ and $N K$ represents the normal bundle, which
are spanned respectively by $(E_a)_a$ and $(E_j)_j$.

We denote by $\nabla$ the connection induced by the metric $\ov{g}$ and by
$\nabla^N$ the corresponding normal connection on the normal bundle.
Given $p \in K$, we use some geodesic coordinates $y$ centered at
$p$. We also assume that at $p$ the normal vectors $(E_i)_i$, $i =
1, \dots, n$, are transported parallely (with respect to $\nabla^N$)
through geodesics from $p$, so in particular
\begin{equation}\label{eq:parall}
    \ov g\left(\nabla_{E_a}E_j\,,E_i\right)=0  \quad \hbox{ at } p,
    \qquad \quad i,j = 1, \dots, n, a = 1, \dots, k.
\end{equation}
In a neighborhood of $p$ in $K$, we consider normal geodesic
coordinates
\[
f(y) : = \exp^K_p (y_a\, E_a), \qquad y := (y_{1}, \ldots, y_{k}),
\]
where $\exp^K$ is the exponential map on $K$ and summation over repeated
indices is understood. This yields the coordinate vector fields
$X_a : = f_* (\del_{y_a})$. We extend the $E_i$ along each $\gamma_E(s)$ so that they are parallel
with respect to the induced connection on the normal bundle $NK$.
This yields an orthonormal frame field $X_i$ for $NK$ in a neighborhood of
$p$ in $K$ which satisfies
\[
\left. \nabla_{X_a} X_i \right|_p \in T_p K.
\]

A coordinate system in a neighborhood of $p$ in $\partial\Omega$ is now defined by
\begin{equation}\label{eqF}
F(y,\bar x) := \exp^{\partial\Omega}_{f(y)}( x_i \, X_i), \qquad
(y,\bar x) :=( y_{1}, \ldots, y_{k},x_1, \ldots, x_{N-1}),
\end{equation}
with corresponding coordinate vector fields
\[
X_i : = F_* (\del_{x_i}) \qquad \mbox{and} \qquad  X_a : = F_*
(\del_{y_a}).
\]


By our choice of coordinates, on $K$ the metric $\ov{g}$ splits in
the following way
\begin{equation}\label{eq:splitovg}
    \ov g(q) = \ov g_{ab}(q)\,d y_a\otimes d y_b+\ov
g_{ij}(q)\,dx_i\otimes dx_j, \qquad \quad q \in K.
\end{equation}
We denote by $\Gamma_a^b(\cdot)$ the 1-forms defined on the normal
bundle, $NK$, of $K$ by the formula~
\begin{equation}\label{eq:Gab}
\ov g_{bc} \Gamma_{ai}^c:=  \ov g_{bc} \Gamma_a^c(X_i)=\ov g(\nabla_{X_a}X_b,X_i) \quad \hbox{at } q=f(y).
\end{equation}
Notice that
\begin{equation}
\label{eq:min}
    K \hbox{ is minimal } \qquad \Longleftrightarrow \qquad \sum_{a=1}^k\G^a_a(E_i)
    = 0 \quad \hbox{ for any } i = 1, \dots N-1.
\end{equation}

Define $q=f(y)=F(y,0)\in K$ and let $(\tilde g_{ab}(y))$ be the induced metric on $K$.

When we consider the metric coefficients in a neighborhood of $K$,
we obtain a deviation from formula \eqref{eq:splitovg}, which is
expressed by the next lemma. The proof follows the same ideas as Proposition 2.1 in \cite{mmp} but we give it here for completeness.   See also the book \cite{Schoen-Yau}.  
We will denote by $R_{\a\b\g\d}$ the components of the curvature tensor
with lowered indices, which are obtained by means of the usual ones
$R_{\b\g\d}^\s$ by~
\begin{equation}
\label{ctens} R_{\a\b\g\d}=\ov
g_{\a\s}\,R_{\b\g\d}^\s.\end{equation}

\begin{lemma} At the point $F(y,\bar x)$, the following expansions hold,
for any $a=1,...,k$ and any $i,j=1,...,N-1$, we have
\[
\begin{array}{rllll}\ov g_{ij}&=\delta_{ij}+\frac{1}{3}\,R_{istj}\,x_s\,x_t\,
+\,{\mathcal O}(|x|^3);\\[3mm]
\ov g_{aj}&=
{\mathcal O}(|x|^2);\\[3mm]
\ov g_{ab}&=\tilde{g}_{ab}-\bigg\{\tilde{g}_{ac}\,\Gamma_{bi}^c+\tilde{g}_{bc}\,\Gamma_{ai}^c\bigg\}\,x_i
+\bigg[R_{sabl}+\tilde g_{cd}\Gamma_{as}^c\,\Gamma_{bl}^d \bigg]x_s x_l+\,{\mathcal O}(|x|^3).
\end{array}
\]
Here $R_{istj}$ (see \equ{ctens}) are  computed at the point of
$K$ parameterized by $(y,0)$. \label{lemovg}
\end{lemma}

\begin{proof}
The Fermi coordinates above are defined such that  the metric coefficients
\[
g_{\alpha \beta} = g( X_\alpha , X_\beta)
\]
is equal to $\delta_{\alpha\beta}$ at $p=F(0,0)$ and $g_{ab} =\tilde g_{ab}(y)$ at the point $q=F(y,0)$   furthermore,
$g(X_a, X_i) =0$ in some neighborhood of $q$ in $K$.
A Taylor expansion of the metric $\ov g_{\a\b}(\ov x,y)$ at $q$ is given by
\begin{eqnarray*}
\ov g_{\a\b}&=&\ov g(X_\a,X_\b)\ |_q+X_j\ov g(X_\a,X_\b)\ |_q\,x_j+O(|x|^2)\\
&=&\ov g(X_\a,X_\b)\ |_q+\ov g(\nabla_{X_j}X_\a,X_\b)\ |_q\,x_j+\ov g(\nabla_{X_j}X_\b,X_\a)\ |_q\,x_j+O(|x|^2).
\end{eqnarray*}
Since $\ov g(X_b,X_i)=0$ in a neighborhood of $q$ we have
\begin{eqnarray*}
0=X_b\ov g(X_i,X_a)&=&\ov g(\nabla_{X_b}X_i,X_a)+\ov g(X_i,\nabla_{X_b}X_a)\\
&=&\ov g(\nabla_{X_i}X_b,X_a)+\ov g(X_i,\nabla_{X_b}X_a).
\end{eqnarray*}
This implies in particular that
$$
\ov g(\nabla_{X_i}X_b,X_a)=-\ov g(X_i,\nabla_{X_b}X_a)=- \Gamma_{ai}^c \tilde g_{cb}.
$$
Then at first order expansion we have
\begin{eqnarray*}
\ov g_{ab}&=&\ov g(X_a,X_b)\ |_q+\ov g(\nabla_{X_j}X_a,X_b)\ |_q\,x_j+\ov g(\nabla_{X_j}X_b,X_a)\ |_q\,x_j+O(|x|^2)\\
&=&\tilde g_{ab}-\bigg(\Gamma_{ai}^c \tilde g_{cb}+\Gamma_{bi}^c \tilde g_{ca}\bigg)\,x_i+O(|x|^2).
\end{eqnarray*}
Similarly using Formula \eqref{eq:parall} we get
\begin{eqnarray*}
\ov g_{ai}&=&\ov g(X_a,X_i)\ |_q+\ov g(\nabla_{X_j}X_a,X_i)\ |_q\,x_j+\ov g(\nabla_{X_j}X_i,X_a)\ |_q\,x_j+O(|x|^2)\\
&=&O(|x|^2).
\end{eqnarray*}
On the other hand, since every vector field
$X \in N_{q} K$  is tangent to the geodesic  $s \longrightarrow
\exp_{q}^{\partial\Omega} (s X)$, we have
\[
\left. \nabla_{X_\ell+X_j}(X_\ell + X_j) \right|_q=0.
\]
Which clearly implies that
\[
\left. (\nabla_{X_\ell}X_j + \nabla_{X_j}X_\ell) \right|_q = 0.
\]
Then the following expansion holds
\begin{eqnarray*}
\ov g_{ij}&=&\ov g(X_i,X_j)\ |_q+\ov g(\nabla_{X_l}X_i,X_j)\ |_q\,x_l+\ov g(\nabla_{X_l}X_j,X_i)\ |_q\,x_l+O(|x|^2)\\
&=&\delta_{ij}+O(|x|^2).
\end{eqnarray*}

To compute the terms of order two in the Taylor expansion it suffices to compute  $X_k \,
X_k \, \ov g_{\alpha\beta}$ at  $q$ and polarize  (i.e.\ replace $X_k$ by
$X_i + X_j$). We have
\begin{equation}
X_k \, X_k \, \ov g_{\alpha\beta} = \ov g (\nabla_{X_k}^2 X_\alpha,X_\beta)
+\ov g(  X_\alpha, \nabla_{X_k}^2 X_\beta) + 2 \,\ov g( \nabla_{X_k}
X_\alpha, \nabla_{X_k} X_\beta).  \label{eq:ll}
\end{equation}

Now, using the fact every normal vector
$X \in N_{q} K$ is tangent to the geodesic  $s \longrightarrow
\exp_{q}^{\partial\Omega} (s X)$, then
\[
\left. \nabla_{X} X \right|_{q}= \left. \nabla^2_X X \right|_{q}=0
\]
for every  $X \in N_{q} K$. In particular, choosing  $X = X_k + \e \, X_j$, we obtain
\[
0 = \nabla_{X_k + \e X_j}\nabla_{X_k + \e X_j}(X_k + \e X_j) \ _{|q}
\]
for every $\e$, which implies
$\nabla_{X_j}\nabla_{X_k}X_k \ _{|p} = -2
\nabla_{X_k}\nabla_{X_k}X_j \ _{|p}$, and hence
\[
\left. 3 \, \nabla_{X_k}^2 X_j \right|_q =  R(X_k,X_j) \, X_k|_q .
\]
We then deduce from   (\ref{eq:ll}) that
\[
\left. X_k \, X_k \, \ov g_{ij} \right|_q= \frac{2}{3} \, \ov g( R(X_k,
X_i) \, X_k , X_j)|_q.
\]
On the other hand we have
\[
\nabla_{X_k}^2 X_\gamma = \nabla_{X_k}\nabla_{X_\gamma }X_k =
\nabla_{X_\gamma} \nabla_{X_k}X_k + R(X_k,X_\gamma) \, X_k.
\]
Hence
\[
\begin{array}{rllll}
X_k \, X_k \, g_{ab}  & = & 2  \, \ov g(R(X_k, X_a) X_k, X_b) + 2 \, \ov g(
\nabla_{X_k} X_a , \nabla_{X_k} X_b )
\\[3mm]
&  +  & \ov g (\nabla_{X_a} \nabla_{X_k} X_k ,X_b) + \ov g( X_a ,
\nabla_{X_b } \nabla_{X_k} X_k).
\end{array}
\]
Now using the fact that
$\nabla_X X =0 \ _{|q}$ at   $q\in K$  for every  $X \in N_{q} K$,  the definition of $\Gamma_{ak}^c$ in \eqref{eq:Gab} and the formula
$$
R(X_k, X_a) X_l=R_{kal}^\gamma X_\gamma
$$
we deduce that at the point $q$
\begin{eqnarray*}
\left. X_k \, X_k \, \ov g_{ab} \right|_q &=& 2  \, \ov g(R(X_k, X_a) X_k,
X_b) + 2 \, \tilde g_{cd}\,\Gamma_{ak}^c \, \Gamma_{bk}^d\\
&=& 2R_{kak}^c \ov g(X_c,X_b)+ 2 \, \tilde g_{cd}\,\Gamma_{ak}^c \, \Gamma_{bk}^d\\
&=& 2R_{kak}^c \ \tilde g_{cb}+ 2 \, \tilde g_{cd}\,\Gamma_{ak}^c \, \Gamma_{bk}^d\\
&=& 2R_{kabk}+ 2 \, \tilde g_{cd}\,\Gamma_{ak}^c \, \Gamma_{bk}^d.
\end{eqnarray*}
This proves the Lemma.

\end{proof}

Next we introduce a parametrization of a neighborhood in $\O$ of $ q
\in \pa \O$ through the map $\U$ given by
\begin{equation}\label{eq:fe}
    \U(y, x) = F( y, \bar x) +
x_N \nu(y, \bar x), \qquad x = (\bar x,x_N) \in \R^{N-1} \times \R,
\end{equation}
where $F$ is the parametrization introduced in \equ{eqF} and
$\nu(y,\bar x)$ is the inner unit normal to $\partial \O$ at
$F(y, \bar x)$. We have
$$
\frac{\partial \U}{\partial y_a} = \frac{\partial F}{\partial
y_a}(y, \bar x) +  \zn \frac{\partial \nu}{\partial y_a}(y, \bar x);
\qquad \qquad \frac{\partial \U}{\partial x_i} = \frac{\partial
F}{\partial x_i}(y, \bar x) +\zn \frac{\partial \nu}{\partial
x_i}(y, \bar x).
$$
Let us define the tensor matrix ${H}$ to be given by
\begin{equation}\label{eq:dn}
    d \nu_x [v] =  -{H}(x)[v].
\end{equation}
We thus find
\begin{eqnarray}\label{eq:dfe1}
&&  \frac{\partial \U}{\partial y_a} = \left[ Id -  \zn
   {H}(y, \bar x) \right] \frac{\partial
F}{\partial y_a}(y, \bar x); \\
&& \frac{\partial \U}{\partial x_i} = \left[ Id - \zn
  {H}(y, \bar x) \right] \frac{\partial
F}{\partial x_i}(y, \bar x).
\end{eqnarray}
Differentiating $\U$ with respect to $\zn$ we also get
\begin{equation}\label{eq:dfe3}
  \frac{\partial \U}{\partial \zn} = \nu(y, \bar x).
\end{equation}
Hence, letting $g_{\a \b}$ be the coefficients of the flat metric
$g$ of $\R^{N+k}$ in the coordinates $(y, \bar x,\zn)$, with easy
computations we deduce for $\tilde{y} = (y, \bar x)$ that
\begin{equation}\label{eq:geij}
    g_{\a\b} (\tilde{y},\zn) = \ov{g}_{\a\b} ( \tilde{y})
  - \zn \left( H_{\a\d} \ov{g}_{\d \b} + H_{\b \d} \ov{g}_{\d \a}
  \right) ( \tilde{y}) +  x_N^2 H_{\a \d} H_{\s \b}
  \ov{g}_{\d \s} (\tilde{y});
\end{equation}
\begin{eqnarray}\label{eq:ge2233}
  g_{\a N} \equiv 0; \qquad \qquad g_{NN} \equiv 1.
\end{eqnarray}
In the above expressions, with $\alpha$ and $\beta$ we denote any
index of the form $a=1, \ldots , k$ or $i=1, \ldots , N-1$.

We first provide a Taylor expansion of the coefficients of the
metric $g $. {From} Lemma \ref{lemovg} and formula \eqref{eq:geij}
we have immediately the following result.

\begin{lemma}\label{l:expgeuz}
For the (Euclidean) metric $g$ in the above coordinates we have the
expansions
\begin{eqnarray*} &&g_{ij}=\delta_{ij} - 2  x_N H_{ij} +
\frac{1}{3} \,R_{istj}\,x_s\,x_t +  x_N^2 (H^2)_{ij}
\,+\,{\mathcal O}(|x|^3), \quad 1\le i,j\le N-1;\\[3mm]
 &&g_{aj}=-  x_N \bigg(H_{aj} + \tilde g_{ac}H_{cj}\bigg)
 +{\mathcal O}(|x|^2), \quad 1\le a\le k , \, 1\le j\le  N-1;\\[3mm]
 &&g_{ab}=\tilde{g}_{ab}-\big\{\tilde{g}_{ac}\,\Gamma_{bi}^c+\tilde{g}_{bc}\,\Gamma_{ai}^c\big\}\,x_i
 -x_N\,\big\{ H_{ac}\,\tilde g_{bc}+ H_{bc}\,\tilde g_{ac}\big\}+
\big[R_{sabl}+\tilde g_{cd}\Gamma_{as}^c\,
\Gamma_{dl}^b \big]x_s x_l
 \\
&&+x_N^2 (H^2)_{ab} +x_N\,x_k\bigg[H_{ac}\big\{\tilde g_{bf}\Gamma_{ck}^f+\tilde g_{cf}\Gamma_{bk}^f\big\}+
H_{bc}\big\{\tilde g_{af}\Gamma_{ck}^f+\tilde g_{cf}\Gamma_{ak}^f\big\}
\bigg]+ {\mathcal O}(|x|^3), \\
&&\qquad 1\le a,b\le k;
\\[3mm]
 && g_{a N} \equiv 0, \quad a=1, \ldots , k; \qquad g_{i N} \equiv 0, \quad i=1, \ldots , N-1; \qquad \qquad g_{NN} \equiv 1.
\end{eqnarray*}
In the above expressions $H_{\a  \b} $ denotes the components of the
matrix tensor ${H}$ defined in \equ{eq:dn}, $R_{istj}$ are the
components of the curvature tensor as defined in \equ{ctens},
$\Gamma_{a}^b(E_i)$ are defined in \equ{eq:Gab}. Here we have set
$$
(A^2)_{\a\b}=A_{\a i}A_{i\b}+\tilde g_{cd}A_{\a c}A_{\b d}.
$$

Furthermore, we have the validity of the following expansion for the
log of the determinant of $g$
\begin{eqnarray*}
 \log\big(\det g\big)& = & \log\big(\det \tilde g\big)- 2x_N {\rm tr }\,(H )-2\G^b_{bk}\,x_k+ \frac13  R_{miil}x_m x_l\\&+&
 \bigg( \tilde g^{ab}\,R_{mabl}-\G_{am}^{c}
\G_{cl}^{a}  \bigg)x_m x_l -x_N^2 \,{\rm tr }\,(H^2)+ \mathcal{O}(|x|^3).
\end{eqnarray*}
\end{lemma}
Recall first that $K$ minimal implies that $\G^b_{bk}=0$.  The expansions of the  metric in the above lemma follow from Lemma \ref{lemovg} and formulas \eqref{eq:geij}-\eqref{eq:ge2233} while the expansion of the log of the determinant of $g$ follows from the fact that  $g=G+M$ with
$$
G=\bigg(\begin{matrix}
  \tilde g & 0 \\
  0& Id_{\R^n}
 \end{matrix}\bigg)\qquad \hbox{and }\quad M =\mathcal{O}(|x|),
$$
then we have the following expansion
$$
\log\big(\det g\big)=\log\big(\det G\big)+{\rm tr}(G^{-1}M)-\frac12 {\rm tr}\bigg((G^{-1}M)^2\bigg)+O(||M||^3).
$$

We are now in position to give the expansion of the Laplace-beltrami operator. Recall that
$$
\Delta_g u= \frac1{\sqrt{\det g}}\,\partial_\a\bigg( \sqrt{\det g} \,g^{\a\b}\, \partial_\b u\bigg)
$$
where summation over repeated indices is understood and where $g^{\a\b}$ denotes the entries of the inverse of the metric $(g_{\a\b})$.
The above formula can be rewritten as
\begin{eqnarray*}
\Delta_g u=g^{\a\b}\, \partial^2_{\a\b} u+\partial_\a\big(g^{\a\b}\big)\, \partial_\b u+\frac12\,\partial_\a\big( \log(\det g)\big) \,g^{\a\b}\, \partial_\b u.
\end{eqnarray*}
Using the expansions in Lemma \ref{l:expgeuz} we have the validity of the following expansion for the Laplace-Beltrami operator
\begin{lemma}\label{l:explaplacian}
In the above coordinates the Laplace-Beltrami operator can be expanded as
\begin{eqnarray*}\label{explap}
\Delta_g u(x,y)&=& \D_K u+\partial^2_{ii}u+\partial^2_{NN}u-tr(H)\partial_Nu+2x_N H_{ij}\,\partial^2_{ij}u\\[3mm]
&+&x_N^2 \,Q(H)_{ij}\partial^2_{ij}u-x_N\, tr(H^2)\,\partial_Nu+2x_N H_{ab}\G_{ai}^b\partial_{i}u\\[3mm]
&+& \bigg(\frac23  R_{mlli}+ \tilde g^{ab}\,R_{iabm}-\G_{am}^{c}
\G_{ci}^{a}  \bigg)x_m \,\partial_{i}u-\frac13 \,R_{islj}x_sx_l\,\partial^2_{ij}u\\[3mm]
&+& 2x_N\,\big( H_{aj} +\tilde g^{ac}\, H_{cj}\big)\partial^2_{aj}u +\big(O(|x|^2)+O(|x|)\,\partial_a \mathfrak{F}^{a\b}(y,x) \big)\,
 \partial_\b u\\[3mm]
&+&  \bigg\{\tilde{g}^{ac}\,\Gamma_{bi}^c+\tilde{g}^{bc}\,\Gamma_{ai}^c\bigg\}\,x_i\,\partial^2_{ab}u
 +x_N\,\bigg\{ H_{ac}\,\tilde g^{bc}+ H_{bc}\,\tilde g^{ac}\bigg\}\,\partial^2_{ab}u \\[3mm]
&+& {\mathcal O}(|x|^3)\,\partial^2_{ij}u+{\mathcal O}(|x|^2)\,\partial^2_{aj}u+{\mathcal O}(|x|^2)\,\partial^2_{ab}u.
\end{eqnarray*}
Here the term $Q(H)$ is a quadratic term of $H$ given by
\begin{equation}\label{eq:Qij}
 Q(H)_{ij}=3x_N^2 \,H_{ik}\,H_{kj}+x_N^2\,\bigg( 2\,H_{ia}\,H_{aj}+\tilde g^{ab}\,H_{ia}\,H_{bj} \bigg),
\end{equation}
while the term $\mathfrak{F}^{a\b}$ ($\b=b$ or $\b=j$) is given by the formulas
$$
{\mathcal O}(|x|)\,\mathfrak{F}^{ab}(y,x) =\bigg(g^{ab}-\tilde g^{ab}\bigg)+\frac12\bigg( \log\big(\det g\big)\,g^{ab}-\log\big(\det \tilde g\big)\,\tilde g^{ab} \bigg)
$$
and
$$
{\mathcal O}(|x|)\,\mathfrak{F}^{aj}(y,x) =g^{aj}+\frac12\, \log\big(\det \tilde g\big)\,g^{aj}.
$$
\end{lemma}

\begin{proof}
Using the expansion of Lemma \ref{l:expgeuz} and the fact that if $g=G+M$ with
$$G=\bigg(\begin{matrix}
  \tilde g & 0 \\
  0& Id_{\R^N}
 \end{matrix}\bigg)\qquad \hbox{and }\quad M=\mathcal{O}(|x|)$$
 then
 $$
 g^{-1}=G^{-1}-G^{-1}MG^{-1}+G^{-1}MG^{-1}MG^{-1}+O(||M||^3),
 $$
  it is easy to check that the following expansions hold true
\begin{eqnarray*} &&g^{ij}=\delta_{ij} +2  x_N H_{ij} -
\frac{1}{3} \,R_{istj}\,x_s\,x_t +  x_N^2 Q(H)_{ij}
\,+\,{\mathcal O}(|x|^3), \quad 1\le i,j\le N-1;\\[3mm]
 &&g^{aj}= x_N \bigg(H_{aj} + \tilde g^{ac}H_{cj}\bigg)+{\mathcal O}(|x|^2), 1\le a\le k , \, 1\le j\le  N-1;\\[3mm]
 &&g^{ab}=\tilde{g}^{ab}+\bigg\{\tilde{g}^{ac}\,\Gamma_{bi}^c+\tilde{g}^{bc}\,\Gamma_{ai}^c\bigg\}\,x_i
 +x_N\,\bigg\{ H_{ac}\,\tilde g^{bc}+ H_{bc}\,\tilde g^{ac}\bigg\}+ {\mathcal O}(|x|^2), \\
 &&\qquad 1\le a,b\le k;
\\[3mm]
 && g^{a N} \equiv 0, \quad a=1, \ldots , k; \qquad g^{i N} \equiv 0, \quad i=1, \ldots , N-1; \qquad \qquad g^{NN} \equiv 1.
\end{eqnarray*}
The proof of the Lemma follows at once.
\end{proof}

\subsection{Non degenerate Minimal submanifold}

\noindent Denoting by $C^\infty (NK)$ the space of smooth normal
vector fields on $K$. Then, for $\Phi \in C^\infty (NK)$, we define
the one-parameter family of submanifolds $t \mapsto K_{t,\Phi}$ by
\begin{equation}
\label{eqKtPhi} K_{t,\Phi}:=\{ \exp_{y}^{\pa \O} (t \Phi( y)) \,:\,
 y\in K \}.
\end{equation}
The  first variation formula of the volume is the equation~
\begin{equation}
\label{eqfvf} \frac{d}{dt}\bigg|_{t=0}\mbox{Vol}(K_{t,\Phi})=\int_K
\langle\Phi, {\bf h}\rangle_N\,dV_K,
\end{equation}
where   ${\bf h}$ stands for the {\em mean curvature} (vector) of
$K$ in $\pa \O$, $\langle \cdot, \cdot \rangle_N$ denotes the
restriction the metric $\ov{g}$ to $N K$, and $dV_K$ the volume
element of $K$.

\

\noindent A submanifold $K$ is said to be  {\em minimal} if it is a
critical point for the volume functional, namely if~
\begin{equation}
\label{eqms} \frac{d}{dt}\bigg|_{t=0}\mbox{Vol}(K_{t,\Phi})=0 \qquad
\mbox{for any $\Phi\in C^\infty (NK)$}
\end{equation}
or, equivalently by \eqref{eqfvf}, if the mean curvature ${\bf h}$
is identically zero on $K$. It is possible to prove that condition \equ{eqms} is equivalent to \equ{eq:min}.

\noindent The {\em Jacobi operator} $\mathfrak{J}$ appears in the
expression of the second variation of the volume functional for a
minimal submanifold $K$
\begin{equation}
\label{eqjac}
\frac{d^2}{dt^2}\bigg|_{t=0}\mbox{Vol}(K_{t,\Phi})=-\int_K
\langle{\mathfrak J}\Phi,\Phi\rangle_N\,dV_K; \qquad \quad \Phi\in
C^\infty (NK),
\end{equation}
and it is given by~
\begin{equation}\label{eq:jacobi}
    {\mathfrak
J}\Phi:=-\Delta_K^N\Phi+\mathfrak{R}^N\Phi-\mathfrak{B}^N\Phi,
\end{equation}
where $\mathfrak{R}^N, \mathfrak{B}^N :NK\rightarrow NK$ are defined
as~
$$
\mathfrak{R}^N\Phi=(R(E_a,\Phi)E_a)^N; \qquad \qquad \ov
g(\mathfrak{B}^N\Phi,n_K):=\G_b^a(\Phi)\G_a^b(n_K),
$$
for any unit normal vector $n_K$ to $K$.
The Jacobi operator defined in \equ{eq:jacobi} expressed in Fermi coordinates take the expression
\begin{equation}
\label{Jacobinodeg} ({\mathfrak J}\Phi)^l =  - \Delta_K
\,\Phi^l+\bigg( \tilde g^{ab}\,R_{mabl}- \G_a^c(E_m) \G_c^a(E_l)  \bigg) \,\Phi^m ,
\quad l=1, \ldots , N-1
\end{equation}
where $R_{maal} $ and $\G_a^c (E_m)$ are smooth functions on $K$ and they are defined respectively in \equ{ctens} and \equ{eq:Gab}.
\noindent A submanifold $K$ is said to be {\em non-degenerate} if
the Jacobi operator ${\mathfrak J}$ is invertible, or equivalently
if the equation ${\mathfrak J}\Phi=0$ has only the trivial solution
among the sections in $N K$.
\setcounter{equation}{0}

\section{Expressing the equation in coordinates}

We recall from \equ{eq:pe} that we want to find a solution to  the problem
\begin{equation}\label{eq:pe1}
  \begin{cases}
    \Delta v -\ve v +v_+^{\frac{N+2}{N-2}} =0 & \text{ in } \O_\ve, \\
    \frac{\del v}{\pa \nu} = 0 & \text{ on } \partial \O_\ve.
  \end{cases}
\end{equation}

The first element to construct an approximate solution to our
problem is the standard bubble
\begin{equation}
\label{w00} w_0 (\xi ) = {\alpha_N \over (1+|\xi|^2)^{N-2 \over 2} }, \quad \alpha_N= (N(N-2))^{\frac{N-2}4}
\quad \foral \quad \xi \in \R^N
\end{equation}
solution to
\begin{equation}
\label{w0} \Delta w + w^{N+2 \over N-2} = 0 \quad {\mbox {in}} \quad
\R^N_+, \quad {\partial w \over \partial \xi_N} = 0
 \quad {\mbox {in}} \quad \partial \R^N_+.
\end{equation}
It is well known that all positive and bounded solutions to \equ{w0}
are given by the family of functions
$$
\mu^{-{N-2 \over 2}} w_0 \left( {x-P \over \mu} \right),
$$
for any $\mu>0$ and any point $P=(P_1, \ldots , P_{N-1} , 0) \in
\partial \R^N_+$. The solution we are building will have at main
order the shape of a copy of $w_0$, centered and translated along
the $k$-dimensional manifold $K$ inside $\partial \Omega$. In the original variables in $\Omega$, this
approximation will be scaled by a small factor, so that it will turn out to be very much
concentrated around the manifold $K$.

\medskip
To describe this approximation, it will be useful to introduce the
following change of variables. Let $(y,x)\in \R^{k+N}$ be the local
coordinates along $K$ introduced in \equ{eq:fe}. Let $z={y\over \ve}
\in K_\ve$ and $X={x\over \ve} \in \R^N$. A parametrization of a
neighborhood (in $\O_\e$) of ${q\over \e} \in  K_\e \subset \partial
\Omega_\e$ close to $K_\e$ is given by the map $\Upsilon_\e$ defined
by
\begin{equation} \label{eq:feee}
\Upsilon_\e (z, \bar X , X_N ) = {1\over \e} \Upsilon (\e z, \e X
),\quad \, X=(\bar X,X_N)\in \R^{N-1}\times \R^+
\end{equation}
where $\Upsilon$ is the parametrization given in \equ{eq:fe}.

Given a positive smooth function $\mu_\ve
= \mu_\ve (y) $  defined on $K$  and a smooth function $\Phi_\e\,:K\longrightarrow\, \R^{N-1}$ defined by $
  \Phi_\e (y)=( \Phi_\e^1(y) , \ldots , \Phi_\e^{N-1} (y) ) $, $ y\in\,K,
$ we consider the change of variables
\begin{equation} \label{b}
  v(z,\ov X, X_N)=\mu_{\varepsilon}^{-\frac{N-2}{2}}(\ve z) \, W(z,\mu_{\varepsilon}^{-1}(\ve z)(\ov X-\Phi_\e (\ve z)),\,\mu_{\varepsilon}^{-1}(\ve z)X_N ),
\end{equation}
with the new $W$ being a function
\begin{equation}
\label{b0}
  W=W(z,\xi), \quad z= {y \over \ve}, \quad \ov\xi=\frac{\ov X-\Phi_\e}{\mu_\ve}, \quad \xi_N=\frac{X_N}{\mu_\ve}.
\end{equation}
To emphasize the dependence of the above change of variables on
$\mu_\ve$ and $\Phi_\e$, we will use the notation
\begin{equation} \label{defTT}
v={\mathcal T}_{\mu_\ve , \Phi_\e} (W)  \quad \Longleftrightarrow  v
\quad {\mbox {and}} \quad W \quad {\mbox {satisfy \equ{b}}}.
\end{equation}

We assume now that the functions $\mu_\e$ and
$\Phi_\e$ are uniformly bounded, as $\ve \to 0$, on $K$. Since the original variables
$(y,x)\in \R^{k+N}$ are local coordinates along $K$, we let the
variables $(z, \xi )$ vary in the set $\DD $ defined by
\begin{equation}
\label{defD} \DD  = \{ (z, \bar \xi , \xi_N ) \, : \, \ve z \in K,
\quad |\bar \xi | <{\delta \over \ve }, \quad 0< \xi_N < {\delta
\over \ve } \}
\end{equation}
for some  fixed positive number $\delta$. We will also use
the notation $ \DD  = K_\ve \times \hat \DD $, where $K_\ve = {K\over
\ve}$ and
$$\hat \DD  = \{ (\bar \xi , \xi_N ) \, : \, |\bar \xi | < {\delta \over \ve }, \quad 0< \xi_N < {\delta \over \ve } \}.
 $$

Having the expansion of the metric coefficients obtained in Section
\ref{sec4}, we easily get the expansion of the metric in the expanded variables: letting
$g_{\alpha , \beta}^\e$ be the coefficients of the
metric $g^\e$, we have
$$
g_{\alpha , \beta}^\e (z, x )= g_{\alpha , \beta} (\e z , \e x)
$$
where $g_{\alpha , \beta}$ are given in Lemma \ref{l:expgeuz}.
With an easy computation we deduce the following
\begin{lemma} \label{metricepsilon}
For the (Euclidean) metric $g^\e$ in the above coordinates $(z,X )$
we have the expansions
\begin{eqnarray*} &&g^\e_{ij}=\delta_{ij} - 2 \e X_N H_{ij} +
\frac{\e^2}{3} \,R_{istj}\,X_s\,X_t +  \e^2 X_N^2 (H^2)_{ij}
\,+\,{\mathcal O}(\e^3(|X|^3), \\
&&\quad 1\le i,j\le N-1;\\[3mm]
 &&g_{aj}^\e=-  \e X_N \bigg(H_{aj} + \tilde g^\e_{ac}H_{cj}\bigg)+{\mathcal O}(\e^2|X|^2)\\
 &&\quad 1\le a\le k , \, 1\le j\le  N-1;\\[3mm]
 &&g_{ab}^\e=\tilde{g}^\e_{ab}-\e\bigg\{\tilde{g}^\e_{ac}\,\Gamma_{bi}^c+\tilde{g}^\e_{bc}\,\Gamma_{ai}^c\bigg\}\,X_i
 -\e X_N\,\bigg\{ H_{ac}\,\tilde g^\e_{bc}+ H_{bc}\,\tilde g^\e_{ac}\bigg\}\\&&\qquad+
\e^2\,\bigg[R_{sabl}+\tilde g^\e_{cd}\Gamma_{as}^c\,
\Gamma_{dl}^b \bigg]X_s X_l+
 \e^2X_N^2 (H^2)_{ab}\\
&&\qquad  +\e^2X_N\,X_k\bigg[H_{ac}\big\{\tilde g^\e_{bf}\Gamma_{ck}^f+\tilde g^\e_{cf}\Gamma_{bk}^f\big\}+
H_{bc}\big\{\tilde g^\e_{af}\Gamma_{ck}^f+\tilde g^\e_{cf}\Gamma_{ak}^f\big\}
\bigg]+ {\mathcal O}(\e^3|X|^3), \\
&&\quad 1\le a,b\le k;
\\[3mm]
 && g_{a N}^\e \equiv 0, \quad a=1, \ldots , k; \qquad g_{i N}^\e \equiv 0, \quad i=1, \ldots , N-1; \qquad \qquad g^\e_{NN} \equiv 1.
\end{eqnarray*}
In the above expressions $H_{\a  \b} $ denotes the components of the
matrix tensor ${H}$ defined in \equ{eq:dn}, $R_{istj}$ are the
components of the curvature tensor as defined in \equ{ctens},
$\Gamma_{ai}^b$ are defined in \equ{eq:Gab} and $\tilde g_{ab}^\e (z)= \tilde g_{ab} (\e z)$.
\end{lemma}
\begin{lemma}\label{lemlogdet} We have the validity of the following expansions for the
square root of the determinant of $g^\e$ and the log of determinant of $g^\e$
\begin{eqnarray}
 \sqrt{\det g^\e}& = & \sqrt{\det \tilde g^\e} \bigg \{1- \e X_N tr (H )+ \frac{\e^2}6  R_{miil}  X_m X_l + \frac{\e^2}2
 \bigg( (\tilde {g}^\e)^{ab}\,R_{mabl}-\G_{am}^c
\G_{cl}^a  \bigg) X_m   X_l \nonumber \\[3mm]
& + & \frac{\e^2}2 X_N^2  {tr (H)}^2 -\e^2 X_N^2 \,tr(H^2)\bigg\}+ \ve^3 \mathcal{O}(|X|^3), \label{expdeterminante}
\end{eqnarray}
and
\begin{eqnarray*}
 \log\big(\det g^\e\big)& = & \log\big(\det \tilde g^\e\big)- 2\e X_N {\rm tr }\,(H )+ \frac{\e^2}3  R_{miil}\,X_m X_l\\&+&
 \e^2\bigg((\tilde {g}^\e)^{ab}\, R_{mabl}-\G_{am}^{c}
\G_{cl}^{a}  \bigg)X_m X_l -\e^2 X_N^2 \,{\rm tr }\,(H^2)+ \mathcal{O}(\e^3|X|^3).
\end{eqnarray*}
\end{lemma}


We are now  in a position to expand the Laplace Beltrami
operator in the new variables $(z, \xi )$ in terms of the parameter
$\ve$, of the functions $\mu_\ve (y)$ and
$\Phi_\e (y)$.
This is the content of next Lemma, whose proof we
postpone to
 Section \ref{alaplacian}.

\begin{lemma} \label{scaledlaplacian}
Given  the change of variables defined  in \eqref{b}, the following
expansion for the Laplace Beltrami operator holds true
\begin{equation}
\label{lap1}
 \mu_\ve^{{N+2 \over 2}} \Delta v =  {\mathcal A}_{\mu_\ve , \Phi_\e } (W) :=  {\mu_\ve^2}  \D_{K_\e} W +
\Delta_{\xi} W + \sum_{\ell=0}^5 {\mathcal A}_\ell W + B(W).
\end{equation}
Above, the expression ${\mathcal A}_k$ denotes the following
differential operators
\begin{equation} \label{D0}
\begin{array}{rlllll}
{\mathcal A}_0 W
&= & \ve^2 \mu_\ve D_{\ov \xi}\, W \,[\D_K \Phi_\e ]  -
\ve^2\,\mu_\ve \,\D_K
 \mu_\ve \,\left( \gamma
W + D_\xi W \, [\xi] \right) \\[3mm]
 & + &   \ve^2\,| \nabla_K\mu_\ve|^2 \left[ D_{\xi\xi} W \, [
\xi]^2 +
2 (1+\gamma ) D_\xi W [\xi] + \gamma (1+ \gamma ) W \right]   \\[3mm]
& + &  \ve^2 \nabla_K \mu_\ve \,\cdot\,\left\{ 2D_{\ov\xi\,\ov\xi} W[\ov\xi] + N
D_{\ov\xi} W
\right\}\, [\nabla_K  \Phi_\ve ] + \ve^2 D_{\ov\xi\,\ov\xi} W \,[\nabla_K \Phi_\e]^2\\[3mm]
& - & 2\, \ve \mu_\ve \,\tilde g^{ab}\,\left[  D_\xi (\del_{\bar a} W ) [\del_b
\mu_\ve \xi] + D_{\ov \xi} (\del_{\bar a} W )[ \del_b
\Phi_\e ] +
 \gamma  \del_a\mu_\ve\, \del_{\bar b} W \right],
\end{array}
\end{equation}
where we have set $\gamma=\frac{N-2}{2}$,
\smallskip
\begin{equation} \label{D1}
\begin{array}{rlllll}
 {\mathcal A}_1 \, W & = & \,\sum\limits_{i,j} \bigg[ 2\mu_\ve \ve
H_{ij}  \xi_N  -{\ve^2\over 3}\,\,\sum\limits_{m,l} R_{mijl}
(\mu_\ve \xi_m + \Phi_\e^m )
(\mu_\ve \xi_l + \Phi_\e^l )  \\[3mm]
& + & \mu_\ve^2 \ve^2 \xi_N^2\,Q(H)_{ij}+\mu_\ve \ve^2
\xi_N\,\sum\limits_{l} \mathfrak{D}_{Nl}^{ij}\,( \mu_\ve \xi_l
+\Phi_\e^l )
  \bigg] \del^2_{ij} W,
\end{array}
\end{equation}
where the functions $\mathfrak{D}_{Nk}^{ij}$ are smooth functions of
the variable $z= \frac{y}{\ve}$ and uniformly bounded. Furthermore,
\smallskip
\begin{equation}
\label{D4} {\mathcal A}_2 W =\ve^2 \mu_\ve \sum\limits_j \bigg[ \sum_s
\frac23 R_{mssj}+\sum\limits_{m,a,b}
 \big( \tilde g_\e^{ab}\, R_{mabj}- \G_{am}^b
\G_{bj}^a \big)\bigg](\mu_\ve\xi_m+\Phi_\e^m)
 \del_j W
\end{equation}
and
\begin{equation}
\label{D5} {\mathcal A}_3 W \, = \bigg[ -\ve {\rm tr }(H)
-2\,\mu_\ve\, \ve^2 {\rm tr }(H^2)\,\xi_N-2\ve^2
\,\sum\limits_{i,a,b}(\mu_\ve
\xi_i+\Phi_\ve^i)H_{ab}\G_{bi}^a)\bigg]
 \mu_\ve \del_N W.
\end{equation}
Moreover
\smallskip
\begin{equation}
\label{D2} {\mathcal A}_4 W   =
4\,\ve\,\mu_\ve\xi_N\sum\limits_{a,j} H_{aj} \left(-\ve  D_y (\del_j
W)  [\del_a\Phi_\e] + \mu_\ve \del^2_{\bar aj} W - \ve
\del_a\mu_\ve(\gamma \del_j W + D_\xi(\del_j W)  [\xi] )\right)
\end{equation}
and
\smallskip
\begin{equation}
\label{D3}
\begin{array}{rlllll}
{\mathcal A}_5 W \, &  =  & \, \left( \sum\limits_{a,j}
\mathfrak{D}_j^a \ve^2 [\mu_\ve  \xi_j + \Phi_\e^j] + \ve^2\mu_\ve\,
\mathfrak{D}_N^a \,\xi_N  \right) \times \\[3mm]
& & \left\{ \mu_\ve \left[ - \ve D_{\ov\xi} W \, [\del_a \Phi_\e ] +
\mu_\ve \del_{\bar a} W - \ve \del_a \mu_\ve (\gamma W +
D_\xi W \, [\xi] ) \right] \right\}
\end{array}
\end{equation}
where $\mathfrak{D}_j^a$ and $\mathfrak{D}_N^a$ are smooth functions
of $z= \frac{y}{\ve}$. Finally,
the operator $B(v)$ is defined below, see \eqref{B}.

\smallskip
We recall that the symbols $\pa_{a}$, $\pa_{\ov a}$ and $\pa_i$
denote the derivatives with respect to $\pa_{y_a}$, $\pa_{z_a}$ and
$\pa_{\xi_i}$ respectively.

\end{lemma}

\medskip
After performing the change of variables in
\eqref{b}, the original equation in $v$  reduces locally close to
$K_\ve = {K\over \ve}$ to the following equation in $W$
\begin{equation}
\label{adesso}
 -{\mathcal A}_{\mu_\ve , \Phi_\e } W +\ve \mu_\ve^2 W-  W^{p} =0 ,
\end{equation}
where ${\mathcal A}_{\mu_\ve , \Phi_\e }$ is defined in \eqref{lap1}
and $p={N+2 \over N-2}$. We denote by ${\mathcal S}_\ve$ the
operator given by \equ{adesso}, namely
\begin{equation}
\label{Sep} {\mathcal S}_\ve (v ) := -{\mathcal A}_{\mu_\ve ,
\Phi_\e } v +\e \mu_\ve^2v-  v^{p}.
\end{equation}
In order to study equation \equ{adesso} in the set $ (z,\xi )$, with
$z \in K_\ve$, $|\bar \xi |\leq {\delta \over \ve} $
 and $0< \xi_N \leq {\delta \over \ve}$,  we will first construct an approximate solution  to
\equ{adesso} in the whole space $K_\ve \times \R^{N-1} \times [0,
\infty)$ (see Section \ref{s:aprsol}). Then, by using proper cut off
functions, we will build a first approximation to \equ{adesso} in
the original region $z \in K_\ve$, $|\bar \xi |\leq {\delta \over
\ve} $
 and $0< \xi_N \leq {\delta \over \ve}$.

\medskip

The basic tool for this construction is a linear theory we describe
below.
\medskip

Let us recall the well known fact that, due to the invariance under
translations and  dilations of equation \equ{w0}, and since $w_0$ is
a non-degenerate solution for \equ{w0},
we have that the functions
 \begin{equation}
 \label{lezetas}
Z_j (\xi ) = {\partial w_0 \over \partial \xi_j} , \quad j=1, \ldots
, N-1\quad {\mbox {and}} \quad Z_0 (\xi ) = \xi \cdot \nabla w_0
(\xi )  + \frac {N-2}2 w_0 (\xi )
\end{equation}
are the only bounded solutions to the linearized equation around
$w_0$ of problem \equ{w0}
$$
-\Delta \phi - p w_0^{p-1} \phi =0 \quad {\mbox {in}} \quad \R^{N-1}
\times \R^+, \quad {\partial \phi \over \partial \xi_N} = 0 \quad
{\mbox {on }} \xi_N =0.
$$

Let us now consider a smooth function  $a:K \to \R$ with $a(y) \geq
\la >0$ for all $y \in K$ and a function $g:K  \times \R^{N-1}
\times \R^+ \to \R$ that depends smoothly on the variable $y \in K$.
Recall that a variable $z\in K_\ve$ has the form $\e z=y \in K$.

\medskip

We want to find a linear theory for the following linear problem
\begin{equation}\label{eq:eqw}
 \left\{
    \begin{array}{ll}
    - \D_{\R^N} \phi - p w_0^{p-1} \phi +\e\, a(\e z) \phi  =g  &\quad   \hbox{ in } \R^{N}_+\\[3mm]
    \frac{\pa \phi }{\pa \xi_N} = 0, &\quad \hbox{ on } \{ \xi_N = 0\}\\[3mm]
    \int_{\R^{N-1} \times [0, \infty)} \phi (\e z, \xi ) Z_j (\xi ) \, d\xi = 0 & \foral z \in K_\e, \, j=0, \ldots N-1.
    \end{array}
  \right.
\end{equation}

To do so we first define the following norms: Let $\delta >0$ be a
positive small fixed number and $r$ be a positive number. For a function
$w$ defined in $K_\ve \times \R^{N-1} \times [0, \infty)$, we define
\begin{equation}\label{eqinftynu}
\|w\|_{\e,r}:=\sup_{(z,\xi)\in K_\ve \times \{|\xi|\le
\frac{\delta}{\sqrt{\e}}\}}\left( \,(1+|\xi|^2)^{r \over
2}|w(z,\xi)| \right)+ \sup_{(z,\xi)\in K_\e\times \{|\xi|\ge
\frac{\delta}{\sqrt{\e}}\}}\left( \e^{-{r \over 2}}|w(z,\xi)|
\right).
\end{equation}

Let $\sigma \in (0,1)$. We further define
\begin{eqnarray}\label{normsigma}
\| w \|_{\e, r, \sigma} &:=& \| w \|_{\e , r} + \sup_{(z,\xi)\in
K_\ve \times \{|\xi|\le \frac{\delta}{\sqrt{\e}}\}}\left(
\,(1+|\xi|^2)^{r
+\sigma\over 2} [w]_{\sigma, B(\xi , 1)} \right)\\[3mm]
& +& \sup_{(z,\xi)\in K_\ve \times \{|\xi|\geq
\frac{\delta}{\sqrt{\e}}\}}\left( \,\e^{-{r +\sigma\over 2}}
[w]_{\sigma, B(\xi , 1)} \right)\nonumber
\end{eqnarray}
where we have set
\begin{equation}\label{eqnorms}
[w]_{\sigma, B(\xi , 1)} := \sup_{ \xi_1, \xi_2\in B(\xi , 1)}
\frac{|w(z,\xi_2)-w(z,\xi_1)|}{|\xi_1-\xi_2|^\sigma}.
\end{equation}

We have the validity of the following lemma.

\medskip
\begin{lemma}\label{l:freed} Let $r$ be a number with  $2<r < N$ and $\sigma \in (0,1)$.
Let $a : K \to \R$ be a smooth function, such that $a (y) \geq \la
>0$ for all $y \in K$. Let $g : K \times \R^{N-1} \times [0, \infty)
\to \R $ be a function that depends smoothly on the variable $y \in
K$, such that $ \| g \|_{\ve , r} < +\infty$
and
$$
\int_{\R^{N-1} \times [0, \infty)} g(\e z, \xi ) Z_j (\xi ) \, d\xi
= 0 \quad \foral z \in K_\e, \quad  j=0, \ldots N-1.
$$
Then there exist a positive constant $C$  such that for all sufficiently small $\ve$ there is a solution $\phi$ to
Problem \equ{eq:eqw} such that
\begin{equation}
\label{est0} \| D^2_\xi \phi \|_{\ve , r , \sigma } + \| D_\xi \phi
\|_{\ve , r -1 , \sigma } +\|\phi \|_{\ve, r- 2 , \sigma}\le C
\|g\|_{\ve,r ,\sigma}.
\end{equation}
Furthermore, the function $\phi$ depends smoothly on the variable
$\e z$, and the following estimates hold true: for any integer $l$
there exists a positive constant $C_l$ such that
\begin{equation}
\label{est1} \| D^l_z \phi \|_{\ve , r- 2 , \sigma }  \le C_l \left(
\sum_{k\leq l} \|D^k_z g\|_{\ve,r , \sigma}\right).
\end{equation}
\end{lemma}

\bigskip
\begin{proof} The proof of this lemma will be divided into several steps.

{\bf Step 1}. \ \ We start showing the validity of an a-priori
external estimate for a solution to problem \equ{eq:eqw}. Assume
$\phi $ is a solution to problem \equ{eq:eqw}. Given $R>0$, we claim that
\begin{equation}
\label{mar1} | \phi (z, \xi ) | \leq C \left (\|  \phi \|_{L^\infty
(|\xi|=\delta R \ve^{-{1\over 2}})} + \ve^{{r-2 \over 2}} \| g
\|_{\ve , r} \right),
\end{equation}
for all $z \in K_\ve$ and $|\xi |>R\delta \ve^{-{1\over 2}}$.

Let $R>0$ be a  fixed number, independent of $\ve$. In the region
$|\xi |>R \delta  \ve^{-{1\over 2}}$ the function $\phi$ solves
$$
-\Delta \phi + \ve b(\e z, \xi ) \phi = g
$$
where
$$
b(\e z,\xi ) = a(\e z) - {p w_0^{p-1} \over \ve} = a(\e z) + \ve
\Theta_\ve (\xi ),
$$
with $\Theta_\ve $ a function  uniformly bounded in the considered
region, as $\ve \to 0$. Thus we have that $b$ uniformly positive and bounded as $\ve \to 0$.
Using the maximum principle, we get
$$
\| \phi \|_{L^\infty (|\xi|>\delta R \ve^{-{1\over 2} }) }\leq C
\left( \ve^{-1} \|  g \|_{L^\infty (|\xi |>\delta R \ve^{-{1\over
2}})} + \|  \phi \|_{L^\infty (|\xi|=\delta R \ve^{-{1\over 2}})}
\right)
$$
$$
\leq C  \left( \ve^{{r-2 \over 2}}\|  g \|_{\ve , r}  + \|  \phi
\|_{L^\infty (|\xi|=\delta R \ve^{-{1\over 2}})} \right),
$$
which gives \equ{mar1}.

\smallskip

{\bf Step 2}. \ \ Assume now that $\phi$ is a solution to
\equ{eq:eqw}. We will now prove that there exists $C>0$ such that
\begin{equation}
\label{mar2} \|\phi \|_{\ve, r- 2}\le C \|g\|_{\ve,r}.
\end{equation}
We argue by contradiction: let us assume that there exist sequences
$\ve_n \to 0$, $g_n$ with $\| g_n \|_{\ve_n , r} \to 0$ and
solutions $\phi_n$  to \equ{eq:eqw} with $\| \phi_n \|_{\ve_n ,  r
-2} =1$.

Let $z_n \in K_{\e_n}$ and $\xi_n$ be such that
$$
|\phi_n (\e_n z_n , \xi_n )| = \sup  |\phi_n (y, \xi )|.
$$
We may assume that, up to subsequences, $(\e_n z_n ) \to \bar y $ in
$K$. On the other hand, from Step 1, we get that
$$
\sup_{z \in K_{\e_n}, |\xi | >\delta R \ve_n^{-{1\over 2}}}
\ve_n^{{r - 2 \over 2}} |\phi_n (\e_n z, \xi ) | \leq C R^{-{r-2
\over 2}},
$$
thus choosing $R$ sufficiently large, but independent of $\ve_n$, we
have that
$$ \sup\limits_{z \in K_{\e_n}, |\xi | >\delta R
\ve_n^{-{1\over 2}}} \ve_n^{{r - 2 \over 2}} |\phi_n (z, \xi ) |
$$
is arbitrarily small. In particular one gets that $|\xi_n | \leq C
\ve_n^{-{1\over 2}}$ for some positive constant $C$ independent of
$\e_n$.

Let us now assume that there exists a positive constant $M$ such
that $|\xi_n |\leq M$. In this case, up to subsequences, one gets
that $\xi_n \to \xi_0$. We then consider the functions $\tilde
\phi_n ( z, \xi ) = \phi_n ( z , \xi + \xi_n )$. This is a sequence
of uniformly bounded functions, and the sequence $(\tilde \phi_n )$
converges uniformly over compact sets of $K \times \R^{N-1} \times
\R^+$ to a function $\tilde \phi$ solution to
$$
 \left\{
    \begin{array}{ll}
    - \D \tilde \phi - p w_0^{p-1} \tilde \phi =0 &\quad   \hbox{ in } \R^{N}_+\\
    \frac{\pa \tilde \phi }{\pa \xi_N} = 0, &\quad \hbox{ on } \{ \xi_N =
    0\}.
    \end{array}
  \right.
  $$
Since the orthogonality conditions pass to the limit, we get that
furthermore
$$
    \int_{\R^{N-1} \times [0, \infty)} \tilde \phi (y, \xi ) Z_j (\xi ) \, d\xi = 0 \quad \foral y \in K, \quad \foral j=0, \ldots N-1.
    $$
These facts imply that $\tilde \phi \equiv 0$, that is a
contradiction.

\medskip
Assume now that $\lim\limits_{n \to \infty} |\xi_n | = \infty$ and
define the function $\tilde \phi_n (z, \xi ) = \phi_n (z, |\xi_n |
\xi + \xi_n )$. Clearly $\tilde \phi_n$ satisfies the equation
$$
\Delta \tilde \phi_n + p C_N {|\xi_n |^2 \over (1+| \, |\xi_n | \xi
+ \xi_n |^2 )^2 } \tilde \phi_n -|\xi_n |^2 \ve_n  a \tilde \phi_n =
|\xi_n |^2 g (z, |\xi_n | \xi +\xi_n ).
$$
Consider first the case in which  $\lim_{n \to \infty} \ve_n |\xi_n
|^2 = 0$. Under our assumptions, we have that $\tilde \phi_n$ is
uniformly bounded and it converges locally over compact sets to
$\tilde \phi$ solution to
$$
\Delta \tilde \phi = 0 \quad {\mbox {in}} \quad \R^N, \quad |\tilde
\phi | \leq C |\xi |^{2-r}.
$$
Since $2<r<N$, we conclude that $\tilde \phi \equiv 0 $, which is a
contradiction.

Consider now the other possible case, namely that  $\lim_{n \to
\infty} \ve_n |\xi_n |^2 = \beta >0$. Then, up to subsequences we
get that $\tilde \phi_n$ converges uniformly over compact sets to
$\tilde \phi$ solution to
$$
\Delta \tilde \phi - \beta a \tilde \phi = 0 \quad {\mbox {in}}
\quad \R^N, \quad |\tilde \phi | \leq C |\xi |^{2-r}.
$$
This implies that $\tilde \phi \equiv 0$, which is a contradiction.
Taking into account the result of Step 1, the proof of \equ{mar2} is
completed.

\smallskip
{\bf Step 3}. \ \  We shall now show that there exists $C>0$ such
that, if $\phi$ is a solution to \equ{eq:eqw}, then
\begin{equation}
\label{mar3}   \| D_\xi \phi \|_{\ve ,
r -1  } +\|\phi \|_{\ve, r- 2 }\le C \|g\|_{\ve,r }.
\end{equation}

Let us first assume we are in the region $|\xi |<\delta
\ve^{-{1\over 2}}$, and $z \in K_\ve$. Thus, using estimate
\equ{mar2}, we have that $\phi$ solves $ -\Delta \phi =  \hat g $
in $ |\xi |< \delta \ve^{-{1\over 2}} $ where $|\hat g |\leq {\|g\|_{\ve, r}
\over (1+|\xi|^{r } )}$.

Let us now fix a point $e \in \R^N$, $|e|=1$ and a positive number $R>0$.
Perform the change of variables $ \tilde \phi (z,t) =   R^{r-2} \phi (z, Rt
+3Re)$, so that
$$
\Delta \tilde \phi =  \tilde g \quad {\mbox {in}}
\quad |t|\leq 1
$$
where $\ttt g(t,z)=  R^r \hat g( z, Rt + 3Re)$, so that   $$\|\ttt\phi \|_{L^\infty(B(0,2))} + \|\tilde g  \|_{L^\infty(B(0,2))} \leq  \|g\|_{\ve, r} .$$
 Elliptic
estimates
give then that $\| D \tilde \phi \|_{L^\infty (B(0,1))} \leq C
\| \tilde g \|_{L^\infty (B(0,2))}$. This inequality implies then that
$$
\| (1+|\xi |)^{r-1} D \phi \|_{L^\infty (|\xi |\leq \delta
\e^{-{1\over2}} )} \leq C \| (1+|\xi |)^{r} g \|_{L^\infty (|\xi
|\leq 2\delta \e^{-{1\over2}} )}.
$$

Assume now that $|\xi |>\delta \e^{-{1\over 2}}$. In this region the
function $\phi$ solves
$$
-\Delta \phi= \hat g
$$
where
 $|\hat g|\leq C\|g\|_{\ve,r} \e^{r \over 2}$, and $ | \phi | \leq C\|g\|_{\ve,r}\ve^{{r - 2 \over 2}}$.  After scaling out $\ve^{\frac 12}$, elliptic estimates yield $
 | D\phi | \leq C \ve^{{r - 1 \over 2}}$. This concludes the proof of \equ{mar3}.

\smallskip

{\bf Step 4}. \ \  We shall now show that
\begin{equation}
\label{mar4} \| D^2_\xi \phi \|_{\ve , r , \sigma } + \| D_\xi \phi
\|_{\ve , r -1 , \sigma } +\|\phi \|_{\ve, r- 2 , \sigma}\le C
\|g\|_{\ve,r ,\sigma}.
\end{equation}

Let us first assume we are in the region $|\xi |<\delta
\ve^{-{1\over 2}}$.
 Thus, we write that
$\phi$ solves $ -\Delta \phi = \hat g $ in $ |\xi |< \delta
\ve^{-{1\over 2}} $ where, thanks to the $C^1$-estimate of Step 3,  $\|\tilde g\|_{\ve,r ,\sigma} \leq C\|g\|_{\ve,r ,\sigma}$.

Arguing as in the previous step, we fix a point $e \in \R^N$, $|e|=1$ and a
positive number $R>0$ and let $\ttt\phi$ and $\ttt g$ be defined as in Step 3.  Elliptic estimates
give then that $\|D^2 \tilde \phi \|_{C^{0,\sigma}
(B(0,1))} \leq C \| \tilde g \|_{C^{0,\sigma} (B(0,2))}$.
This inequality gets translated in terms of $\phi$ as the desired Schauder estimate within $|\xi| \le  \delta \ve^{-\frac 12}$.  The H\"older estimate for  $D\phi$
follows by interpolation.
In the region $|\xi |>\delta \e^{-{1\over 2}}$, we argue exactly as in the
proof of Step 3.   This concludes the
proof of \equ{mar4}.

\smallskip

{\bf Step 5}. \ \ Now we shall prove the existence of the solution
$\phi$ to problem \equ{eq:eqw}. We consider first the following
auxiliary problem: find $\bar \phi$ and $\alpha :K_\e \to \R$
solution to
\begin{eqnarray} \label{au}
 \left\{
    \begin{array}{ll}
    - \D \bar \phi - p w_0^{p-1} \bar \phi +\e\, a(y) \bar \phi  =g  +\alpha (z) Z (\xi ) &\quad   \hbox{ in }
    \R^{N}_+\\[3mm]
    \frac{\pa \bar \phi }{\pa \xi_N} = 0, &\quad \hbox{ on } \{ \xi_N = 0\}\\[3mm]
   \int_{\R^{N-1} \times \R^+} \bar \phi (z, \xi ) Z_j (\xi ) \, d\xi =
   \int_{\R^{N-1} \times \R^+} \bar \phi (z, \xi ) Z(\xi ) \, d\xi = 0 & z \in K_\e, \\
    ( j=0, \ldots N-1),
    \end{array}
  \right.
\end{eqnarray}
where $Z$ is
 the first eigenfunction, with corresponding first eigenvalue $\lambda_0 >0$, in
$L^2(\R^N)$  of the problem
\begin{equation}
\label{denotareZ0}
 \Delta_\xi \bar \phi + pw_0(\xi )^{p-1} \bar \phi  = \la \bar \phi\quad   \mbox{ in} \quad
 \R^{N}.
\end{equation}
The above problem is variational: for any fixed $z \in K_\e$,
solutions to \equ{au} correspond to critical points of the energy
functional
$$
E(\bar \phi ) = {1\over 2} \int_{\R^{N-1} \times \R^+}  |
\nabla \bar \phi|^2 - p w_0^{p-1} \bar \phi^2 +\ve a \bar \phi^2
 - \int_{\R^{N-1} \times \R^+} g\bar\phi
$$
for functions $\bar \phi \in H^1 (\R^{N-1} \times \R^+)$, that
satisfy $$\int_{\R^{N-1} \times \R^+} \bar \phi Z_j =\int_{\R^{N-1}
\times \R^+} \bar \phi Z = 0$$ for all $j=0 , \ldots , N-1$. This
functional is smooth, uniformly bounded from below, and satisfies
the Palais-Smale condition. We thus conclude that $E$ has a minimum,
which gives a solution to \equ{au}.

Observe now that, multiplying the equation in \equ{au} against $Z$,
integrating in $\R^{N-1} \times \R^+$, and using the orthogonality
conditions in \equ{au}, one easily gets that
\begin{equation}  \label{auu}
\alpha (z) = \int_{\R^{N-1} \times \R^+ }g(z, \xi ) Z (\xi ) \, d\xi
\quad \foral z \in K_\e.
\end{equation}

Given $(\bar \phi , \alpha )$ solution to \equ{au}, we define
$$
\phi = \bar \phi +\beta Z \quad {\mbox {with}} \quad \beta (z) =
{\int g(z, \xi ) Z (\xi ) \, d \xi \over \lambda_0 + \ve a(\e z)}.
$$
A straightforward computation shows that $\phi$ is a solution to
\equ{au}.

\medskip
Finally, estimate \equ{est1} follows by a direct differentiation of
equation \equ{eq:eqw} and a use of estimate \equ{est0}. This
concludes the proof of Lemma \ref{l:freed}.
\end{proof}

\setcounter{equation}{0}
\section{Construction of approximate solutions}\label{s:aprsol}
\smallskip

This section will be devoted to build an approximate solution to
Problem \eqref{adesso} locally close to $K_\ve$,  using  an iterative method that we describe below: Let $I$ be an integer. The
expanded variables $(z, \xi)$ will be defined as in \equ{b0} with
\begin{equation}
\label{deffmu} \mu_\ve(\ve z)=  \mu_0+\mu_{1,\e}+\dots +\mu_{I,\e},
\end{equation}
where $\mu_0 , \, \mu_{1,\e} , \ldots , \mu_{I , \e} $ are
smooth functions on $K$, with $\mu_0$ positive,   and
\begin{equation}
\label{deffPhi} \Phi_\e(\ve z) =  \Phi_{1,\e} + \dots +  \Phi_{I,\e} ,
\end{equation}
where $\Phi_{1,\e} , \,  \dots , \, \Phi_{I,\e}$ are smooth
functions defined along $K$ with values in $\R^{N-1}$.
 In the $(z, \xi )$ variables, the shape of the approximate solution will be  given by
\begin{equation}\label{vIe}
  W_{I+1,\e}(z, \ov \xi, \xi_N)= w_0\left( \xi \right) +
  w_{1,\e}\left(z, \xi \right)
  + \dots +  w_{I+1,\e}\left(z, \xi \right), \quad \xi = (\bar \xi , \xi_N )
\end{equation}
where $w_0$ is given by  \equ{w00}
and  the functions $w_{j, \e}$'s for $j\ge 1$ are to be
determined so that the above function $W_{I+1,\e}$
satisfies formally
$$
-{\mathcal A}_{\mu_\ve , \Phi_\e } W_{I+1,\e}+\ve\,
\mu_\ve^2\,W_{I+1,\e}-W_{I+1,\e}^\frac{N+2}{N-2}=\mathcal{O}(\ve^{1+{I\over
2}}) \quad {\mbox {in}} \quad K_\ve \times \R^{N-1} \times \R^+.
$$
This can be done expanding the equation \eqref{adesso} formally in
powers of $\e$ (and in terms of $\mu_\ve$ and $\Phi_\e$) for $W=
W_{I+1,\e}$ (using basically Lemma \ref{scaledlaplacian}) and
analyzing each term separately. For example, looking at the
coefficient of $\e$ in the expansion we will determine $\mu_0$ and
$w_{1 \e}$, while looking at the coefficient of $\e^{1+{j\over 2}}$ we will determine
$w_{j,\e}$, $\mu_{j-1 , \e}$ and $\Phi_{j-1 , \e}$, for $j = 2, \dots,
I+1$. In this procedure we use crucially the non degeneracy
assumption on $K$ (which implies the invertibility of the Jacobi
operator) when considering the projection on some elements of the
kernel of the linearization of \eqref{eq:pe} at $w_0$, while when
projecting on the remaining part of the kernel we have to choose the
functions $\mu_{j, \ve}$. This Section is devoted to do this construction.

\
\begin{lemma} \label{Construction}
For any integer $I \in \N$ there exist  smooth functions $\mu_\e :K
\to \R$ and $\Phi_\e : K \to \R^{N-1}$,  
such that
\begin{equation}
\label{bf2} \| \, \mu_{\e}\|_{L^\infty (K)} +\| \partial_a
\mu_{\e}\|_{L^\infty (K)} +\|\partial^2_a \, \mu_{\e}\|_{L^\infty
(K)} \leq C
\end{equation}
\begin{equation}
\label{bf3} \| \, \Phi_{\e}\|_{L^\infty (K)} +\| \partial_a
\Phi_{\e}\|_{L^\infty (K)} +\|\partial^2_a
\Phi_{\e}\|_{L^\infty (K)} \leq C
\end{equation}
for some positive constant $C$, independent of $\e$. In particular, we have
\begin{equation} \label{boh}
\| \, \mu_\e  - \mu_0 \|_{L^\infty (K)} \to 0 , \quad \|
\Phi_\e  - \Phi_0 \|_{L^\infty (K)} \to 0
\end{equation}
where $\mu_0 $ is the function defined explicitly as
$$
\mu_0 (y) = {\int_{\R^N_+} \xi_N |\partial_1 w_0|^2 \over \int_{\R^N_+} w_0^2 } \left[ 2 H_{aa} (y) + H_{ii} (y) \right].
$$
By assumption \equ{condizionesumu}, it turns out that the function $\mu_0$ is strictly positive along $K$.
Moreover $\Phi_0$ is
a smooth function along $K$ with values in $\R^{N-1}$. Furthermore there exists a positive
function $W_{I+1, \e} :K_\ve \times \R^N_+ \to \R$ such that
$$
{\mathcal A}_{\mu_\e , \Phi_\e} (W_{I+1 , \e} ) -\ve \mu_\e^2 W_{I+1
, \e} + W_{I+1, \e}^p = {\mathcal E}_{I+1 , \e} \quad {\mbox {in}}
\quad K_\ve \times \R^N_+
$$
$$
{\partial W_{I+1 , \e} \over \partial \nu} = 0 \quad {\mbox {on}}
\quad \partial \R^N_+
$$
with
\begin{equation}
\label{boh1} \| W_{I+1 , \e} - W_{I , \e} \|_{\e , N-4 , \sigma } \leq C
\e^{1+{I\over 2}}
\end{equation}
and
\begin{equation} \label{bf4}
\| {\mathcal E}_{I+1 , \e} \|_{\e , N-2 , \sigma} \leq C \e^{1+{I+1
\over 2}}.
\end{equation}

\end{lemma}

\medskip
We should emphasize that $\int w_0^2$ is indeed finite thanks to the fact that $N\ge 5$, and we are actually assuming that $N\ge 7$.
The rest of this Section is devoted to do this construction.

\medskip

\noindent $\bullet$ {\bf  Construction of $w_{1,\e}$ :}
 Using Lemma
\ref{scaledlaplacian}, we
\textit{formally} have
\begin{eqnarray*}
  - {\mathcal A}_{\mu_\ve , \Phi_\e } W_{1,\e} + \ve\,\mu_\ve^2 W_{1,\e} - W_{1,\e}^p & = &
  - \D_{\R^{N}} w_0 - w_0^p \\&+& \left( - \D_{\R^{N}} w_{1,\e}
  - p w_0^{p-1} w_{1,\e}+\e \mu_0^2w_{1,\e}\right) \\
   & + & \ve\mu_0 \left[ H_{\alpha\alpha} \partial_{\xi_N} w_0 - 2 \xi_N H_{ij} \partial^2_{ij}
   w_0 + \mu_0(y)\,w_0\right]\\
  &+& \mathcal{E}_{1,\e} + Q_\ve (w_{1,\ve }),
\end{eqnarray*}
where $\mathcal{E}_{1, \e}$ is a sum of
functions of the form
$$
\ve \mu_0 \left( \ve \mu_0 + \ve \partial_a \mu_0 + \ve \partial^2_a
\mu_0 \right) a (z) b (\xi )
$$
and $a(\ve z)$ is a smooth function uniformly bounded, together
with its derivatives,  as $ \e \to 0$, while the function $b$ is
such that
$$
\sup_{\xi } (1+|\xi|^{N-2} ) |b (\xi ) |< \infty.
$$
The term $Q_\ve (w_{1,\e} )$ is quadratic in $w_{1,\e}$, in fact it
is explicitly given by
$$
-(w_0 + w_{1,\e} )^p + w_0^p + p w_0^{p-1} w_{1,\e}.
$$
Observe now that the term of order $0$ (in the power expansion in
$\e$) vanishes because of the equation satisfied by  $w_0$. In
order to make the coefficient of $\e$ vanish, $w_{1,\e}$ must
satisfy the following equation
\begin{equation}\label{eq:eqw1}
 \left\{
    \begin{array}{ll}
    - \D w_{1,\e} - p w_0^{p-1} w_{1,\e}+\e\,\mu_0^2 w_{1,\e} =\e g_{1,\e} (\ve z , \xi )& \hbox{ in } \R^{N}_+, \\
    \frac{\pa w_{1,\e}}{\pa \xi_N} = 0, & \hbox{ on } \{ \xi_N =
    0\},
    \end{array}
  \right.
\end{equation}
where
\begin{equation}
\label{g1e} g_{1,\e} (\ve z , \xi ) =  - \mu_0(y)\bigg[H_{\alpha
\alpha} \partial_{\xi_N} w_0 - 2 \xi_N H_{ij} \partial^2_{ij}
    w_0+\mu_0(y)\,w_0\bigg].
    \end{equation}
Using Lemma \ref{l:freed}, we see that equation \eqref{eq:eqw1} is
solvable if the right-hand side is $L^2$-orthogonal to the functions $Z_j$, for $j=0, \ldots , N-1$. These conditions,
for $j = 1, \dots, N-1$ are clearly satisfied since both
$\partial_{N} w_0$ and $\partial^2_{ij} w_0$ are even in $\bar\xi$,
while the $Z_i$'s are odd in $\bar\xi$ for every $i$. It remains to
compute the $L^2$ product of the right-hand side against $Z_0$. We
claim that
\begin{equation}
 \int_{\R^N_+}\left(H_{\alpha\alpha}\del_Nw_0-2H_{ij}\xi_N\,\del^2_{ij} w_0\right)Z_0=
\mathfrak{A}_0  \,H_{\alpha\alpha} - \mathfrak{A}_1 H_{ii}
\label{condw0}
\end{equation}
where $\mathfrak{A}_0$ and $\mathfrak{A}_1$ are the constants
defined by
\begin{equation}
\label{cccon1} \mathfrak{A}_0= {1\over 2} \int_{\R^N_+} \xi_N
|\nabla w_0 |^2 -{N-2 \over 2N} \int_{\R^N_+} \xi_N w_0^{{2N\over
N-2}},
\end{equation}
\begin{equation}
\label{cccon2} \mathfrak{A}_1= \int_{\R^N_+} \xi_N |\partial_1
w_0|^2  >0.
\end{equation}
Furthermore, we have
\begin{equation}
\label{condw1} \int_{\R^N_+} w_0 Z_0 = -\int_{\R^N_+} w_0^2.
\end{equation}
Indeed, since $( \pa_\mu w_\mu )_{|\mu=1} = -Z_0$, we have that
$$
\int_{\R^N_+} w_0 Z_0 =-\int_{\R^N_+} w_0 ( \pa_\mu w_\mu
)_{|\mu=1}= -{1\over 2} \pa_\mu \left( \int_{\R^N_+} w_\mu^2
\right)_{|\mu=1}=-{1\over 2} \pa_\mu \left( \mu^2\,\int_{\R^N_+}
w_0^2 \right)_{|\mu=1}.
$$

\noindent We postpone the proof of \equ{condw0}. We turn now to the
solvability in $w_{1,\e}$. Assuming the  quantity on the right hand
side of \eqref{condw0} is negative, we define
\begin{equation}\label{mu}
\mu_0 (y):=  \frac{\left[\mathfrak{A}_0   \,H_{\a\a} -
\mathfrak{A}_1 \, H_{ii} \right]}{\int_{\R^N_+}w_0^2 }.
\end{equation}
With this choice for $\mu_0$, the integral of  the right hand side
in \eqref{eq:eqw1} against $Z_{0}$ vanishes on $K$ and this implies
the existence of $w_{1,\e}$, thanks to Lemma \ref{l:freed}.
Moreover, it is straightforward to check that
$$
\|  g_{1,\e} \|_{\ve , N-2 , \sigma } \leq C
$$
for some $\sigma \in (0,1)$. Lemma \ref{l:freed} thus gives that
\begin{equation}
\label{ew1} \| D^2_\xi w_{1,\e} \|_{\ve , N-2 , \sigma } + \| D_\xi
w_{1,\e} \|_{\ve , N-3 , \sigma } +\|w_{1,\e} \|_{\ve, N-4 ,
\sigma}\le C \ve
\end{equation}
and that there exists a positive constant $\beta$ (depending only on
$\Omega, K$ and $N$) such that for any integer $\ell$ there holds
\begin{equation}\label{eq:estw1}
    \|\nabla^{(\ell)}_{z} w_{1,\e}(z,\cdot)\|_{\ve,N-2,\sigma} \leq \beta C_l
    \ve \qquad \quad z \in K_\ve
\end{equation}
where $C_l$ depends only on $l$, $p$, $K$ and $\Omega$.

\medskip
\noindent {\textbf{Proof of
}\equ{condw0}--\equ{cccon1}--\equ{cccon2}}. \ \

We first compute $\int_{\R^N_+} w_0 \partial_N w_0$. To do so, for
any $\mu>0$ we denote by $w_\mu$ the scaled function
$$
w_\mu (x) = \mu^{-{N-2 \over 2}} w_0 ( \mu^{-1}x).
$$
Since $ ( \partial_\mu w_\mu )_{|_{\mu=1}} = -Z_0$,  a direct
differentiation and integration by parts gives that
$$
\int_{\R^N_+} Z_0 \partial_N w_0 = \partial_\mu \left[ {1\over 2}
\int_{\R^N_+} \xi_N |\nabla w_\mu |^2 -{N-2 \over 2N} \xi_N
w_\mu^{{2N\over N-2}} \right]_{|_{\mu = 1}}.
 $$
Now changing variables $\xi\mapsto \mu \xi$ an direct computation
gives
$$
 {1\over 2} \int_{\R^N_+} \xi_N |\nabla w_\mu |^2 -{N-2 \over 2N} \int_{\R^N_+} \xi_N w_\mu^{{2N\over N-2}} = \mu \left[
 {1\over 2} \int_{\R^N_+} \xi_N |\nabla w_0 |^2 -{N-2 \over 2N} \int_{\R^N_+} \xi_N w_0^{{2N\over N-2}}
\right].
 $$
 from which \equ{cccon1} follows. Next, we compute $-  2\int_{\R^N_+} \xi_N \pa_{ij}^2 w_0 Z_0$. By
symmetry we have that $ 2\int_{\R^N_+} \xi_N \pa_{ij}^2 w_0 Z_0 = 0
$ if $i\not= j$. Assume then that $i=j$ is fixed and integrations by
parts, a direct differentiation yields
$$
-2\int_{\R^N_+} \xi_N \pa_{ii}^2 w_0 Z_0 = 2\int_{\R^N_+} \xi_N
\pa_{i} w_0 \pa_i Z_0 = -\partial_\mu \left[ \int_{\R^N_+} \xi_N
|\partial_1 w_\mu |^2 \right]_{|_{\mu=1}}
$$
and also
$$
\int_{\R^N_+} \xi_N |\partial_1 w_\mu |^2 = \mu \int_{\R^N_+} \xi_N
|\partial_1 w_0 |^2.
$$
The above facts give the validity of \equ{condw0}. Now we claim that
$\mathfrak{A}_0=2\mathfrak{A}_1$.  Indeed
\begin{eqnarray*}
\int_{\R^N_+} \xi_N|\nabla w_0 |^2&=& \int_{\R^N_+} \xi_N|\pa_N w_0|^2 +\int_{\R^N_+} \xi_N|\pa_i w_0|^2\\
  &=&\int_{\R^N_+} \xi_N|\pa_N w_0|^2 +(N-1)\int_{\R^N_+} \xi_N|\pa_1
w_0|^2\\
&=&(N+1)\int_{\R^N_+} \xi_N|\pa_1 w_0|^2=(N+1)\mathfrak{A}_1.
\end{eqnarray*}
Here we used the fact that
$$
\int_{\R^N_+} \xi_N|\pa_N w_0|^2=2\int_{\R^N_+} \xi_N|\pa_1 w_0|^2
$$
since
\begin{eqnarray*}
 \int_{\R^N_+} \xi_N|\pa_N w_0|^2&=& \alpha_N^2\int_{\R^N_+} \xi_N^2 \frac{\xi_N}{(1+|\xi|^2)^N} \\
   &=& -\alpha_N^2\frac{1}{2(N-1)}\int_{\R^N_+} \xi_N^2\pa_{N}\bigg( \frac{1}{(1+|\xi|^2)^{N-1}}
\bigg)\\
&=&\alpha_N^2\frac{1}{(N-1)}\int_{\R^N_+}
\frac{\xi_N}{(1+|\xi|^2)^{N-1}}
\end{eqnarray*}
and
\begin{eqnarray*}
 \int_{\R^N_+} \xi_N|\pa_1 w_0|^2&=& \alpha_N^2\int_{\R^N_+} \xi_N \xi_1\frac{\xi_1}{(1+|\xi|^2)^N} \\
   &=& -\alpha_N^2\frac{1}{2(N-1)}\int_{\R^N_+} \xi_N \xi_1\pa_{1}\bigg( \frac{1}{(1+|\xi|^2)^{N-1}}
\bigg)\\
&=&\alpha_N^2\frac{1}{2(N-1)}\int_{\R^N_+}
\frac{\xi_N}{(1+|\xi|^2)^{N-1}}.
\end{eqnarray*}

On the other hand
\begin{eqnarray*}
 \int_{\R^N_+} \xi_N w_0^{{2N\over
N-2}} &=& \frac12\int_{\R^N_+} \pa_N(\xi_N^2)  w_0^{{2N\over
N-2}} \\
   &=& -\frac12\int_{\R^N_+} {2N\over
N-2} \xi_N^2  w_0^{{N+2\over N-2}}\pa_N w_0\\
&=& -\frac{N}{N-2}\int_{\R^N_+}\xi_N^2 \pa_N w_0
\underbrace{w_0^{{N+2\over
N-2}}}_{-\Delta w_0}\\
&=&\frac{N}{N-2}\int_{\R^N_+}\xi_N^2 \pa_N w_0 \bigg( \pa^2_{NN}w_0
+\pa_{ii} w_0\bigg).
\end{eqnarray*}
Now we use the fact that
$$
\int_{\R^N_+}\xi_N^2 \pa_N w_0 \pa^2_{NN}w_0=-2\int_{\R^N_+}\xi_N
\pa_N w_0 \pa_{N}w_0-\int_{\R^N_+}\xi_N^2 \pa^2_{NN}w_0\pa_N w_0
$$
which implies that
$$
\int_{\R^N_+}\xi_N^2 \pa_N w_0 \pa^2_{NN}w_0=-\int_{\R^N_+}\xi_N
\pa_N w_0 \pa_{N}w_0=-2\mathfrak{A}_1.
$$
Now integrating first in $\xi_1$ and then in $\xi_N$ yields
\begin{eqnarray*}
 I:= \int_{\R^N_+}\xi_N^2 \pa_N w_0 \pa^2_{11}w_0&=& -\int_{\R^N_+}\xi_N^2
\pa^2_{N1}w_0\pa_1 w_0 \\
  &=&\int_{\R^N_+}\pa_1 w_0 \pa_N(\xi_N^2 \pa_{1}w_0)\\
&=&\int_{\R^N_+}\pa_1 w_0 \,\bigg(2\xi_N
\pa_{1}w_0+\xi_N^2\pa^2_{1N}w_0\bigg)\\
&=& 2\mathfrak{A}_1-I.
\end{eqnarray*}
This implies that
$$
I=\mathfrak{A}_1 \quad \hbox{and }\quad \int_{\R^N_+}\xi_N^2 \pa_N
w_0 \pa^2_{ii}w_0=(N-1)\int_{\R^N_+}\xi_N^2 \pa_N w_0
\pa^2_{11}w_0=(N-1)\mathfrak{A}_1.
$$
We deduce that
\begin{eqnarray*}
 \int_{\R^N_+} \xi_N w_0^{{2N\over
N-2}} &=&\frac{N}{N-2} \bigg(-2 \mathfrak{A}_1 +(N-1)\mathfrak{A}_1
\bigg)=\frac{N(N-3)}{N-2}\mathfrak{A}_1.
\end{eqnarray*}
Hence
\begin{eqnarray*}
\mathfrak{A}_0&=& {1\over 2} \int_{\R^N_+} \xi_N |\nabla w_0 |^2
-{N-2\over 2N} \int_{\R^N_+} \xi_N w_0^{{2N\over N-2}} \\
&=& \frac{N+1}{2}\mathfrak{A}_1-{(N-2)\over
2N}\;\frac{N(N-3)}{N-2}\mathfrak{A}_1=2\mathfrak{A}_1.
\end{eqnarray*}
This proves the claim. In particular Equation \eqref{mu} can be
written as
\begin{equation}\label{mu0}
\mu_0 (y):=
\mathfrak{A}_1 \, \frac{\left[ 2\,H_{\a\a} -
 H_{ii} \right]}{\int_{\R^N_+}w_0^2 }=
\mathfrak{A}_1 \, \frac{\left[ 2\,H_{aa}
+ H_{ii} \right]}{\int_{\R^N_+}w_0^2 }.
\end{equation}

\

\

\noindent $\bullet$ {\bf  Construction of $w_{2,\e}$ :} We take
$I=2$, $\mu_\e= \mu_0+\mu_{1,\e}$, $\Phi_\e =  \Phi_{1, \e}$ and
$W_{2,\e}(z, \xi)= w_0\left( \xi \right) +
  w_{1,\e}\left(z, \xi \right)
  +   w_{2,\e}\left(z, \xi \right)$, where $\mu_0$ and $w_{1,\ve}$ have already been constructed  in the previous step.
  Computing $S(W_{2,\e})$ (see \equ{Sep}) we get
  \begin{equation}\label{edith1}
  - \Delta w_{2,\e} + \e \,\mu_0^2 w_{2,\e} - p w_0^{p-1} w_{2,\e} = \ve g_{2,\ve} +
   \mathcal{E}_{2,\e} + Q_\ve (w_{2, \e}).
\end{equation}
In \equ{edith1} the function $g_{2,\e}$ is given by
\begin{eqnarray}\label{w2epsilon}
g_{2,\ve}&=&  \mu_{1,\e}(y)\bigg[-H_{\alpha\alpha}
\partial_{\xi_N} w_0 + 2 \xi_N H_{ij} \partial^2_{ij}
    w_0-2\mu_0w_0\bigg] \nonumber \\
    &+& \mu_0(y)\bigg[-H_{\alpha\alpha} \partial_{\xi_N} w_{1,\e} +
2 \xi_N H_{ij} \partial^2_{ij}
    w_{1,\e}\bigg] + \ve \mathfrak{G}_{2,\e}({ \xi}, z,
w_0,  \mu_0) \nonumber\\
&-&  \e \mu_0 \,\D_K \Phi^j_{1, \e}\, \pa_jw_0-{\e\over
3}\,\mu_0\,R_{mijl} (\xi_m  {\Phi }^l_{1, \e}+\xi_l  {\Phi }^m_{1,
\e} )\,\del^2_{ij}w_0
\\
&+&\e  \frac23\,\mu_0\, R_{mssj}\,\Phi^m_{1, \e}\,\pa_jw_0+ \e \mu_0
\bigg( (\tilde g^\epsilon)^{ab}\,R_{maaj}- \G_a^c(E_m) \G_c^a(E_j)  \bigg) \,\Phi^m_{1 ,
\e}\,\pa_jw_0. \nonumber
\end{eqnarray}
In \equ{w2epsilon} $\mathfrak{G}_{2,\e}({ \xi}, z, w_0,  w_{1,\e},
\mu_0)$ is the sum of functions of the form
$$
{\mathcal Q} (\mu_0 , \partial_a \mu_0 , \partial^2_a \mu_0 ) a(\ve
z) b(\xi )
$$
where ${\mathcal Q}$ denotes a quadratic function of its arguments,
$a(\ve z)$ is a smooth function uniformly bounded, together with its
derivatives, in $\e$ as $ \e \to 0$, while the function $b$ is such
that
$$
\sup_{\xi } (1+|\xi|^{N-2} ) |b (\xi ) |< \infty.
$$
In \equ{edith1} the term $ \mathcal{E}_{2,\e}$ can be described as
the sum of functions of the form
$$
\left( \ve {\mathcal L} (\mu_1 , \Phi_1 ) + {\mathcal Q} (\mu_1 ,
\Phi_1 ) \right) a(\ve z) b(\xi )
$$
where $(\mu_1 , \Phi_1 ) = (\mu_{1,\e} , \partial_a \mu_{1,\e} ,
\partial^2_a \mu_{1,\e}, \Phi_{1,\e} , \partial_a \Phi_{1,\e} ,
\partial^2_a \Phi_{1,\e} ) $, ${\mathcal L}$ denotes a linear
function of its arguments, ${\mathcal Q}$ denotes a quadratic
function of its arguments, $a(\ve z)$ is a smooth uniformly bounded
function, together with its derivatives, in $\e$ as $ \e \to 0$,
while the function $b$ is such that
$$
\sup_{\xi } (1+|\xi|^{N-2} ) |b (\xi ) |< \infty.
$$
Finally the term $Q_\ve (w_{2,\e} )$ is a sum of quadratic terms in
$w_{2,\e}$ like
$$
-(w_0 + w_{1,\e} + w_{2,\e} )^p + (w_0 + w_{1,\e} )^p  +p (w_0 +
w_{1,\e} )^{p-1}  w_{2,\e}
$$
and linear terms in $w_{2,\e}$ multiplied by a term of order
$\ve$, like
$$
p \left( (w_0 + w_{1,\ve} )^{p-1} - w_0^{p-1} \right) w_{2,\ve}.
$$
We will choose $w_{2,\e}$ to satisfy the following equation
\begin{equation}\label{eq:eqw2}
 \left\{
    \begin{array}{ll}
    - \D w_{2,\e} - p w_0^{p-1} w_{2,\e}+\e\,\mu_0^2 w_{2,\e} =\e\,g_{2,\ve}, & \hbox{ on } \R^N_+ \\
    \frac{\pa w_{2,\e}}{\pa \xi_N} = 0, & \hbox{ on } \{ \xi_N = 0\}.
    \end{array}
  \right.
\end{equation}
Again by Lemma \ref{l:freed}, the above equation is solvable if $g_{2,\ve}$ is $L^2$-orthogonal to $Z_j$, $j=0,
1,\cdots,N-1$. These orthogonality conditions will define the
parameters $\mu_{1,\e}$ and the normal section $\Phi_{1,\e}$.

\noindent $\circ$ {\bf Projection onto $Z_0$ and choice of
$\mu_{1,\e}$:} Recalling that by definition of $\mu_0$
 one has
 \be
 \int_{\R^{N-1} \times \R^+} \bigg[ -
H_{\alpha\alpha} \partial_{\xi_N} w_0 + 2 \xi_N H_{ij} \partial^2_{ij}
    w_0-\mu_0(y)\,w_0\bigg] \, Z_0 \, d\xi =0
 \label{poto}\ee
 and using the fact that
$w_0$ is an even function in $\bar \xi$, we have
\begin{eqnarray*}
 \int_{\R^N_+} g_{2,\ve}Z_0&=&  \mu_0 \int_{\R^N_+} \bigg[-H_{\alpha\alpha} \partial_{\xi_N} w_{1,\e} +2 \xi_N H_{ij} \partial^2_{ij}
    w_{1,\e}\bigg]Z_0d\xi \\
&-&\mu_0\,\mu_{1,\e} \int_{\R^N_+} w_0\,Z_0 d\xi+ \ve \int_{\R^N_+}
\mathfrak{G}_{2,\e}({ \xi}, z, w_0,  \mu_0)Z_0 d\xi.
\end{eqnarray*}
We observe that the term that factors like $\mu_{1,\ve} $  in the definition of $g_{2,\ve}$ in \equ{w2epsilon}  goes away after the integration thanks to relation \equ{poto}. Here and later  $\mathfrak{G}_{i,\e}$ designates a quantity that may change from line to line and which is uniformly bounded in $\e$ depending smoothly on its arguments.

\medskip \noindent Then,
we define $\mu_{1,\e}$  to make the above quantity  zero. The above
relation defines $\mu_{1,\e}$ as a smooth function of $y $ in
$K$. From estimates \equ{ew1} for $w_{1,\e}$ we get that
\begin{equation}
\label{emu1e} \| \mu_{1,\e}\|_{L^\infty (K)} +\| \partial_a
\mu_{1,\e}\|_{L^\infty (K)} +\|\partial^2_a \mu_{1,\e}\|_{L^\infty
(K)} \leq C \ve.
\end{equation}

\noindent $\circ$ {\bf Projection onto $Z_l$ and choice of
$\Phi_{1,\ve}$:} Multiplying $g_{2,\ve}$ with $\pa_l w_0$,
integrating over $\R^{N}_+$ and using the fact $w_0$ is even in the
variable $\ov\xi$, one obtains
\begin{eqnarray}\label{projZl}
(\e \mu_0)^{-1}  \int_{\R^{N}_+}g_{2,\ve}\,\pa_l w_0 &=& - \,\D_K \Phi^j_{1,\e}\,\int_{\R^{N}_+} \pa_jw_0 \pa_lw_0+ \mu_o^{-1}\,\int_{\R^{N}_+}\mathfrak{G}_{2,\e}({ \xi}, z, w_0,  \mu_0,w_{1,\e} )\,\pa_l w_0
\nonumber\\[2mm]
  &-&{1\over 3}\,\,R_{mijs}\int_{\R^{N}_+} (\xi_m  {\Phi }^s_{1,\e}+\xi_s  {\Phi }^m_{1,\e} )\,\del^2_{ij}w_0  \pa_lw_0\\[2mm]
 &+& \,\left[\frac23\, R_{mssj}\,\Phi^m_{1,\e}+  \bigg( (\tilde g^\e)^{ab}\, R_{mabj}- \G_a^c(E_m) \G_c^a(E_j)  \bigg)
\,\Phi^m_{1,\e}\right]\, \int_{\R^{N}_+}\pa_jw_0\,\pa_lw_0.
\nonumber
\end{eqnarray}

First of all, observe that  by oddness in $\ov\xi$ we have that
$$
\int_{\R^{N}_+} \pa_jw_0 \pa_lw_0=\d_{lj}\,C_0\qquad \hbox{with
}\qquad C_0:=\int_{\R^{N}_+} |\pa_1w_0|^2.
$$
On the other hand the integral $\int_{\R^{N}_+} \xi_m
\,\del^2_{ij}w_0  \pa_lw_0$ is non-zero only if, either $i = j$ and
$m = l$, or $i = l$ and $j = m$, or $i = m$ and $j = l$. In the
latter case we have $R_{mijs}=0$ (by the antisymmetry of the
curvature tensor in the first two indices). Therefore, the first
term of the second line of the above formula becomes simply
\begin{eqnarray*}
 {1\over 3}\,\,R_{mijs}\int_{\R^{N}_+}\xi_m  {\Phi }^s_{1,\e}\,\del^2_{ij}w_0  \pa_lw_0&=&
  \frac 13  R_{liis} {\Phi }^s_{1,\e} \int_{\R^{N}_+} \xi_l \partial_{l}
  w_0 \partial^2_{ii} w_0 d \xi \\
  &+& \frac 13   R_{jijs} {\Phi }^s_{1,\e}
  \int_{\R^{N}_+} \xi_j \partial_{i}
  w_0 \partial^2_{ij} w_0 d \xi.
\end{eqnarray*}
Observe that, integrating by parts, when $l \neq i$ (otherwise
$R_{liis}=0$) there holds
$$
  \int_{\R^{N}_+} \xi_l \partial_{l} w_0 \partial^2_{ii}
  w_0 d \xi = - \int_{\R^{N}_+} \xi_l \partial_{i}
  w_0 \partial^2_{li} w_0 d \xi.
$$
Hence, still by the antisymmetry of the curvature tensor we are left
with
$$
  - \frac {2}{3}  \, R_{ijjs}{\Phi }^s_{1,\e}
  \int_{\R^{N}_+} \xi_j \partial_{i}
  w_0 \partial^2_{ij} w_0 d \xi.
$$
The last integral can be computed with a further integration by
parts and is equal to $- \frac 12 C_0$, so we get
$$
    \frac {1}{3}  \, C_0\,\, R_{ijjs} {\Phi }^s_{1,\e}.
$$
In a similar way (permuting the indices $s$ and $m$ in the above
argument), one obtains
$$
 {1\over 3}\,\,R_{sijm}\int_{\R^{N}_+}\xi_s  {\Phi }^m_{1,\e}\,\del^2_{ij}w_0  \pa_lw_0=
    \frac {1}{3}  \, C_0\,\sum_{i} R_{ijjm} {\Phi }^m_{1,\e}.
$$
Collecting the above computations, it holds
\begin{eqnarray*}
-{1\over 3}\,\,R_{mijs}\int_{\R^{N}_+} (\xi_m  {\Phi }^s_{1,\e}
+\xi_s  {\Phi }^m_{1,\e} )\,\del^2_{ij}w_0  \pa_lw_0+
 \frac23\,\, R_{mssj}\,\Phi^m_{1,\e} \int_{\R^{N}_+}\pa_jw_0\,\pa_lw_0=0.
\end{eqnarray*}
Hence Formula \eqref{projZl} becomes simply
\begin{eqnarray*}
 \int_{\R^{N}_+}g_{2,\e}\,\pa_l w_0 &=& -\e \mu_0 \, C_0\,\D_K \,\Phi^l_{1,\e}
 +\e \mu_0\,C_0\,\bigg( (\tilde g^\epsilon)^{ab}R_{maal}- \G_a^c(E_m) \G_c^a(E_l)  \bigg)
\,\Phi^m_{1,\e}\\
&+&\e \int_{\R^{N}_+}\mathfrak{G}_{2,\e}\,\pa_l w_0.
\end{eqnarray*}
We then conclude that $g_{2,\e}(z,\xi,  w_0,  \dots, w_{1,\e})$, the
right-hand side of \eqref{eq:eqw2}, is $L^2$-orthogonal to $Z_l$
($l=1,\cdots,N-1$) if and only if $\Phi_{1,\e}$ satisfies an
equation of the form
\begin{equation}\label{phi1def}
  \D_K \,\Phi^l_{1,\e}-\bigg( (\tilde g^\e)^{ab}\, R_{mabj}- \G_a^c(E_m) \G_c^a(E_l)  \bigg)
\,\Phi^m_{1,\e} = G_{2,\e}(\ve z),
\end{equation}
for some expression $G_{2,\e}$ smooth on its argument. Observe that
the operator acting on $\Phi_{1,\ve}$ in the left hand side is
nothing but the Jacobi operator, see \equ{Jacobinodeg}, which is invertible by the
non-degeneracy condition on $K$. This implies the solvability of the
above equation in $\Phi_{1,\e}$.

Furthermore, equation \equ{phi1def} defines $\Phi_{1, \e}$ as a
smooth function on $K$, of order $\ve$, more precisely we have
\begin{equation}
\label{ePhi1e} \| \Phi_{1,\e}\|_{L^\infty (K)} +\| \partial_a
\Phi_{1,\e}\|_{L^\infty (K)} +\|\partial^2_a \Phi_{1,\e}\|_{L^\infty
(K)} \leq C .
\end{equation}

By our choice of $\mu_{1,\e}$ and $\Phi_{1,\ve}$ we have solvability
of equation \eqref{eq:eqw2} in $w_{2,\e}$. Moreover, it is
straightforward to check that
$$
|\e g_{2,\e} (\e z, \xi ) | \leq C \ve | \partial_{\xi_N} w_{1, \e}|
\leq C {\ve^{3\over2} \over (1+|\xi|)^{N-2}}.
$$
Furthermore, for a given $\sigma \in (0,1)$ we have
$$
\| \e g_{2,\e} \|_{\e, N-2, \sigma} \leq C \ve^{3\over 2}.
$$
Lemma \ref{l:freed} gives then that
\begin{equation}
\label{ew2} \| D^2_\xi w_{2,\e} \|_{\ve , N-2 , \sigma } + \| D_\xi
w_{2,\e} \|_{\ve , N-3 , \sigma } +\|w_{2,\e} \|_{\ve, N-4 ,
\sigma}\le C \ve^{3\over 2}
\end{equation}
and that there exists a positive constant $\beta$ (depending only on
$\Omega, K$ and $n$) such that for any integer $\ell$ there holds
\begin{equation}\label{eq:estw2}
    \|\nabla^{(\ell)}_{z} w_{2,\e}(z,\cdot)\|_{\ve,N-2,\sigma} \leq \beta C_l
    \ve^{3\over 2} \qquad \quad z \in K_\ve
\end{equation}
where $C_l$ depends only on $l$, $p$, $K$ and $\Omega$.

\medskip

\noindent $\bullet$ {\bf Expansion at an arbitrary order:} We take
now an arbitrary integer $I$. Let
\begin{equation}\label{eqmu}
 \mu_\ve:= \mu_0+\mu_{1,\e}+\cdots +\mu_{I-1,\e} + \mu_{I, \ve},
\end{equation}
\begin{equation}\label{eqPhi}
\Phi=\Phi_{1,\e}+\cdots +\Phi_{I-1,\e} + \Phi_{I, \ve}
\end{equation}
and
\begin{equation}
\label{eqWWW} W_{I+1, \ve} = w_0 (\xi ) + w_{1, \ve} (z, \xi) +
\ldots + w_{I, \ve} (z, \xi ) + w_{I+1, \ve} (z, \xi )
\end{equation}
where $\mu_0,  \mu_{1,\e} , \cdots ,  \mu_{I-1,\e}$, $\Phi_{1,\e},
\cdots , \Phi_{I-1,\e}$ and $w_{1, \ve} $, .. , $w_{I, \ve}$ have
already been constructed following an iterative scheme, as described
in the previous steps of the construction.

In particular one has, for any $i=1, \ldots , I-1$
\begin{equation}
\label{emuie} \| \mu_{i,\e}\|_{L^\infty (K)} +\| \partial_a
\mu_{i,\e}\|_{L^\infty (K)} +\|\partial^2_a \mu_{i,\e}\|_{L^\infty
(K)} \leq C \ve^{1+{i-1 \over 2}}
\end{equation}
\begin{equation}
\label{ePhiie} \| \Phi_{i,\e}\|_{L^\infty (K)} +\| \partial_a
\Phi_{i,\e}\|_{L^\infty (K)} +\|\partial^2_a \Phi_{i,\e}\|_{L^\infty
(K)} \leq C \ve^{{i-1 \over 2}}
\end{equation}
and, for now $i=0, \ldots , I-1$,
\begin{equation}
\label{ewi} \| D^2_\xi w_{i+1,\e} \|_{\ve , N-2 , \sigma } + \|
D_\xi w_{i+1,\e} \|_{\ve , N-3 , \sigma } +\|w_{i+1,\e} \|_{\ve, N-4
, \sigma}\le C \ve^{1+{i \over 2}}
\end{equation}
and, for any integer $\ell$
\begin{equation}\label{eq:estwi}
    \|\nabla^{(\ell)}_{z} w_{i+1,\e}(z,\cdot)\|_{\ve,N-2,\sigma} \leq \beta C_l \ve^{1+{i \over 2}},
     \qquad \quad z \in K_\ve.
\end{equation}
The new triplet $(\mu_{I, \ve} , \Phi_{I, \ve} , w_{I+1 , \ve} )$
will be found reasoning as in the construction of the triplet
$(\mu_{1, \ve} , \Phi_{1, \ve} , w_{2 , \ve} )$. Computing $S
(W_{I+1 , \ve })$ (see \equ{Sep}) we get
  \begin{equation}\label{edithI}
  - \Delta w_{I+1,\e} + \e\,\mu_0^2 w_{I+1,\e} -p w_0^{p-1} w_{I+1,\e} = \ve g_{I+1,\ve} +
   \mathcal{E}_{I+1,\e} +Q_\e (w_{I+1, \e} ).
\end{equation}
In \equ{edith1} the function $g_{I+1,\e}$ is given by
\begin{eqnarray}\label{wIepsilon}
g_{I+1,\ve}&=&  \mu_{I,\e}(y)\bigg[-H_{\alpha\alpha}
\partial_{\xi_N} w_0 + 2 \xi_N H_{ij} \partial^2_{ij}
    w_0-2\mu_0w_0\bigg] \nonumber \\
    &+& \mu_0(y)\bigg[-H_{\alpha\alpha} \partial_{\xi_N} w_{I,\e}   +
2 \xi_N H_{ij} \partial^2_{ij}
    w_{I,\e}\bigg] \nonumber \\
    &+&  \ve  \mathfrak{G}_{I+1,\e}({ \xi}, z,
w_0, .., e_{I,\e},  \mu_0, ..., \mu_{I-1,\e}, \Phi_{1,\e}, ... \Phi_{I-1,\e} ) \nonumber\\
&-&  \e \mu_0 \,\D_K \Phi^j_{I, \e}\, \pa_jw_0-{\e \over
3}\,\mu_0\,R_{mijl} (\xi_m  {\Phi }^l_{I, \e}+\xi_l  {\Phi }^m_{I,
\e} )\,\del^2_{ij}w_0
\\
&+& \frac23\,\e \mu_0\, R_{mssj}\,\Phi^m_{I, \e}\,\pa_jw_0+  \mu_0 \e
\bigg( (\tilde g^\e)^{ab}\,R_{mabj}- \G_a^c(E_m) \G_c^a(E_j)  \bigg) \,\Phi^m_{I ,
\e}\,\pa_jw_0. \nonumber
\end{eqnarray}
In \equ{wIepsilon} $\mathfrak{G}_{I+1,\e}({ \xi}, z, )$ is a smooth
function with
\begin{equation}
\label{eedith} \|  \mathfrak{G}_{I+1,\e} \|_{\ve , N-2 , \sigma}
\leq C \ve^{1+ {I \over 2}}.
\end{equation}
In \equ{edithI} the term $ \mathcal{E}_{I+1,\e}$ can be described as
the sum of functions of the form
$$
\left( \ve {\mathcal L} (\mu_{I} , \Phi_{I} ) + {\mathcal Q}
(\mu_{I} , \Phi_{I} ) \right) a(\ve z) b(\xi )
$$
where $(\mu_{I} , \Phi_{I} ) = (\mu_{I,\e} , \partial_a \mu_{I,\e} ,
\partial^2_a \mu_{I,\e}, \Phi_{I,\e} , \partial_a \Phi_{I,\e} ,
\partial^2_a \Phi_{I,\e} ) $, ${\mathcal L}$ denotes a linear
function of its arguments, ${\mathcal Q}$ denotes a quadratic
function of its arguments, $a(\ve z)$ is a smooth function uniformly
bounded, together with its derivatives, in $\e$ as $ \e \to 0$,
while the function $b$ is such that
$$
\sup_{\xi } (1+|\xi|^{N-2} ) |b (\xi ) |< \infty.
$$
Finally the term $Q_\ve (w_{I+1,\e} )$ is a sum of quadratic terms
in $w_{I+1,\e}$ like
$$
(w_0 + w_{1,\e} + \ldots +w_{I+1,\e} )^p - (w_0 + w_{1,\e} +\ldots +
w_{I+1, \e})^p  -p (w_0 + w_{1,\e} +\ldots + w_{I, \e} )^{p-1}
w_{I+1,\e}
$$
and linear terms in $w_{I+1,\e}$ multiplied by a term of order
$\ve^2$, like
$$
p \left( (w_0 + w_{1,\ve} )^{p-1} - w_0^{p-1} \right) w_{I+1,\ve}.
$$

We define $w_{I+1,\e}$ to satisfy the following equation
\begin{equation}\label{eq:eqwI}
 \left\{
    \begin{array}{ll}
    - \D w_{I+1,\e} - p w_0^{p-1} w_{I+1,\e}+\e\,\mu_0^2 w_{I+1,\e} =\e\,g_{I+1,\ve}, & \hbox{ on } \R^N_+ \\
    \frac{\pa w_{I+1,\e}}{\pa \xi_N} = 0, & \hbox{ on } \{ \xi_N = 0\}.
    \end{array}
  \right.
\end{equation}

Again by Lemma \ref{l:freed}, the above equation is solvable if  $g_{I+1,\ve}$ is $L^2$-orthogonal to $Z_j$, $j=0,
1,\cdots,N-1$. These orthogonality conditions will define the
parameters $\mu_{I,\e}$ and the normal section $\Phi_{I,\e}$.

\noindent $\circ$ {\bf Projection onto $Z_0$ and choice of
$\mu_{I,\e}$:} Thanks to the definition of  $\mu_0$
 one has
\begin{eqnarray*}
 \int_{\R^N_+} g_{I+1,\ve}Z_0&=&  \mu_0 \int_{\R^N_+} \bigg[-H_{\alpha\alpha} \partial_{\xi_N} w_{I,\e} + 2 \xi_N H_{ij} \partial^2_{ij}
    w_{I,\e}\bigg]Z_0d\xi \\
&-&\mu_0\,\mu_{I,\e} \int_{\R^N_+} w_0\,Z_0 d\xi+ \e \int_{\R^N_+}
\mathfrak{G}_{I+1,\e}({ \xi}, z)Z_0 d\xi.
\end{eqnarray*}
We define $\mu_{I,\e}$  to make the above quantity  zero. The above
relation defines $\mu_{I,\e}$ as a smooth function of $\e z $ in
$K$. From estimates \equ{ewi} for $w_{I,\e}$ we get that
\begin{equation}
\label{emu1e} \| \mu_{I,\e}\|_{L^\infty (K)} +\| \partial_a
\mu_{I,\e}\|_{L^\infty (K)} +\|\partial^2_a \mu_{I,\e}\|_{L^\infty
(K)} \leq C \ve^{ 1+ {I-1 \over 2}}.
\end{equation}

\noindent $\triangleright$ {\bf Projection onto $Z_l$ and choice of
$\Phi_{I,\ve}$:} Multiplying $g_{I+1,\ve}$ with $\pa_l w_0$,
integrating over $\R^{N}_+$ and arguing as in the construction of
$\Phi_{1, \ve}$, we get
\begin{eqnarray*}
 \int_{\R^{N}_+}g_{I+1,\e}\,\pa_l w_0 &=& -\e \mu_0 \, \D_K \,\Phi^l_{I,\e}
 +\e \mu_0 \,\bigg( (\tilde g^\e)^{ab}\,R_{mabl}- \G_a^c(E_m) \G_c^a(E_l)  \bigg)
\,\Phi^m_{I,\e}\\
&+& \ve \int_{\R^{N}_+}\mathfrak{G}_{I+1,\e}\,\pa_l
w_0.
\end{eqnarray*}
We then conclude that $g_{2,\e}(z,\xi,  w_0,  \dots, w_{I,\e})$, the
right-hand side of \eqref{eq:eqwI}, is $L^2$-orthogonal to $Z_l$
($l=1,\cdots,N-1$) if and only if $\Phi_{I,\e}$ satisfies an
equation of the form
\begin{equation}\label{phi1def}
  \D_K \,\Phi^l_{I,\e}-\bigg( (\tilde g^\e)^{ab}\,R_{mabl}- \G_a^c(E_m) \G_c^a(E_l)  \bigg)
\,\Phi^m_{I,\e} =G_{I+1,\e}(\ve z),
\end{equation}
where $G_{I+1,\e}$ is a smooth function on $K$. Using again the
non-degeneracy condition on $K$ we have solvability of the above
equation in $\Phi_{I,\e}$. Furthermore,  taking into account
\equ{eedith}, we get
\begin{equation}
\label{ePhiIe} \| \Phi_{I,\e}\|_{L^\infty (K)} +\| \partial_a
\Phi_{I,\e}\|_{L^\infty (K)} +\|\partial^2_a \Phi_{I,\e}\|_{L^\infty
(K)} \leq C \ve^{1+{I-1\over 2}}.
\end{equation}

By our choice of $\mu_{I+1,\e}$ and $\Phi_{I+1,\ve}$ we have
solvability of equation \eqref{eq:eqwI} in $w_{I+1,\e}$. Moreover,
it is straightforward to check that
$$
|\e g_{I+1,\e} (\e z, \xi ) | \leq C {\ve^{1+{I\over 2}} \over
(1+|\xi|)^{N-2}} .
$$
Furthermore, for a given $\sigma \in (0,1)$ we have
$$
\| \e g_{I+1,\e} \|_{\e, N-2, \sigma} \leq C \ve^{1+{I\over 2}}.
$$
Lemma \ref{l:freed} gives then that
\begin{equation}
\label{ewI} \| D^2_\xi w_{I+1,\e} \|_{\ve , N-2 , \sigma } + \|
D_\xi w_{I+1,\e} \|_{\ve , N-3 , \sigma } +\|w_{I+1,\e} \|_{\ve, N-4
, \sigma}\le C \ve^{1+{I\over 2}}
\end{equation}
and that there exists a positive constant $\beta$ (depending only on
$\Omega, K$ and $n$) such that for any integer $\ell$ there holds
\begin{equation}\label{eq:estw2}
    \|\nabla^{(\ell)}_{z} w_{I+1,\e}(z,\cdot)\|_{\ve,N-2,\sigma} \leq \beta C_l
    \ve^{1+{I\over 2}} \qquad \quad z \in K_\ve.
\end{equation}

\medskip
This concludes our construction and have the validity of Lemma \ref{Construction}.

\medskip

\setcounter{equation}{0}
\section{A global approximation and Expansion of a quadratic functional}\label{s:linear}
\smallskip

Let $\mu_\e (y)  $, $\Phi_\e (y)$ and $W_{I+1 , \e}$ be the functions whose existence and
properties have been established in Lemma \ref{Construction}. We
define locally around $K_\e := {K\over \e} \subset \partial
\Omega_\e$ in $\Omega_\e$ the function
$$
V_\e (z, X):= \, \mu_{\e}^{-{N-2 \over 2}} (\e z) \, W_{I+1 , \e}
\left( z, \, \mu_\e^{-1} (\e z) (\bar X - \, \Phi_\e (\e z) ),
 \, \mu_\e^{-1} (\e z) X_N \right) \times
 $$
 \begin{equation} \label{Vdef}
 \chi_\e (|(\bar X - \, \Phi_\e (\e z), X_N )|)
\end{equation}
where $z \in K_\e$. In \equ{Vdef} the function $\chi_\e$
is a smooth cut-off function with
\begin{equation}
\label{magaly}
 \chi_\e (r) =
\left\{
 \begin{array}{ll}
     1, & \hbox{for} \quad r \in [0,2 \e^{-\gamma} ] \\[3mm]
    0, & \hbox{for} \quad  r \in [3 \e^{-\gamma} , 4\e^{-\gamma} ],
  \end{array}
\right.
\end{equation}
and
$$
|\chi_\e^{(l)} (r) | \leq C_l \e^{l \gamma}, \quad \foral l\geq 1,
$$
for some $\gamma \in ({1\over 2} , 1)$ to be fixed later.

The function $V_\e$ is well defined in a small neighborhood of
$K_\e$ inside $\Omega_\e$. We will look at a solution to \equ{eq:pe}
of the form
$$
v_\e = V_\e +\phi.
$$
This translates into the fact that $\phi$ has to satisfy the non
linear problem
\begin{equation}\label{nonlinearproblem}
  \begin{cases}
    -\Delta \phi + \ve \phi- p V_\e^{p-1} \phi =S_\e (V_\e) + N_\e (\phi) & \text{ in } \O_\ve, \\
    \frac{\del \phi}{\pa \nu} = 0 & \text{ on } \partial \O_\ve,
  \end{cases}
\end{equation}
where
\begin{equation}
\label{eomegaeps} S_\e (V_\e )= \Delta V_\e - \e V_\e +
V_\e^p
\end{equation}
and
\begin{equation}
\label{Nomegaeps} N_\e (\phi) =  (V_\e + \phi )^p - V_\e^p - p
V_\e^{p-1} \phi.
\end{equation}
Define
$$
L_\e (\phi) = -\Delta \phi + \ve \phi - p V_\e^{p-1} \phi.
$$
Our strategy consists in solving the Non-Linear Problem
\equ{nonlinearproblem} using a fixed point argument based on the
contraction Mapping Principle. To do so, we need to establish some
invertibility properties of the linear problem
$$
L_\e (\phi) = f \quad {\mbox {in}} \quad \Omega_\ve , \quad
{\partial \phi \over \partial \nu} =
 0 \quad {\mbox {on}} \quad \partial \Omega_\e,
$$
with $f\in L^2 (\Omega_\ve )$. We do this  in two steps. First
we  study  the above problem in a strip close to the scaled
manifold $K_\e= {K\over \e}$ in $\partial \Omega_\e$. Then we
establish a complete theory for the problem in the whole domain
$\Omega_\e$: this is done in Section \ref{proof}.

Let $\gamma \in ({1\over 2} , 1)$ be the number fixed before in
\equ{magaly} and consider
\begin{equation}
\label{omegaepsilongamma} \Omega_{\e, \gamma} := \{ x \in \Omega_\e
\, : \, {\mbox {dist}} (x,K_\e ) <2 \e^{-\gamma} \}.
\end{equation}
We are first interested in solving the following problem: given $f
\in L^2 (\Omega_{\e , \gamma})$
\begin{equation} \label{lineare}
  \begin{cases}
    -\Delta \phi + \ve \phi- p V_\e^{p-1} \phi =f & \text{ in } \O_{\ve , \gamma}, \\
    \frac{\del \phi}{\pa \nu} = 0 & \text{ on } \partial \O_\ve \bigcap \bar \Omega_{\e, \gamma},
    \\ \phi = 0 & \text{ in } \partial \O_{\ve, \gamma} \setminus \partial
\Omega_\e.
  \end{cases}
\end{equation}

Observe that in the region we are considering the function $V_\ve$
is nothing but $V_\ve= {\mathcal T}_{\, \mu_\e , \, \Phi_\e}
(W_{I+1 , \e})$, where $W_{I+1, \e}$ is the function whose existence
and properties are proven in Lemma \ref{Construction}. For the
argument in this part of our proof it is enough to take $I=3$, and
for simplicity of notation we will denote by $\hat w$ the function
$W_{I+1 , \e}$ with $I=3$. Referring to \equ{vIe} we have
\begin{equation}
\label{hatw} \hat w (z, \xi) = w_0 (\xi) + \sum_{i=1}^4 w_{i,\e} (z,
\xi)
\end{equation}
where $w_0$ is defined by \equ{w00} and
\begin{equation}
\label{hatw1} \| D^2_\xi w_{i+1,\e} \|_{\ve , N-2 , \sigma } + \|
D_\xi w_{i+1,\e} \|_{\ve , N-3 , \sigma } +\|w_{i+1,\e} \|_{\ve, N-4
, \sigma}\le C \ve^{1+{i \over 2}}
\end{equation}
and, for any integer $\ell$
$$
    \|\nabla^{(\ell)}_{z} w_{i+1,\e}(z,\cdot)\|_{\ve,N-2,\sigma} \leq \beta C_l \ve^{1+{i \over 2}}
     \qquad \quad z \in K_\ve
$$
for any $i=0, 1, 2, 3$.

We will establish a solvability theory for Problem \equ{lineare} in Section \ref{eigenvalues}. For the moment, we devote the rest of this Section to expand
the quadratic functional associated to \equ{lineare}.

Define
\begin{equation}
\label{defH1e} H^1_\e = \{ u \in H^1 (\O_{\e , \gamma} ) \, : \,
u(x) = 0 \quad {\mbox {for}} \quad x \in \partial \O_{\ve, \gamma}
\setminus \partial \Omega_\e \}
\end{equation}
and the quadratic functional
given by
\begin{equation}
\label{functional1} E(\phi ) = {1\over 2} \int_{\O_{\e,\gamma}}
(|\nabla \phi |^2 + \e \phi^2 - p V_\e^{p-1} \phi^2 )
\end{equation}
for functions $\phi \in H^1_\e $.

Let
$(z,X)\in \R^{k+N}$ be the local coordinates along $K_\e$ introduced
in \equ{eq:feee}, with abuse of notation we will denote
\begin{equation}
\label{bb}
  \phi (\U_\e (z,X))= \phi(z, X ).
\end{equation}

Since the original variable $(z, X)\in \R^{k+N}$ (see \equ{eq:feee})
are only local coordinates along $K_\e$ we let the variable $(z, X)$
vary in the set $\CC_\e$ defined by
\begin{equation}\label{ddomain}
\CC_\e = \{ (z,\bar X, X_N) \ / \ \e z\in \ K,\quad  0< X_N <
\e^{-\gamma} , \quad  |\bar X|  <  \e^{-\gamma} \}.
\end{equation}
We write $\CC_\e=\frac1\e K\times\hat \CC_\e$ where
\begin{equation}\label{dddomain}
\hat \CC_\e = \{ (\bar X, X_N) \ / \ 0 < X_N < \e^{-\gamma} , \quad
|\bar X|  <  \e^{-\gamma} \}.
\end{equation}
Observe  that $\hat \CC_\e$ approaches, as $\ve \to 0$, the half
space $\R_+^N$.

In these new local coordinates,  the energy density associated to
the energy $E$ in \equ{functional1} is given by
\begin{equation} \label{e}
\left[\frac12\left(|\nabla_{g^\e} \phi|^2+\e \phi^2 -p V_\e^{p-1}
\phi^2 \right)  \right] \sqrt{\det(g^\e)},
\end{equation}
where $\nabla_{g^\e}$ denotes the gradient in the new variables and
where $g^\e$ is the flat metric in $\R^{N+k}$ in the coordinates
$(z, X)$.

Having the expansion of the metric coefficients, see Lemma \ref{metricepsilon}, we are in a
position to expand the energy  \equ{functional1} in the new variable
$(z,X)$. Precisely we have the following

\begin{lemma}\label{le5.1} Let $(y,x)\in \R^{k+N}$ be the local coordinates along the submanifold $K$ introduced
in \equ{eq:fe}, let $(z,X)$ be the expanded variables introduced in
\equ{bb}. Assume that $(z,X)$ vary in $\CC_\e$ (see \equ{ddomain}).
Then, the energy functional \equ{functional1} in the new variables
\equ{bb} is given by
\begin{eqnarray}\label{energydensity}
E  ( \phi) & = &\int_{K_\e \times \hat \CC_\e} \left(\frac12 (
|\nabla_X \phi|^2 +\e \phi^2 - p V_\e^{p-1} \phi^2 )  \right)
\sqrt{\det(g^\e)} \, dz \, dX\nonumber
\\
&& +\int_{K_\e \times \hat \CC_\e} \frac12\,\Xi_{ij} (\e z ,
X)\,\partial_{i}\phi\partial_{j}\phi\,\sqrt{\det(g^\e)} \, dz \, dX\\
 &&
 +\frac12 \int_{K_\e \times \hat \CC_\e} |\nabla_{K_\e}\phi|^2\,\sqrt{\det(g^\e)} \, dz \, dX+
 \int_{K_\e \times \hat \CC_\e} B(\phi,\phi)\,\sqrt{\det(g^\e)} \, dz \, dX. \nonumber
\end{eqnarray}
In the above expression, we have
\begin{equation}\label{eqXiij}
\Xi_{ij} (\e z , X) =2 \e H_{ij} X_N-\frac{\e^2}3\,R_{islj}X_l X_s,
\end{equation}
we denoted by $B(\phi,\phi)$ a quadratic term in $\phi$ that can be
expressed in the following form
\begin{eqnarray}
 B(\phi,\phi)&=& O \left( \e^2 X_N^2 + \e^3 | \bar X  |^3 +  \e^3 X_N |\bar X|^2
 + \ve^3 X_N^2 | \bar X | \right)\del_{i} \phi \del_{j} \phi \nonumber\\
\nonumber\\
&&+{\e^2}\,|\nabla_{K_\e}\phi|^2\,
  O(\e |X|) +\pa_j \phi \pa_{\bar a}\phi \left(\mathcal{O}(\ve | \bar X
|+\ve^2 X_N^2)\right). \label{defBB}
\end{eqnarray}
and we used the Einstein convention over repeated indices.
Furthermore we use the notation $\partial_a =
\partial_{y_a} $ and $\pa_{\bar a} = \pa_{z_a}$.

\end{lemma}

\begin{proof}[Proof of Lemma \ref{le5.1}]
Our aim is to expand
$$
\int \left[\frac12\left(|\nabla_{g^\e} \phi|^2+\e \phi^2 -p
V_\e^{p-1} \phi^2 \right)  \right] \sqrt{\det(g^\e)}.
$$
For simplicity we will omit the $\e$ in the notation of $g^\e$.
Recalling our notation about repeated indices, we write $|\nabla_{g^\e}
\phi |^2$ as
\begin{eqnarray*}
|\nabla_{g^\e}\phi |^2
&=&(g^\e)^{NN}\pa_N \phi\pa_N \phi+(g^\e)^{ab}\pa_a \phi\pa_b \phi+(g^\e)^{ij}\pa_i
\phi\pa_j \phi+2(g^\e)^{aj}\pa_a \phi\pa_j \phi,
\end{eqnarray*}
where $(g^\e)^{\a \b}$ represent the coefficients of the inverse of the
metric $g^\e = (g^\e_{\alpha \beta})$. Using the expansion of the
metric in Lemma \ref{metricepsilon}, we see that
\begin{eqnarray*}
|\nabla_{g^\e}\phi |^2&=& |\partial_N \phi|^2 + |\partial_i \phi|^2 + (
2\e H_{ij} X_N  -{\e^2 \over 3} R_{islj} X_s X_l ) \pa_i \phi \pa_j
\phi +\pa_{\bar a} \phi \pa_{\bar a}\phi (1+ \e |X|)
\\
&+& O\left(\e^2 X_N^2 + \e^3 |X|^3 \right) \pa_i \phi \pa_j \phi+ O(\e X_N + \e^2
O(|X|^2 )) \pa_{\bar a }\phi \, \pa_i \phi.
\end{eqnarray*}
This together with the expansion of $\sqrt{{\mbox {det}} g}$ given
in Lemma \ref{metricepsilon}, prove Lemma \ref{le5.1}.
\end{proof}

Given a function $\phi \in H^1_\e$ (see \equ{defH1e}), we write it
as
\begin{equation}\label{decomp}
\phi= \left[{\delta \over \, \mu_\e} {\mathcal T}_{\, \mu_\e ,
\, \Phi_\e} ( Z_0) + \sum_{j=1}^{N-1} {d^j \over \, \mu_\e}
{\mathcal T}_{\, \mu_\e , \, \Phi_\e} ( Z_j)
 + {e \over \, \mu_\e}   {\mathcal T}_{\, \mu_\e , \, \Phi_\e} (Z) \right] \bar \chi_\ve   + \phi^\bot
\end{equation}
where the expression ${\mathcal T}_{\, \mu_\e , \, \Phi_\e} (
v)$ is defined in \eqref{defTT}, the functions $Z_0$ and $Z_j$ are
already defined in \equ{lezetas} and where $Z$ is the eigenfunction, with $\int_{\R^N} Z^2 =1$,
corresponding to the unique positive eigenvalue $\la_0$ in $L^2
(\R^N)$ of the problem
\begin{equation}
\label{lambda0} \Delta_{\R^N} \phi + p w_0^{p-1} \phi = \la_0 \phi
\quad {\mbox {in}} \quad \R^N.
\end{equation}
It is worth mentioning that $Z (\xi )$ is even and it has
exponential decay of order $O(e^{-\sqrt{\la_0} |\xi|} )$ at
infinity. The function $\bar \chi_\e$ is a smooth cut off function
defined by
\begin{equation}
\label{chibar} \bar \chi_\e (X) = \hat \chi_\e \left(  \left |\left ({\bar X -
\Phi_\e \over \, \mu_\e}, {X_N \over \, \mu_\e}\right ) \right | \right ),
\end{equation}
with $\hat \chi(r) = 1$ for $r \in (0,{3\over 2} \e^{-\gamma} )$,
and $\chi(r)=0$ for $r>2\e^{-\gamma}$. Finally, in \equ{decomp} we
have that $\delta = \delta (\e z)$, $d^j = d^j (\e z)$ and $e= e(\e
z)$ are function defined in $K$ such that $\foral z\in K_\e$
\begin{equation}
\label{orth1} \int_{\hat{\mathcal{C}}_\e} \phi^\bot {\mathcal
T}_{\, \mu_\e , \, \Phi_\e} (Z_0) \bar \chi_\e d X =
\int_{\hat{\mathcal{C}}_\e} \phi^\bot {\mathcal T}_{\, \mu_\e ,
\, \Phi_\e} (Z_j) \bar \chi_\e = \int_{\hat{\mathcal{C}}_\e}
\phi^\bot {\mathcal T}_{\, \mu_\e , \, \Phi_\e} (Z) \bar
\chi_\e=0
\end{equation}
We will denote by $(H_\e^1)^\bot$ the subspace of the functions in
$H_\e^1$ that satisfy the orthogonality conditions (\ref{orth1}).

A direct computation shows that
$$
\delta (\e z) = {\int \phi  {\mathcal T}_{\, \mu_\e ,
\Phi_\e} (Z_0) \over \, \mu_\e \int Z_0^2} (1+ O(\e^2)) + O(\e^2)
(\sum_j d^j (\e z) + e (\e z)), \quad
$$
$$
d^j (\e z) = {\int \phi  {\mathcal T}_{\, \mu_\e , \, \Phi_\e}
(Z_j) \over \, \mu_\e \int Z_j^2} (1+ O(\e^2))  + O(\e^2) (\delta
(\e z) + \sum_{i\not= j}  d^i (\e z) + e (\e z)),
$$
and
$$
e(\e z) = {\int \phi  {\mathcal T}_{\, \mu_\e , \, \Phi_\e} (Z)
\over \, \mu_\e \int Z^2} (1+ O(\e^2))  + O(\e^2) (\delta (\e z)
+\sum_j d^j (\e z) ).
$$
Observe that, since $\phi \in H_\e^1$, one easily get that the
functions $\delta $, $d^j$ and $e$ belong to the Hilbert space
\begin{equation}\label{H1K}
{\mathcal H}^1 (K) = \{ \zeta \in {\mathcal L}^2 (K) \, : \,
\partial_a \zeta \in {\mathcal L}^2 (K),\quad a=1,\cdots,k \}.
\end{equation}

Thanks to the above decomposition \equ{decomp}, we have the validity
of the following expansion for $E(\phi)$.

\begin{theorem} \label{teo4.1} Let $\gamma= 1-\sigma$, for some $\sigma >0$ and small.
Assume we write $\phi \in H^1_\e$ as in \equ{decomp} and let $d=
(d^1 , \ldots , d^{N-1})$. Then, there exists $\e_0>0$ such that,
for all  $0<\e <\e_0$, the following expansion holds true
\begin{equation}
\label{EEE} E(\phi ) = E(\phi^\bot ) + \e^{-k} \left[ P_\e (\delta )
+Q_\e (d ) + R_\e (e)  \right]  +{\mathcal M} (\phi^\bot , \delta ,
d , e).
\end{equation}
In \equ{EEE}
\begin{equation}\label{q0e}
P_\e (\delta ) = P (\delta ) + P_1 (\delta)
\end{equation}
with
\begin{equation}
\label{q0} P (\delta )  =\left[ {A_\e \over 2} \int_K \e^2
|\partial_a ( \delta (1 + o(\e^2)  \beta_1^\e (y) ) )|^2 +\e {B\over
2} \int_K \delta^2 \right]
\end{equation}
with $A_\e$ a real number such that $\lim_{\e \to 0 } A_\e = A:=
\int_{\R^N_+} Z_0^2 $, $B= -\int_{\R^N_+} w_0 Z_0 >0$ and
$\beta_1^\e$ is an explicit smooth function defined on $K$ which is
uniformly bounded as $\e \to 0$; furthermore,  $P_1 (\delta ) $ is a
small compact perturbation in ${\mathcal H}^1 (K) $ whose shape is a
sum of quadratic functional in $\delta$ of the form
$$
 \ve^{2} \int_K b(y) |\delta |^2
$$
where $b(y)$ denotes a generic explicit function, smooth and
uniformly bounded, as $\e \to 0$, in $K$. In \equ{EEE},
\begin{equation}
\label{qje} Q_\e (d ) = Q(d ) + Q_1 (d)
\end{equation}
with
\begin{equation} \label{Qj}
Q (d)=  {\e^2 \over 2} C_\e \left( \int_K |\partial_a ( d (1+ o(\e^2
) \beta_2^\e (y)  ) )|^2 + \int_K ((\tilde g^\e)^{ab}R_{mabl} - \Gamma_a^c (E_m)
\Gamma_c^a (E_l) ) d^m d^l \right)
\end{equation}
where $C_\e$ is a real number such that $\lim_{\e \to 0} C_\e = C:=
\int_{\R^N_+} Z_1^2$, $\beta_2^\e$ is an explicit smooth function
defined on $K$ which is uniformly bounded as $\e \to 0$ and the
terms $R_{maal} $ and $\Gamma_a^c (E_m)$ are smooth functions on $K$
defined respectively in \equ{ctens} and \equ{eq:Gab}. Furthermore,
$Q_1 (d)$ is a small compact perturbation  in ${\mathcal H}^1 (K) $
whose shape is a sum of quadratic functional in $d$ of the form
$$
\e^{3} \int_K b(y) d^i d^j
$$
where again   $b(y)$ is a generic explicit function, smooth and
uniformly bounded, as $\e \to 0$, in $K$. In \equ{EEE},
\begin{equation}
\label{qe0} R_\e (e) = R(e) + R_1 (e)
\end{equation}
\begin{equation}
\label{Q} R(e) = \e^{-k} \left[ {D_\e \over 2} \left( \ve^2 \int_K
|\partial_a ( e (1+ e^{-{\lambda_0 \over 2} \e^{-\gamma}} \beta_3^\e
(y)  ) )|^2 -\la_0  \int_K  e^2 \right)\right]
\end{equation}
with $D_\e$ a real number so that  $\lim_{\e \to 0 } D_\e = D:=
\int_{\R^N_+} Z^2 $, $\beta_3^\e$ an explicit smooth function in
$K$,
 which is uniformly bounded as $\e \to 0$,  and $\la_0$ the positive number defined in \equ{lambda0}. Furthermore, $R_1$
 is a small compact perturbation  in ${\mathcal H}^1 (K) $ whose shape is a sum of quadratic functional in $e$ of the form
$$
\e^{3} \int_K b(y) e^2
$$
where again   $b(y)$ is a generic explicit function, smooth and
uniformly bounded, as $\e \to 0$, in $K$. Finally in \equ{EEE}
$$
{\mathcal M} : (H^1_\e )^\bot \times ({\mathcal H}^1 (K) )^{N+1} \to
\R
$$
is a continuous and differentiable functional with respect to the
natural topologies, homogeneous of degree $2$
$$
{\mathcal M} (t \phi^\bot , t \delta , t d , t e ) = t^2 {\mathcal
M} ( \phi^\bot , \delta ,  d ,  e ) \quad \foral t.
$$
The derivative of ${\mathcal M}$ with respect to each one of its
variable is given by a small multiple of a linear operator in
$(\phi^\bot , \delta , d , e)$ and it satisfies
$$
\| D_{(\phi^\bot , \delta , d)} {\mathcal M}(\phi_1^\bot , \delta_1
, d_1 , e_1) - D_{(\phi^\bot , \delta , d)} {\mathcal M}(\phi_2^\bot ,
\delta_2 , d_2 , e_2) \| \leq C \e^{\gamma (N-3)} \times
$$
\begin{equation}
\label{lips} \left[ \| \phi_1^\bot - \phi_2^\bot \| + \e^{-k} \|
\delta_1 - \delta_2 \|_{{\mathcal H}^1 (K)} + \e^{-k} \| d_1 - d_2
\|_{({\mathcal H}^1 (K))^{N-1}} + \e^{-k} \| e_1 - e_2 \|_{{\mathcal
H}^1 (K)}\right].
\end{equation}
 Furthermore, there exists a constant $C>0$ such that
\begin{equation}
\label{estMM} \left| {\mathcal M } ( \phi^\bot , \delta ,  d ,  e )
\right| \leq C \ve^{2} \left[ \| \phi^\bot \|^2 + \ve^{-2k} \left( \|
\delta \|_{{\mathcal H}^1 (K)}^2  + \| d \|_{{\mathcal H}^1 (K)}^2 +
\| e\|_{{\mathcal H}^1 (K)}^2   \right) \right].
\end{equation}
\end{theorem}

We postpone the proof of Theorem \ref{teo4.1} to Appendix \ref{AA2}.

\setcounter{equation}{0}
\section{Solving a linear problem close to the manifold $K$}\label{eigenvalues}

In this Section we study the problem of finding $\phi \in H^1_\e$
(see \equ{defH1e}) solution to the linear problem \equ{lineare} for
a given $f \in L^2 (\Omega_{\e , \gamma } )$, (see
\equ{omegaepsilongamma}), and we establish a-priori bounds for the
solution.
This section is devoted to prove this. The result is contained in the following

\smallskip

\begin{theorem} \label{teouffa}
There exist a constant $C>0$ and a sequence $\e_l = \e \to 0$ such
that, for any $f \in L^2 (\Omega_{\e , \gamma} )$ there exists a
solution $\phi \in H^1_\e $ to Problem \equ{lineare} such that
\begin{equation}
\label{uffa1} \| \phi \|_{H^1_\e} \leq C \e^{- \max (2, k)} \| f
\|_{L^2 (\Omega_{\e , \gamma} )}.
\end{equation}

\end{theorem}

The entire section is devoted to prove Theorem \ref{teouffa}.

\medskip

Given $\phi\in H^1_\e (\Omega_{\e , \gamma })$. As in \equ{decomp},
we have the following decomposition of $\phi$
$$
\phi= \left[{\delta \over \, \mu_\e} {\mathcal T}_{\, \mu_\e ,
\, \Phi_\e} ( Z_0)
 + \sum_{j=1}^{N-1} {d^j \over \, \mu_\e} {\mathcal T}_{\, \mu_\e , \, \Phi_\e} ( Z_j)
 + {e \over \, \mu_\e}   {\mathcal T}_{\, \mu_\e , \, \Phi_\e} (Z) \right] \bar \chi_\ve   +
 \phi^\bot.
$$
We then define the energy functional associated to Problem
\equ{lineare}
$$
{\mathcal E}: (H^1_\e)^\bot \times ( {\mathcal H}^1 (K) )^{N+1} \to
\R
$$
by
\begin{equation}
\label{functional2} {\mathcal E} (\phi^\bot , \delta , d , e) =
E(\phi ) -{\mathcal L}_f (\phi )
\end{equation}
where $E$ is the functional in \equ{functional1} and ${\mathcal L}_f
(\phi )$ is the linear operator given by
$$
{\mathcal L}_f (\phi ) = \int_{\Omega_{\e , \gamma}} f \phi.
$$
Observe that
$$
{\mathcal L}_f (\phi ) = {\mathcal L}_f^1 (\phi^\bot ) + \e^{-k}
\left[ {\mathcal L}_f^2 (\delta ) + {\mathcal L}_f^3 (d ) +{\mathcal
L}_f^4 (e) \right]
$$
where ${\mathcal L}_f^1 :H^1_\e \to \R$, $ {\mathcal L}_f^2 ,
{\mathcal L}_f^4 \, : \, {\mathcal H}^1 (K) \to \R$ and  ${\mathcal
L}_f^3 \, : \, ( {\mathcal H}^1 (K) )^{N-1} \to \R$ with
$$
{\mathcal L}_f^1 (\phi^\bot ) = \int_{\Omega_{\e , \gamma}} f
\phi^\bot , \quad \e^{-k} {\mathcal L}_f^2 (\delta ) =
\int_{\Omega_{\e , \gamma}} f {\delta \over \, \mu_\e} {\mathcal
T}_{\, \mu_\e ,\, \Phi_\e } ( Z_0 ) \bar \chi_\e
$$
$$
\e^{-k} {\mathcal L}_f^3 (d ) = \sum_{j=1}^{N-1}\int_{\Omega_{\e ,
\gamma}} f {d^j \over \, \mu_\e} {\mathcal T}_{\, \mu_\e ,\bar
\Phi_\e } ( Z_j ) \bar \chi_\e \quad {\mbox {and}} \quad \e^{-k}
{\mathcal L}_f^4 (e ) = \int_{\Omega_{\e , \gamma}} f {e \over \bar
\mu_\e} {\mathcal T}_{\, \mu_\e ,\, \Phi_\e } ( Z ) \bar \chi_\e
.
$$
Finding a solution $\phi \in H^1_\e$ to Problem \equ{lineare}
reduces to finding a critical point $(\phi^\bot , \delta , d , e) $
for ${\mathcal E}$. This will be done in several steps.

\bigskip
\noindent \textrm{\textbf{Step 1}}. We claim that there exist
$\sigma >0$ and $\e_0$ such that for all $\e \in (0,\e_0)$ and all
$\phi^\bot \in (H^1_\e)^\bot$ then
\begin{equation}\label{punto1}
E(\phi^\bot ) \geq \sigma \| \phi^\bot \|^2_{L^2}.
\end{equation}

\medskip
Using the local change of variables \equ{eq:feee} and \equ{bb},
together with the result of Lemma \ref{le5.1}, we see that, for
sufficiently small $\e>0$
$$
E(\phi^\bot ) \geq {1\over 4} E_0 (\phi^\bot ), \quad {\mbox {with}}
\quad E_0 (\phi^\bot ) = \int_{K_\e \times \hat {\mathcal C}_\e}
\left[ |\nabla_X \phi^\bot |^2 - p V_\e^{p-1} \phi^\bot
\right]\sqrt{\det(g^\e)}
$$
for any $\phi^\bot = \phi^\bot (\e z , X),$ with $ z \in K_\e = {1
\over \e}K$. The set $\hat {\mathcal C}_\e$ is defined in
\equ{dddomain} and the function $V_\e$ is given by \equ{Vdef}. We
recall that $\hat {\mathcal C}_\e \to \R^N_+$ as $\e \to 0$.

We will establish \equ{punto1} showing that
\begin{equation}
\label{pas1} E_0 (\phi^\bot ) \geq \sigma \| \phi^\bot \|^2_{L^2}
\quad \foral \phi^\bot.
\end{equation}
To do so, we first observe that if we scale in the $z$-variable,
defining $\varphi^\bot (y,X) = \phi^\bot ({y\over \e }, X)$, the
relation \equ{pas1} becomes
\begin{equation}
\label{pas2} E_0 (\varphi^\bot ) \geq \sigma \| \varphi^\bot
\|^2_{L^2}.
\end{equation}
Thus we are led to show the validity of \equ{pas2}. We argue by
contradiction,  for any $n\in \N^*$, there exist
$\e_n \to 0 $ and $\varphi_n^\bot \in (H_{\e_n}^1)^\bot $ such that
\begin{equation}
\label{shen} E_0 (\varphi_n^\bot ) \leq \frac1n \| \varphi_n^\bot
\|^2_{L^2}.
\end{equation}
Without loss of generality we can assume that the sequence $(\|
\varphi_n^\bot\|)_n$ is bounded, as $n \to \infty$.
Hence, up to subsequences, we have that
$$ \varphi_n^\bot \rightharpoonup \varphi^\bot \quad \hbox{in}
\quad H^1 (K \times \R^N_+ )\qquad  \hbox{and } \quad \varphi_n^\bot
\to \varphi^\bot \quad \hbox{in} \quad L^2 (K \times \R^N_+ ).
$$
Furthermore, using the estimate in \equ{boh1} we get that
$$
\sup_{y \in K , X \in \R^N_+} \left| (1+ |X|)^{N-4}  \left[ V_\e
({y\over \e} , X ) - \mu_0^{-{N-2 \over 2}} (y) w_0 ({\bar X -
\Phi_0 (y) \over \mu_0 (y) } , {X_N \over \mu_0 (y)}) \right]
\right| \to 0,
$$
as $\e \to 0$, where $\mu_0$ and $\Phi_0 $ are the smooth explicit
function defined in \equ{boh} and \equ{mu}.

Passing to the limit as $n \to \infty$ in \equ{shen} and applying
dominated convergence Theorem, we get
\begin{equation}
\label{pas3} \int_{K \times \R^N_+} \left[ |\nabla_X \varphi^\bot
|^2 - p \left( \mu_0^{-{N-2 \over 2}} (y) w_0 ({\bar X - \Phi_0 (y)
\over \mu_0 (y) } ,
 {X_N \over \mu_0 (y)}) \right)^{p-1} (\varphi^\bot)^2 \right] dy dX \leq
 0.
\end{equation}
Furthermore, passing to the limit in the orthogonality condition we
get, for any $y \in K$
\begin{equation}
\label{shen1} \int_{\R^N_+} \varphi^\bot (y, X) Z_0 ({\bar X -
\Phi_0 (y) \over \mu_0 (y) } , {X_N \over \mu_0 (y)}) dX = 0 ,
\end{equation}
\begin{equation}
\label{shen2} \int_{\R^N_+} \varphi^\bot (y, X) Z_j ({\bar X -
\Phi_0 (y) \over \mu_0 (y) } , {X_N \over \mu_0 (y)}) dX = 0 , \quad
j=1, \ldots N-1
\end{equation}
and
\begin{equation}
\label{shen3} \int_{\R^N_+} \varphi^\bot (y, X) Z ({\bar X - \Phi_0
(y) \over \mu_0 (y) } , {X_N \over \mu_0 (y)}) dX = 0 .
\end{equation}
We thus get a contradiction with \equ{pas3}, since for any function
$\varphi^\bot$ satisfying the orthogonality conditions
\equ{shen1}--\equ{shen3} for any $y \in K$ one has
$$
\int_{K \times \R^N_+} \left[ |\nabla_X \varphi^\bot |^2 - p \left(
\mu_0^{-{N-2 \over 2}} (y) w_0 ({\bar X - \Phi_0 (y) \over \mu_0 (y)
} , {X_N \over \mu_0 (y)}) \right)^{p-1} (\varphi^\bot)^2 \right] dy
dX > 0
$$
(see for instance \cite{dmpa,rey}).


\bigskip
\noindent \textrm{\textbf{Step 2}}. For all $\e>0$ small, the
functional $P_\e (\delta )$ defined in \equ{q0e} is continuous and
differentiable in
 ${\mathcal H}^1 (K)$; furthermore, it is strictly convex and bounded from below since
\begin{equation} \label{step2}
P_\e (\delta ) \geq {1\over 4} \left[ {A \over 2} \e^2 \int_K
|\partial_a \delta |^2 + {B \over 2} \e \int_K \delta^2 \right] \geq
\sigma \e^2 \| \delta \|_{{\mathcal H}^1 (K)}^2
\end{equation}
for some small but fixed $\sigma >0$. A direct consequence of these
properties is that
$$
\delta \in {\mathcal H}^1 (K) \longmapsto P_\e (\delta ) - {\mathcal
L}^2_f (\delta )
$$
has a unique minimum $\delta$, and furthermore
$$
\e^{-{k\over 2}} \| \delta \|_{{\mathcal H}^1 (K) } \leq \mathcal{C}
\ve^{-2} \| f \|_{L^2 (\Omega_{\e , \gamma} )}
$$
for a given positive constant $\mathcal{C}$.

\bigskip
\noindent \textrm{\textbf{Step 3}}. For all $\e >0$ small, the
functional $Q_\e$ defined in \equ{qje} is a small perturbation in
$({\mathcal H}^1 (K) )^{N-1}$ of the quadratic form $\e^2 Q_0 (d)$,
defined by
$$
\e^2 Q_0 (d)=  {\e^2 \over 2} C \left[ \int_K |\partial_a d |^2 +
\int_K ((\tilde g_\e )^{ab}\,R_{mabl} - \Gamma_a^c (E_m) \Gamma_c^a (E_l) ) d^m d^l
\right] $$ with $C:= \int_{\R^N_+} Z_1^2$ and the terms $R_{maal} $
and $\Gamma_a^c (E_m)$ are smooth functions on $K$ defined
respectively in \equ{ctens} and \equ{eq:Gab}. Recall that the
non-degeneracy assumption on the minimal submanifold $K$ is
equivalent to the invertibility of the operator $Q_0 (d)$.

A consequence, for each $f \in L^2 (\Omega_{\e
, \gamma} )$,
$$
d \in ({\mathcal H}^1 (K))^{N-1} \longrightarrow \R , \quad d
\longmapsto Q_\e (d) - {\mathcal L}_f^3 (d)
$$
has a unique critical point $d$, which satisfies
$$
\ve^{-{k\over 2}} \| d \|_{({\mathcal H}^1 (K))^{N-1}} \leq
\widetilde{\sigma} \ve^{-2} \| f \|_{L^2 (\Omega_{\e , \gamma} )}
$$
for some proper $\widetilde{\sigma} >0$.

\bigskip
\noindent \textrm{\textbf{Step 4}}. Let $f \in L^2 (\Omega_{\e ,
\gamma})$ and assume that $e$ is a given (fixed) function in
${\mathcal H}^1 (K)$. We claim that for all $\e >0$ small enough,
the functional ${\mathcal G} : (H_\e^1)^\bot \times ({\mathcal H}^1
(K) )^N \to \R$
$$
(\phi^\bot , \delta , d) \to {\mathcal E} (\phi^\bot , \delta , d ,
e)
$$
has a critical point $(\phi^\bot , \delta , d) $. Furthermore there
exists a positive constant $C$, independent of $\e$, such that
\begin{equation} \label{punto4}
\| \phi^\bot \| + \e^{-{k\over 2}} \bigg[ \| \delta \|_{{\mathcal
H}^1 (K)} + \| d \|_{({\mathcal H}^1 (K))^{N-1}} \bigg] \leq C
\e^{-2} \bigg[ \|f \|_{L^2 (\Omega_{\e , \gamma} )} + \e^{-{k\over
2}} \ve^{2} \| e \|_{{\mathcal H^1 (K) } }\bigg].
\end{equation}

To prove the above assertion, we first consider the functional
$${\mathcal G}_0 (\phi^\bot , \delta , d) = {\mathcal G} (\phi^\bot , \delta , d , e) - {\mathcal M} (\phi^\bot , \delta , d , e)
$$
where ${\mathcal M} $ is the functional that recollects all mixed
terms, as defined in \equ{EEE}. A direct consequence of Step 1, Step
2 and Step 3 is that ${\mathcal G}_0$ has a critical point
$(\phi^\bot = \phi^\bot (f) , \delta = \delta (f) , d = d(f))$,
namely the system
$$
D_{\phi^\bot } E (\phi^\bot ) = D_{\phi^\bot} {\mathcal L}_f^1
(\phi^\bot ), \quad \e^{-{k\over 2}} D_\delta P_\e (\delta ) =
D_\delta {\mathcal L}_f^2 (\delta ), \quad \e^{-{k\over 2}} D_d Q_\e
(d) = D_d {\mathcal L}_f^3 (d)
$$
is uniquely solvable in $(H_\e^1 )^\bot \times ({\mathcal H}^1 (K)
)^N $ and furthermore
$$
\| \phi^\bot \|_{H^1_\e} +\e^{-{k\over 2}} \| \delta \|_{{\mathcal
H}^1 (K)} + \e^{-{k\over 2}} \| d \|_{({\mathcal H}^1 (K) )^{N-1}}
\leq C \e^{-2} \| f \|_{L^2 (\Omega_{\e , \gamma })}
$$
for some constant $C>0$, independent of $\e$.

If we now consider the complete functional ${\mathcal G}$, a
critical point of ${\mathcal G}$ shall satisfy the system
\begin{equation}\label{syst2}
  \begin{cases}
  D_{\phi^\bot } E(\phi^\bot ) = D_{\phi^\bot} {\mathcal L}_f^1 (\phi^\bot ) + D_{\phi^\bot } {\mathcal M} (\phi^\bot , \delta , d , e) \\
 D_\delta P_\e (\delta ) = D_\delta {\mathcal L}_f^2 (\delta ) + D_{\delta } {\mathcal M} (\phi^\bot , \delta , d , e)
    \\   D_d Q_\e (d ) = D_d {\mathcal L}_f^3 (d ) + D_{d } {\mathcal M} (\phi^\bot , \delta , d ,
    e).
  \end{cases}
\end{equation}
On the other hand, as we have already observed in Theorem
\ref{teo4.1}, we have
$$
\| D_{(\phi^\bot , \delta , d)} {\mathcal M}(\phi_1^\bot , \delta_1
, d_1 , e_1) - D_{(\phi^\bot , \delta , d)} {\mathcal M}(\phi_2^\bot ,
\delta_2 , d_2 , e_2) \| \leq C \e^{2} \times
$$
$$
\left[ \| \phi_1^\bot - \phi_2^\bot \| + \e^{-{k\over 2}} \|
\delta_1 - \delta_2 \|_{{\mathcal H}^1 (K)} + \e^{-{k\over 2}} \|
d_1 - d_2 \|_{({\mathcal H}^1 (K))^{N-1}} + \e^{-{k\over 2}} \| e_1
- e_2 \|_{{\mathcal H}^1 (K)}\right].
$$
Thus the contraction mapping Theorem guarantees the existence of a
unique solution $(\bar \phi^\bot , \bar \delta , \bar d)$ to
\equ{syst2} in the set
$$
\| \phi^\bot \|_{H_\e^1} + \e^{-{k\over 2}} \| \delta \|_{{\mathcal
H}^1 (K)} + \e^{-{k\over 2}} \| d \|_{({\mathcal H}^1 (K))^{N-1}}
\leq C \left[ \e^{-2} \| f \|_{L^2 (\Omega_{\e , \gamma} )} + \e^{2}
\e^{-{k\over 2}} \| e \|_{{\mathcal H}^1 (K)} \right].
$$
Furthermore, the solution $\bar \phi^\bot = \bar \phi^\bot (f,e) $,
$\bar \delta =\bar  \delta (f,e )$ and $\bar d= \bar d(f,e)$ depend
on $e$ in a smooth and non-local way.

\bigskip
\noindent {\rm \textbf{Step 5}}. Given $f\in L^2 (\Omega_{\e ,
\gamma} )$, we replace the critical point $(\bar \phi^\bot = \bar
\phi^\bot (f,e), \bar \delta =\bar  \delta (f,e ), \bar d= \bar
d(f,e) )$ of ${\mathcal G}$ obtained in the previous step into the
functional ${\mathcal E}(\phi^\bot , \delta , d , e)$ thus getting a
new functional depending only on $e \in {\mathcal H}^1 (K)$, that we
denote by ${\mathcal F}_\e(e)$, given by
\begin{eqnarray*}
{\mathcal F}_\e(e) &=& \e^{-k} [ R_\e (e) - {\mathcal L}_f^4 (e) ] +
E(\bar \phi^\bot (e) ) -\ve^{-k} {\mathcal L}_f^1 (\bar \phi^\bot
(e)) +\e^{-k} [ P_\e (\bar \delta (e)) - {\mathcal L}_f^2 (\bar
\delta (e) )]
\\
&+& \e^{-k} [ Q_\e (\bar d (e)) - {\mathcal L}_f^3 (\bar d (e)) ] +
{\mathcal M} (\bar \phi^\bot (e) , \bar \delta (e) , \bar d (e) ,
e).
\end{eqnarray*}
The rest of the proof is devoted to show that there exists a
sequence $\e = \e_l \to 0$ such that
\begin{equation}
\label{effe1} D_e {\mathcal F}_\e(e) = 0
\end{equation}
is solvable. Using the fact that $(\bar \phi^\bot , \bar \delta ,
\bar d )$ is a critical point for ${\mathcal G}$ (see Step 4 for the
definition), we have that
\begin{equation}
\label{effe11} D_e {\mathcal F}_\e (e) = \e^{-k} D_e [ R_\e (e) -
{\mathcal L}_f^4 (e) ]  + D_e {\mathcal M} (\bar \phi^\bot (e) ,
\bar \delta (e) , \bar d (e) , e).
\end{equation}
Define
\begin{equation}
\label{defLe} {\mathcal L}_\e := \e^{-k} D_e  R_\e (e)  + D_e
{\mathcal M} (\bar \phi^\bot (e) , \bar \delta (e) , \bar d (e) ,
e),
\end{equation}
regarded as self adjoint in ${\mathcal L}^2 (K)$. The work to solve
the equation $D_e {\mathcal F}_\e (e)=0$ consists in showing the
existence of a sequence $\e_l \to 0$ such that $0$ lies suitably far
away from the spectrum of ${\mathcal L}_{\e_l}$.

We recall now that the map
$$
(\phi^\bot , \delta , d , e ) \to D_e {\mathcal M} (\phi^\bot ,
\delta , d , e)
$$ is a linear operator in the variables $\phi^\bot , \delta , d$, while it is constant in $e$. This is contained in the result of Theorem \ref{teo4.1}. If we
furthermore take into account that the terms $\bar \phi^\bot $, $\bar \delta $ and $\bar d $ depend smoothly and in a non-local way through $e$, we conclude that, for
any $e \in {\mathcal H}^1 (K)$,
\begin{equation}
\label{effe2} D_e {\mathcal M} (\bar \phi^\bot (e) , \bar \delta (e)
, \bar d (e) , e )[e] = \e^{\gamma (N-3)} \e^{-k} \int_K \left( \e
\eta_1 (e) \partial_a e + \eta_2 (e) e \right)^2
\end{equation}
where $\eta_1$ and $\eta_2$ are  non local operators in $e$, that
are bounded, as $\e \to 0$, on bounded sets of ${\mathcal L}^2 (K)$.
Thanks to the result contained in Theorem \ref{teo4.1} and the above
observation, we conclude that the quadratic from
$$
\Upsilon_\e (e) := \e^{-k}  D_e R_\e (e) [e] + D_e {\mathcal M}
(\bar \phi^\bot (e) , \bar \delta (e) , \bar d (e) , e) [e]
$$
can be described as follows
\begin{equation}
\label{Uptilde} \tilde \Upsilon_\e (e) = \e^k \Upsilon_\e (e ) =
\Upsilon^0_\e (e)
 - {\bar \la_0 } \int_K e^2 + \e \Upsilon_\e^1 (e)
\end{equation}
where
\begin{equation}
\label{Up0} \Upsilon_\e^0 (e) = \e^2 \int_K  (1+ \e^{\gamma (N-3)}
\eta_1 (e) ) \left| \partial_a \left( e (1+ e^{-\e^{-\lambda'}}
\beta_3^\e (y)  ) \right) \right|^2.
\end{equation}
In the above expression $\bar \la_0 $ is the positive number defined
by
$$\bar \la_0 = (\int_{\R^N_+} Z_1^2 )\, \la_0,$$
$\Upsilon_e^1 (e)$ is a compact quadratic form in ${\mathcal H}^1
(K)$, $\beta_3^\e$ is a smooth and bounded (as $\e \to 0$) function
on $K$, given by \equ{Q}. Finally, $\eta_1 $ is a non local operator
in $e$, which is uniformly bounded, as $\e \to 0$
 on bounded sets of ${\mathcal L}^2 (K)$.

Thus, for any $\e >0$, the eigenvalues of
$$
{\mathcal L}_\e e = \la e , \quad e \in {\mathcal H}^1 (K)
$$
are given by a sequence $\la_j (\e)$, characterized by the
Courant-Fisher formulas
\begin{equation}\label{courantFisher}
\la_j (\e )= \sup_{dim (M)= j-1} \inf_{e \in M^\bot \setminus \{0
\}} {\tilde \Upsilon_\e (e ) \over \int_K e^2 } = \inf_{dim (M)= j }
\sup_{e \in M\setminus \{0 \}} {\tilde \Upsilon_\e (e) \over \int_K
e^2 }.
\end{equation}
The proof of Theorem \ref{teouffa} and of the inequality \equ{uffa1}
will follow then from Step 4 and formula \equ{punto4}, together with
the validity of the following

\bigskip
\begin{lemma} \label{gafa}
There exist a sequence $\e_l \to 0 $ and a constant $c>0$ such that,
for all $j$, we have
\begin{equation}
\label{autovalore} |\la_j (\e_l  ) | \geq c \e_l^k.
\end{equation}
\end{lemma}

\bigskip
The proof of this Lemma follows  closely the proof of a related result established in \cite{dkwy}, but we
reproduce it for completeness. We shall thus devote the rest of this
section to prove Lemma \ref{gafa}.

We call $\Sigma_\e (e) = {\tilde \Upsilon_\e (e) \over \int_K e^2
}$.

For notational convenience, we let $\sigma = \e^2$. We are thus
interested in studying the eigenvalue problem
\begin{equation}
\label{dep5} {\mathcal L}_\sigma \eta = \la \eta, \quad \eta \in
{\mathcal H}^1 (K).
\end{equation}
With this notation and using \eqref{Uptilde} and \eqref{Up0}
together with the fact that $\g (N-3)>2$ we have that
$$
\Sigma_\sigma (e) = {\sigma \int_K  (1+ o(\sigma ) \eta_1 (e) )
\left| \partial_a \left( e (1+ o(\sigma ) \beta_3^\sigma (y)  )
\right) \right|^2 \over \int_K e^2 } -\bar \la_0 + \sqrt{\sigma}
{\Upsilon_\sigma^1 (e) \over \int_K e^2 }
$$
where $o(\sigma ) \to 0 $ as $\sigma \to 0$.

We claim that there exists a number $\delta >0$ such that for any
$\sigma_2 >0$ and for any $j\geq 1$ such that
$$
\sigma_2 + |\lambda_j (\sigma_2 )| < \delta
$$
and any $\sigma_1 $ with ${\sigma_2 \over 2} < \sigma_1 < \sigma_2
$, we have that
\begin{equation}
\label{dep3} \lambda_j (\sigma_1 ) < \lambda_j (\sigma_2).
\end{equation}

To prove the above assertion, we start observing that, since
$\beta_3^\sigma$ is an explicit, smooth and bounded (as $\sigma \to
0$) function on $K$, as given by \equ{Q} and since $\eta_1 $ is a
non local operator in $e$, which is uniformly bounded, as $\sigma
\to 0$ on bounded sets of ${\mathcal L}^2 (K)$, we have that
\begin{equation}
\label{lim1} \lim_{\sigma \to 0} { \int_K  (1+ o(\sigma ) \eta_1 (e)
) \left| \partial_a \left( e (1+ o(\sigma ) \beta_3^\sigma (y)  )
\right) \right|^2 \over \int_K e^2 } = \frac{\int_K |\partial_a
e|^2}{\int_K e^2 }
\end{equation}
and
\begin{equation}
\label{lim2} \lim_{\sigma \to 0} \sqrt{\sigma} {\Upsilon_\sigma^1
(e) \over \int_K e^2 } = 0
\end{equation}
uniformly for any $e $.

Consider now two numbers $0<\sigma_1 < \sigma_2 $. Then for any $e $
with $\int_K e^2 =1$, we have that
\begin{eqnarray*}
\sigma_1^{-1} \Sigma_{\sigma_1} (e) - \sigma_2^{-1}
\Sigma_{\sigma_2} (e) &=& -\bar \la_0 {\sigma_2 - \sigma_1 \over
\sigma_1 \sigma_2} + \sigma_1^{-1} \Upsilon_{\sigma_1}^0 (e)
-\sigma_2^{-1} \Upsilon_{\sigma_2}^0 (e)
\\
&+& \sigma_1^{-\frac12}\, \Upsilon_{\sigma_1}^1 (e)
-\sigma_1^{-\frac12}\, \Upsilon_{\sigma_1}^1 (e).
\end{eqnarray*}
A consequence of \equ{lim1} and \equ{lim2} is that there exists a
real number $\sigma^*>0$ such that, for all $\sigma_1 <\sigma_2
<\sigma^*$
$$
\left| \sigma_1^{-1} \Upsilon_{\sigma_1}^0 (e) -\sigma_1^{-1}
\Upsilon_{\sigma_1}^0 (e) \right| \leq c (\sigma_2 - \sigma_1 )
$$
and
$$
\left| \sigma_1^{-\frac12} \,\Upsilon_{\sigma_1}^1 (e)
-\sigma_1^{-\frac12} \,\Upsilon_{\sigma_1}^1 (e)  \right| \leq c
{\sigma_2 - \sigma_1 \over \sqrt{\sigma_1 \sigma_2} (\sigma_1 +
\sigma_2 )}
$$
for some constant $c$, and uniformly for any $e$ with $\int_K e^2
=1$. Thus we have that, for some $\gamma_-$, $\gamma_+ >0$
$$
\sigma_1^{-1} \Sigma_{\sigma_1} (e) + (\sigma_2 - \sigma_1)
{\gamma_- \over 2\sigma_2^2} \leq \sigma_2^{-1} \Sigma_{\sigma_2}
(e) \leq \sigma_1^{-1} \Sigma_{\sigma_1} (e) + (\sigma_2 - \sigma_1)
{2\gamma_+ \over \sigma_1^2}
$$
for any $\sigma_1 < \sigma_2 < \sigma^*$ and any $e$ with $\int_K
e^2 =1$. Thus in particular we get that, there exists $\sigma^*$
such that for all $0<\sigma_1 <\sigma_2 <\sigma^*$ and for all
$j\geq 1$
\begin{equation}
\label{dep2} (\sigma_2 - \sigma_1 ) {\gamma_- \over 2 \sigma_2^2 }
\leq \sigma_2^{-1} \la_j (\sigma_2 ) - \sigma_1^{-1} \la_j (\sigma_1
) \leq 2 (\sigma_2 - \sigma_1 ) {\gamma_+ \over \sigma_1^2}.
\end{equation}
From \equ{dep2} it follows directly that, for all $j\geq 1$, the
function $\sigma \in (0,\sigma^*) \to \lambda_j (\sigma )$
 is continuous. If we now assume that $\sigma_1 \geq {\sigma_2 \over 2}$, formula \equ{dep2} gives
\begin{equation}
\label{dep0} \lambda_j (\sigma_1 ) \leq \lambda_j (\sigma_2 ) +
{\sigma_1 - \sigma_2  \over \sigma_2 } \left[ \lambda_j (\sigma_2 )
+ \gamma {\sigma_1 \over \sigma_2} \right]
\end{equation}
for some $\gamma>0$. This gives the proof of \equ{dep3}.

\bigskip
We will find a sequence $\sigma_l \in (2^{-(l+1)} , 2^{-l} )$ for $l
$ large as in the statement of the Lemma.

Define
$$
{\L} = \{ \sigma \in (2^{-(l+1)} , 2^{-l} ) \, : \, ker {\mathcal
L}_\sigma \not= \{ 0 \} \}.
$$
If $\sigma \in {\L}$ then for some $j$,  $\lambda_j (\sigma ) =
0$. Choosing $l$ sufficiently large, the continuity of the function
$\sigma \to \lambda_j (\sigma)$ together with \equ{dep3} imply that
$\lambda_j (2^{-(l+1)} )<0$. In other words, for all $l$
sufficienlty large
\begin{equation}
\label{dep4} {\mbox {card}} ({\L} ) \leq N (2^{-(l+1)})
\end{equation}
where $N(\sigma)$ denotes the number of negative eigenvalues of
problem \equ{dep5}. We next estimate $N(\sigma)$, for small values
of $\sigma$. To do so, let $a>0$ be a positive constant such that
$a>\bar \la_0$ and consider the operator
\begin{equation}
\label{dep6} {\mathcal L}_\sigma^+ = -\Delta_K - {a\over \sigma}.
\end{equation}
We call $\lambda_j^+ (\sigma)$ its eigenvalues. Courant-Fisher
characterization of eigenvalues gives that $\lambda_j (\sigma ) \leq
\lambda_j^+ (\sigma)$ for all $j$ and for all $\sigma $ small. Thus
$N(\sigma ) \leq N_+ (\sigma )$, where $N_+ (\sigma)$ is the number
of negative eigenvalues of \equ{dep6}.

Denote now by $\mu_j$ the eigenvalues of $-\Delta_K$ (ordered to be
non-decreasing in $j$ and counted with their multiplicity). Weyl's
asymptotic formula (see for instance \cite{chavel}) states that
$$
\mu_j = C_K j^{2\over k} + o(j^{2\over k} ) \quad {\mbox {as}} \quad
j \to \infty
$$
for some positive constant $C_K$ depending only on the dimension $k$
of $K$. Since $\lambda_j^+ = \mu_j - {a\over \sigma}$, we get
$$
N_+ (\sigma ) = C \sigma^{-{k \over 2}} + o(\sigma^{-{k\over 2}} )
\quad {\mbox {as}} \quad \sigma \to 0.
$$
This fact, together with \equ{dep4}, give ${\mbox {card}} ({\L} )
\leq C 2^{l {k\over 2}}$. Hence there exist an interval $(a_l , b_l
) \subset (2^{-(l+1)} , 2^{-l} )$ such that $a_l $, $b_l \in
{\L}_l$, and all $\sigma \in (a_l , b_l )$ with Ker ${\mathcal
L}_\sigma \not= \{ 0 \}$ so that
\begin{equation}
\label{dep7} b_l - a_l \geq {2^{-l} - 2^{-(l+1)} \over {\mbox
{card}} ({\L}_l )} \geq C 2^{-l (1+ {k\over 2})}.
\end{equation}
Let $\sigma_l = {a_l + b_l \over 2}$. We will show that this is a
sequence that verifies the statement of Lemma \ref{gafa} and the
corresponding estimate \equ{autovalore}. By contradiction, assume
that for some $j$ we have
\begin{equation}
\label{dep8} |\lambda_j (\sigma_l ) |\leq \delta \sigma_l^{k\over 2}
\end{equation}
for some arbitrary $\delta >0$ small. Assume first that $0<\la_j
(\sigma_l ) <\delta \sigma_l^{k\over 2}$. Then from \equ{dep0} we
get
$$
\la_j (a_l ) \leq \la_j (\sigma_l ) - {\sigma_l - a_l \over \sigma_l
} \left[ \la_j (\sigma_l ) + \gamma {a_l \over 2\sigma_l } \right]
$$
and using \equ{dep7}-\equ{dep8} we get to
$$
\la_j (a_l ) \leq \delta \sigma_l^{k\over 2} - C{ 2^{-l (1+{k\over
2})} \over 2\sigma_l } \left[ \la_j (\sigma_l ) +{\gamma a_l \over
2\sigma_l } \right] <0,
$$
having chosen $\delta$ small. From this is follows that $\la_j
(\sigma)$ must vanish at some $\sigma \in (a_l , b_l )$, but this is
in contradiction with the choice of the interval $(a_l , b_l )$.

The case $-\delta \sigma_l^{k\over 2} <\la_j (\sigma_l ) <0$ can be
treated in a very similar way.

This concludes the proof of Lemma \ref{gafa}.

\setcounter{equation}{0}
\section{Proof of the Result}\label{proof}

In this section, we show the existence of a solution to Problem
\equ{eq:pe} of the form
$$
v_\e = V_\e +\phi
$$
where $V_\e$ is defined in \equ{Vdef}. As already observed at the
end of Section \ref{s:aprsol}, this reduces to find a solution
$\phi$ to
\begin{equation}\label{nonno}
  \begin{cases}
    -\Delta \phi + \ve \phi- p V_\e^{p-1} \phi =S_\e (V_\e) + N_\e (\phi) & \text{ in } \O_\ve, \\
    \frac{\del \phi}{\pa \nu} = 0 & \text{ on } \partial \O_\ve,
  \end{cases}
\end{equation}
where $S_\e (V_\e )$ is defined in \equ{eomegaeps}, and
$N_\e (\phi)$ in \equ{Nomegaeps}.

Given the result  of Lemma \ref{Construction}, a first fact is that
\begin{equation}
\label{nonna} \| S_\e (V_\e )\|_{L^2 (\Omega_\e ) } \leq
C \e^{1+{I+1 \over 2}}
\end{equation}
as a direct consequence of estimate \equ{bf4}.

Define $L_\e \phi := -\Delta \phi +\e \phi - p V_\e^{p-1} \phi$. We
claim that there exist a sequence $\e_l \to 0$ and a positive
constant $C >0$, such that, for any $f \in L^2 (\Omega_{\e_l} )$,
there exists a solution $\phi \in H^1 (\Omega_{\e_l} )$ to the
equation
$$L_{\e_l }\phi = f \quad \hbox{in } \Omega_{\e_l }, \qquad {\partial
\phi \over
\partial \nu } = 0 \quad \hbox{on }\quad  \partial \Omega_{\e_l}.
$$
Furthermore,
\begin{equation}
\label{gigio} \| \phi \|_{H^1 (\Omega_{\e_l} )} \leq C\, \e_l^{-\max \{
2, k \}} \| f \|_{L^2 (\Omega_{\e_l})}.
\end{equation}

We postpone for the moment the proof of this fact and we assume its
validity. For simplicity of notations we omit the dependance of $\e$
on $l$ setting $\e_l = \e$. Thus $\phi \in H^1 (\Omega_\e )$ is a
solution to \equ{nonno} if and only if
$$
\phi = L_\e^{-1} \left( S_\e (V_\e ) + N_\e (\phi )
\right).
$$
Notice that
\begin{equation}
\label{nonno1} \| N_\e (\phi ) \|_{L^2 (\Omega_\e )} \leq C
\begin{cases}
    \| \phi \|_{H^1 (\Omega_\e )}^p & \text{ for } p\leq 2, \\
    \| \phi \|_{H^1 (\Omega_\e )}^2 & \text{ for } p>2
  \end{cases} \quad\qquad  \| \phi \|_{H^1 (\Omega_\e )} \leq 1
\end{equation}
and
\begin{eqnarray}
\label{nonno2} &&\| N_\e (\phi_1 ) - N_\e (\phi_2 ) \|_{L^2
(\Omega_\e
)}  \nonumber\\
&&\leq C \,\begin{cases}
   \left(  \| \phi_1  \|_{H^1 (\Omega_\e )}^{p-1} + \| \phi_2  \|_{H^1 (\Omega_\e )}^{p-1} \right) \| \phi_1 -\phi_2 \|_{H^1 (\Omega_\e )} & \text{ for } p\leq 2, \\
    \left(  \| \phi_1  \|_{H^1 (\Omega_\e )} + \| \phi_2  \|_{H^1 (\Omega_\e )} \right) \| \phi_1 -\phi_2 \|_{H^1 (\Omega_\e )}  & \text{ for } p>2
  \end{cases},
\end{eqnarray}
for any $\phi_1$, $\phi_2$ in $H^1 (\Omega_\e )$ with $\| \phi_1
\|_{H^1 (\Omega_\e )}$, $ \| \phi_2 \|_{H^1 (\Omega_\e )} \leq 1.$

Defining $T_\e :H^1 (\Omega_\e ) \to H^1 (\Omega_\e )$ as
$$
T_\e (\phi ) = L_\e^{-1} \left( S_\e (V_\e ) + N_\e (\phi
) \right)
$$
we will show that $T_\e$ is a contraction in some small ball in $H^1
(\Omega_\e )$. A direct consequence of \equ{nonna}, \equ{nonno1},
\equ{nonno2} and \equ{gigio}, is that
$$
\| T_\e (\phi ) \|_{H^1 (\Omega_\e )} \leq C \e^{- \max \{ 2 , k \}}
\begin{cases}
   \left( \e^{1+{I+1 \over 2}} +  \| \phi  \|_{H^1 (\Omega_\e )}^p \right) & \text{ for } p\leq 2, \\
    \left( \e^{1+ {I+1 \over 2}} + \| \phi  \|_{H^1 (\Omega_\e )}^2 \right)   & \text{ for } p>2.
  \end{cases}
$$
Now we choose integers $d$ and $I$ so that
$$
d> \begin{cases}
   {\max \{ 2 , k \} \over p-1}& \text{ for } p\leq 2, \\
    \max \{ 2 , k \}  & \text{ for } p>2
  \end{cases} \quad I > d-1 + \max \{ 2 , k \}.
  $$
  Thus one easily gets that $T_\e$ has a unique fixed point in set
  $${\mathcal B} = \{ \phi \in H^1 (\Omega_\e ) \, : \, \| \phi \|_{H^1 (\Omega_\e )} \leq \e^d \},
  $$
  as a direct application of the contraction mapping Theorem. This concludes the proof of Theorem \ref{teo1}.

 \bigskip

  We next prove the assertion previously made: we recall it. We claim that there exist a sequence $\e_l \to 0$ and a positive constant $\sigma >0$, such that, for any $f
  \in L^2 (\Omega_{\e_l} )$, there exists a solution $\phi \in H^1 (\Omega_{\e_l} )$ to
$L_{\e_l }\phi = f$ in $\Omega_{\e_l }$, with ${\partial \phi \over
\partial \nu } = 0$ on $\partial \Omega_{\e_l}$, furthermore
estimate \equ{gigio} holds true.

\bigskip
By contradiction, assume that for all $\e \to 0$ there exists a
solution $(\phi_\e , \la_\e )$, $\phi_\e \not= 0$, to
\begin{equation}
\label{uno} L_\e (\phi_\e ):= \Delta \phi_\e -\e \phi_\e +
pV_\e^{p-1} \phi_\e = \la_\e \phi_\e \quad {\mbox {in}} \quad
\Omega_\e , \quad {\partial \phi_\e \over \partial \nu} = 0 \quad
{\mbox {on}} \quad \partial \Omega_\e \end{equation} with
\begin{equation}
\label{due} |\la_\e | \e^{-\max \{ 2, k\} } \to 0 , \quad {\mbox
{as}} \quad \e \to 0.
\end{equation}
Let $\eta_\e $ be a smooth cut off function (like the one defined in
\equ{magaly}) so that $\eta_\e = 1 $ if dist $(y, K_\e )
<{\e^{-\gamma} \over 2}$ and $\eta_\e =0$ if dist $(y, K_\e ) >
\e^{-\gamma} $. In particular one has that $|\nabla \eta_\e| \leq c
\e^{\gamma} $ and $|\Delta \eta_\e |\leq c \e^{2\gamma}$, in the
whole domain.

Define $\tilde \phi_\e = \phi_\e \eta_\e$. Then $\tilde \phi_\e$
solves
\begin{equation}\label{tre}
  \begin{cases}
   L_\e (\tilde \phi_\e ) = \la_\e \tilde \phi_\e -\nabla \eta_\e \nabla \phi_\e - \Delta \eta_\e \phi_\e & \text{ in } \O_{\ve,\gamma} \\
    \frac{\del \phi_\e}{\pa \nu} = 0 & \text{ on } \partial \Omega_\e \setminus {\overline {\Omega}}_{\e , \gamma},
    \\ \phi_\e = 0 & \text{ in } \partial \Omega_\e \bigcap \partial \Omega_{\e , \gamma},
  \end{cases}
\end{equation}
where $\Omega_{\e , \gamma}$ is the set defined in
\equ{omegaepsilongamma}. We now apply Theorem \ref{teouffa}, that
guarantees the existence of a sequence $\e_l \to 0$ and a constant
$c$ such that
\begin{equation}
\label{quattro} \| \tilde \phi_{\e_l} \|_{H^1_{\e_l}} \leq c
\e_l^{-\max \{2,k\}} \left[ \la_{\e_l} \| \tilde \phi_{\e_l}
\|_{L^2} + \| \nabla \eta_{\e_l} \nabla  \phi_{\e_l} \|_{L^2} + \|
\Delta \eta_{\e_l}  \phi_{\e_l} \|_{L^2} \right].
\end{equation}
Observe now that, in the region where $\nabla \eta_{\e_l} \not= 0$
and $\Delta \eta_{\e_l} \not= 0$, the function $V_{\e_l}$ can be
uniformly bounded $|V_\e (y) |\leq c \e $, with a positive constant
$c$, fact that follows directly from \equ{Vdef} and \equ{boh1}.
Furthermore, since we are assuming \equ{due}, we see that in the
region we are considering, namely where $\nabla \eta_{\e_l} \not= 0$
and $\Delta \eta_{\e_l} \not= 0$, the function $\phi_{\e_l}$
satisfies the equation $- \Delta \phi_{\e_l} + \e_l a_{\e_l} (y)
\phi_{\e_l} = 0$, for a certain smooth function $a_{\e_l}$, which is
uniformly positive and bounded as $\e_l \to 0$. Elliptic estimates
give that, in this region, $|\phi_{\e_l} |\leq c
e^{-\e_l^{\gamma'}}$, and $|\nabla \phi_{\e_l} | \leq c
e^{-\e_l^{\gamma'}}$ for some $\gamma'>0$ and $c>0$. Inserting this
information in \equ{quattro}, it is easy to see that
$$
\| \tilde \phi_{\e_l} \|_{H^1_{\e_l}} \leq c \e_l^{-\max \{2,k\}}
\la_{\e_l} \| \tilde \phi_{\e_l} \|_{H^1_{\e_l}}  (1+ o(1))
$$
where $o(1) \to 0$ as $\e_l \to 0$. Taking into account \equ{due}
the above inequality gives a contradiction with the fact that, for
all $\e$, the function $\phi_\e$ is not identically zero. This
concludes the prove of the claim.

\setcounter{equation}{0}
\section{Appendix: Proof of Lemma \ref{scaledlaplacian}}\label{alaplacian}

The proof is simply based on a Taylor expansion of the metric
coefficients in terms of the geometric properties of $\pa \O$ and
$K$, as in Lemma \ref{metricepsilon}. Recall that the Laplace-Beltrami
operator is given by~
\[ \Delta_{
g^\e}=\frac{1}{\sqrt{\det g^\e}}\,\partial_A(\,\sqrt{\det g^\e}\,(
g^\e)^{AB}\,\partial_B\,)\,,\] where indices $A$ and $B$ run between
$1$ and $n=N+k$. We can write~
\[{\mathcal A}_{\mu_\ve , \Phi_\e }=({g^\e})^{AB}\,\partial^2_{AB}+\partial_A\,({
g^\e})^{AB}\,\partial_B+\partial_A(\,\log{\sqrt{\det
g^\e}}\,)\,({g^\e})^{AB}\,\partial_B.\] Now, if $v$ and $W$ are  defined
as in \eqref{b}, one has
$$
\mu_{\varepsilon}^{\frac
{N+2}{2}}\pa^2_{X_iX_j}v=\pa^2_{\xi_i\xi_j}W(z,\xi)=\pa^2_{ij}W,
$$
$$
\mu_{\varepsilon}^{\frac
{N+2}{2}}\pa^2_{X_N}v=\pa^2_{\xi_N}W(z,\xi)=\pa^2_{NN}W,
$$
\begin{eqnarray*}
\mu_{\varepsilon}^{\frac
{N+2}{2}}\pa^2_{z_aX_l}v&=&-\frac{N}{2}\,\pa_{\bar a}\mu_\ve\,\pa_l W
+\mu_\e \,\pa^2_{\bar al}W
-\pa_{\bar a}\mu_\ve\xi_J\,\pa^2_{lJ}W-\pa_{\bar a}\Phi^j\,\pa^2_{lj}W
\end{eqnarray*}
and
\begin{eqnarray*}
\mu_{\varepsilon}^{\frac
{N+2}{2}}\pa^2_{z_az_b}v&=&\frac{N(N-2)}{4}\,\pa_{\bar a}\mu_\ve\,\pa_{\bar b}\mu_\ve
W
-\frac{N-2}{2}\mu_\ve \left(\pa_{\bar a} \mu_\ve\pa_{\bar b}W +\pa_{\bar b}\mu_\ve\pa_{\bar a}W \right)\\
&+&N\, \pa_{\bar a} \mu_\ve\,\pa_{\bar b} \mu_\ve \xi_J\pa_J W+\frac{N}{2}\,\{\pa_{\bar a} \mu_\ve\,\pa_{\bar b} \Phi^j+\pa_{\bar b} \mu_\ve\,\pa_{\bar a}
\Phi^j\}\,\pa_j W \\
&-&\frac{N-2}{2}\mu_\ve \pa^2_{{\bar a} {\bar b}}\mu_\ve W+\mu_\ve^2 \pa^2_{\bar a\bar b}W-\mu_\ve\,\pa^2_{\bar a \bar b}\mu_\ve \xi_{J}\,\pa_JW-\mu_\ve\,\pa^2_{\bar a
\bar b}\Phi^j \,\pa_jW\\
&-&\mu_\ve\bigg(\pa_{\bar b}\mu_\ve \xi_J\,\pa^2_{J\bar
a}W+\pa_{\bar a}\mu_\ve \xi_J\,\pa^2_{J\bar b}W\bigg)-
\mu_\ve\bigg(\pa_{\bar b}\Phi^j\,\pa^2_{j\bar
a}W+\pa_{\bar a}\Phi^j\,\pa^2_{j\bar b}W\bigg)
\\&+&\pa_{\bar a}\mu_\ve\,\pa_{\bar b}\mu_\ve\, \xi_{J}\,\xi_{L}\,\pa^2_{JL}W
+\left\{\pa_{\bar a}\mu_\ve\pa_{\bar b}\Phi^l+ \pa_{\bar b}\mu_\ve\pa_{\bar a}\Phi^l  \right\}\xi_J\pa^2_{Jl}W+\pa_{\bar a}\Phi^l\,\pa_{\bar b}\Phi^j \pa^2_{jl}W\\
&:=& \mathcal{A}_{ab}
\end{eqnarray*}
where $\partial_a = \partial_{y_a} $ and $\pa_{\bar a} = \pa_{z_a}$
and the indices $J$, $L$ run between $1$ and $N$ while as before the
indices $j$, $l$ run between $1$ and $N-1$.

\

\noindent Using the above expansions of the metric coefficients, we
easily see that~
\[
\begin{array}{rllll}
&\mu_{\varepsilon}^{\frac {N+2}{2}}({g^\e})^{AB}\,\partial^2_{AB}v=\partial^2_{ii}W+\partial^2_{NN}W+\mu_\ve^2\,(\tilde g^\e)^{ab}\pa^2_{ab}W
+2\, \mu_\ve \ve \xi_N\,H_{ij}\,\partial^2_{ij}W\\[3mm]
&+\left\{3\,\ve^2 \mu_\ve^2\, \xi_N^2(H^2)_{ij}-\frac{\ve^2}3\,R_{mijl}(\mu_\ve \xi_l+\Phi^l)(\mu_\ve \xi_m+\Phi^m)\right\}\partial^2_{ij}W\\[3mm]
&+2\,\mu_\ve \ve \xi_N\, \big(H_{ja}+(\tilde g^\e)^{ac}H_{cj}\big)\left( -\frac {N-2}{2}
\pa_{\bar a} \mu_\ve\,\pa_jW +\frac{\mu_\ve}{\ve}\pa^2_{\bar aj}W-
\pa_{\bar a}\mu_\ve \xi_L\,\pa^2_{jL}W-\pa_{\bar a}\Phi^l\,\pa^2_{jl}W\right)\\[3mm]
&+\mathcal{A}_{aa}+ \mathcal{B}_1(W)
\end{array}
\]
where
\begin{eqnarray*}
\mathcal{B}_1(v)&=&O \left(\ve^2  ( \mu_\ve \bar y  + \Phi )^2 + \ve^2 \mu_\ve
\xi_N(\mu_\ve \bar \xi
+ \Phi )+\ve^2  \mu_\ve^2 \xi_N^2 \right) \times\\
&\times& \left( -\frac{N}{2}\,\pa_{\bar a}\mu_\ve\,\pa_lW
+\frac{\mu_\ve}{\ve}\,\pa^2_{\bar al}W
-\pa_{\bar a}\mu_\ve\xi_J\pa^2_{lJ}W-\pa_{\bar a}\Phi^j\pa^2_{lj}W \right)\\
&+&O \left( \ve^3 |\mu_\ve \bar y +  \Phi |^3 + \ve^3 \mu_\ve \xi_N\, |\mu_\ve
\bar y +  \Phi |^2+ \ve^3 \mu_\ve^2 \xi_N^2|\mu_\ve \bar y +  \Phi |+ \ve^3
\mu_\ve^3 \xi_N^3  \right) \del^2_{ij} W
\\[3mm]
&+&O \left( | \mu_\ve \bar y  + \Phi | \ve + \ve \mu_\ve \xi_N ) \right)
\mathcal{A}_{ab}.
\end{eqnarray*}
Now recall the expansion of $\log({\det g^\e})$ given in Lemma  \ref{lemlogdet}
\begin{eqnarray*}
 \log\big(\det g^\e\big)& = & \log\big(\det \tilde g^\e\big)- 2\e X_N {\rm tr }\,(H )-2\e \G^b_{bk}\,X_k+ \frac{\e^2}3  R_{miil}x_m X_l\\&+&
 \e^2\,\bigg( (\tilde g^\e)^{ab}\, R_{mabl}-\G_{am}^{c}
\G_{cl}^{a}  \bigg)X_m X_l -\e^2 X_N^2 \,{\rm tr }\,(H^2)+ \mathcal{O}(\e^3|X|^3).
\end{eqnarray*}
Hence, differentiating with respect to $X_i$, $X_N$ and $z_a$ (and
performing the change of variables $z=\frac y\ve$ and
$\ov\xi=\frac{\ov X-\Phi}{\mu_\ve}$ and $\xi_N=\frac{
X_N}{\mu_\ve}$) one has
\begin{eqnarray*}
 \pa_{X_N}\log \sqrt{\det g^\e} &=& -\ve tr(H)
-2\,\mu_\ve \ve^2 \xi_N\,tr(H^2)+\mathcal{O}(|(\mu_\ve\xi+\phi)|^2 \ve^3), \\[3mm]
\pa_{X_j}\log \sqrt{\det g^\e} &=&  \ve^2
 \bigg(\frac{1}3
R_{mssj}+ (\tilde g^\epsilon)^{ab}\, R_{mabj}- \G_a^c(E_m)
\G_c^a(E_j)  \bigg)(\mu_\ve\xi_m+\Phi^m) \\
  &+& \mathcal{O}(\ve^3 |(\mu_\ve\xi+\Phi)|^2),
\end{eqnarray*}
and
\begin{eqnarray*}
  \pa_{z_a}\log \sqrt{\det g^\e} &=& \ve \mu_\ve\xi_N\,\pa_{\bar a} tr(H)+\mathcal{O}(\ve^2
  |(\mu_\ve\xi+\Phi)|^2).
\end{eqnarray*}
It then follows
\[
\begin{array}{rllll}
&\mu_\ve^\frac{N+2}{2}\pa_A\left(\log \sqrt{\det
g^\e}\right)\,g^{AB}\pa_B v= \ve \mu_\ve tr(H)\pa_N W
+2\mu_\ve \ve^2 \left(-\mu_\ve\xi_N\, tr(H^2)
\right)\pa_N W\\[3mm]
&+\mu_\ve \ve^2 \bigg(\frac13 R_{mssj}(\mu_\ve\xi_m+\Phi^m)+
 \left\{ (\tilde g^\e)^{ab}\, R_{mabj}- \G_a^c(E_m)
\G_c^a(E_j)  \right\}(\mu_\ve\xi_m+\Phi^m) \bigg)\pa_j W\\[3mm]
 & 
 +\mathcal{A}_{51}+
 \mathcal{B}_2(W),
\end{array}
\]
where
\begin{eqnarray*}
\mathcal{A}_{51}=-\ve^2 \mu_\ve^2\,\xi_N\,\pa_a tr(H) \, \left(
-\frac{N-2}{2}\,\pa_{\bar a}\mu_\ve\,W
+\mu_\ve\,\pa_{\bar a}W
-\left(\pa_{\bar a}\mu_\ve\xi_J\,\pa_{J}v+\pa_{\bar a}\Phi^j\,\pa_{j}W
\right)\right)
\end{eqnarray*}
\begin{eqnarray*}
 \mathcal{B}_2(W)&=&O \left( \ve^2 ( \mu_\ve \bar\xi  + \Phi )^2 + \ve^2 \mu_\ve \xi_N  (\mu_\ve \bar \xi
+ \Phi ) + \ve^2 \mu_\ve^2\xi_N^2 \right) (\ve \mu_\ve  \del_j W+ \ve\mu_\ve  \del_N W)\\
  &+& O \left(\ve^2 ( \mu_\ve \bar \xi + \Phi )^2 + \ve^2\mu_\ve \xi_N  (\mu_\ve \bar \xi
+ \Phi )  +\ve^2 \mu_\ve^2 \xi_N^2 \right) \times\\
&\times& \left( -\frac{N}{2}\,\pa_{\bar a}\mu_\ve\,\pa_lW
+\mu_\ve\,\pa^2_{\bar al}W -\left(\pa_{\bar a}\mu_\ve
\xi_J\pa^2_{lJ}W+\pa_{\bar a}\Phi^j\pa^2_{lj}W\right) \right).
\end{eqnarray*}
Finally, using the properties of the curvature tensor ($R_{iilj}=0$
and $R_{smsj}=-R_{mssj}$)
\begin{eqnarray*}
  \mu_\ve^\frac{N+2}{2}\pa_A (g^{AB})\pa_B v&=&\pa_i (g^{ij})\pa_jv+\pa_a (g^{ab})\pa_bv+\pa_a (g^{aj})\pa_jv+\pa_j (g^{aj})\pa_av \\
  &=& \frac{\ve^2}3\mu_\ve R_{liij}(\mu_\ve \xi_l+\Phi^l)\pa_j W+{\mathcal A}_{52} W+\mathcal{B}_3(W)
 \end{eqnarray*}
where we have set
\begin{eqnarray*}
{\mathcal A}_{52} W  &=&  \left(  \mathfrak{D}_j^a [\mu_\ve  \xi_j+
{\Phi}^j] \ve^2 + \ve \mu_\ve \xi_N\,
\mathfrak{D}_N^a   \right) \times \\[3mm]
&\times & \left\{ \mu_\ve \left[ -\ve D_{\ov\xi}\, W \, [\pa_{\bar a} \Phi ] +
\mu_\ve
 \del_{\bar a} W -\ve \del_{\bar a} \mu_\ve (\frac{N-2}{2} W + D_\xi W \,
[\xi] ) \right] \right\},
\end{eqnarray*}
where $\mathfrak{D}_N^j$ and $\mathfrak{D}_N^a$ are smooth functions
on the variable $z$, and where
\begin{eqnarray*}
 \mathcal{B}_3(W)&=&O \big( \ve^2 ( \mu_\ve \bar\xi  + \Phi )^2 + \ve^2\mu_\ve \xi_N  (\mu_\ve \bar \xi
+ \Phi ) + \ve^2 \mu_\ve^2\xi_N^2 \big)\,\mu_\ve \ve(\del_j W+\pa_{\bar a}W) \\
&+& \big( (( \mu_\ve \bar\xi  + \Phi )^2 +\mu_\ve \xi_N  (\mu_\ve
\bar \xi
+ \Phi ) + \mu_\ve^2\xi_N^2\big)\,\times \\
&\times & \left\{ \mu_\ve \left[ - \ve D_{\ov\xi}\, W \, [\pa_{\bar a} \Phi ] +
\mu_\ve
 \del_{\bar a} W - \ve \del_{\bar a} \mu_\ve (\frac N2 W + D_\xi W \,
[\xi] ) \right] \right\}.
\end{eqnarray*}
Collecting these formulas together and setting
$$
\mathcal{A}_0=\sum_{a=1}^k\mathcal{A}_{aa}, \qquad \mathcal{A}_5=
\mathcal{A}_{51}+\mathcal{A}_{52},
$$
and
 \begin{equation}\label{B}
 B(v)=\mathcal{B}_1(v)+\mathcal{B}_2(v)+\mathcal{B}_3(v),
\end{equation}
the result follows at once.

\medskip
\setcounter{equation}{0}
\section{Appendix: Proof of Theorem \ref{teo4.1}}\label{AA2}

\bigskip
The main ingredient to prove  theorem \ref{teo4.1} is the following

\bigskip

\begin{lemma}\label{lemma1} We assume the same assumptions as in Theorem \ref{teo4.1} and we use the same notations. Then there exists $\ve_0>0 $ such that for all
$0<\e<\e_0$, we have
\begin{equation} \label{leQ0}
E({\delta \over \, \mu_\e} {\mathcal T}_{\, \mu_\e ,\, \Phi_\e
} ( Z_0 ) \bar \chi_\e)= \ve^{-k} P_\e (\delta ),
\end{equation}
\begin{equation} \label{leQj}
E({d^j \over \, \mu_\e} {\mathcal T}_{\, \mu_\e ,\, \Phi_\e }
( Z_ j) \bar \chi_\e)= \e^{-k} Q_\e (d^j),
\end{equation}
\begin{equation} \label{leQ}
E({e\over \, \mu_\e} {\mathcal T}_{\, \mu_\e ,\, \Phi_\e } (Z)
\bar \chi_\e)= \e^{-k} R_\e (e).
\end{equation}
\end{lemma}

\begin{proof}[Proof of Lemma \ref{lemma1}]
Define
\begin{eqnarray}
\label{defFunctionalF}
 F (u ) :&=&\int_{K_\e \times \hat \CC_\e}
\left(\frac12  |\nabla_X u|^2 +\frac12 \e u^2 - \frac1{p+1} u^{p+1}
\right) \sqrt{\det(g^\e)} \, dz \, dX \nonumber\\
&+&\int_{K_\e \times \hat \CC_\e} \frac12\,\Xi_{ij} (\e z ,
X)\,\partial_{i} u \partial_{j} u\sqrt{\det(g^\e)} \, dz \, dX\\
 &+&\frac12 \int_{K_\e \times \hat \CC_\e}  \pa_{\bar a}u \,\pa_{\bar a}u\sqrt{\det(g^\e)} \, dz \, dX
  + \int_{K_\e \times \hat \CC_\e} B(u,u)\sqrt{\det(g^\e)} \, dz \,
  dX.\nonumber
\end{eqnarray}
We refer to Lemma \ref{le5.1} for the definitions of the objects appearing in \equ{defFunctionalF}.

 \smallskip

\noindent {\textbf{Step 1:}} {\it Proof of \equ{leQ0}}. Given a
small $t \not= 0$, a Taylor expansion gives that
\begin{eqnarray}\label{tere1}
&&\left[ DF({\mathcal T}_{\, \mu_\e + t \delta , \, \Phi_\e }
(\hat w ) \bar \chi_\e) - DF({\mathcal T}_{\, \mu_\e  , \bar
\Phi_\e } (\hat w ) \bar \chi_\e )\right] ({\delta \over \bar
\mu_\e} {\mathcal T}_{\, \mu_\e , \, \Phi_\e} (Z_0 ) \bar
\chi_\e)\\
&=&-t \, D^2 F \left( {\mathcal T}_{\, \mu_\e , \, \Phi_\e}
(\hat w) \bar \chi_\e \right) \left[ {\delta \over \, \mu_\e}
{\mathcal T}_{\, \mu_\e , \, \Phi_\e} (Z_0 ) \bar \chi_\e
\right] (1+
O(t))\\
&=& -2t E ({\delta \over \, \mu_\e} {\mathcal T}_{\, \mu_\e ,
\, \Phi_\e} (Z_0 ) \bar \chi_\e ) (1+ O(t)).
\end{eqnarray}

On the other hand, we write for any $\psi$
\begin{equation}
\label{tere2} \left[ DF({\mathcal T}_{\, \mu_\e + t \delta , \bar
\Phi_\e } (\hat w )\bar \chi_\e) - DF({\mathcal T}_{\, \mu_\e  ,
\, \Phi_\e } (\hat w ) \bar \chi_\e )\right] ( \psi) =
\mathfrak{a} (t) - \mathfrak{a} (0)  + \mathfrak{b}(t) +
\mathfrak{c}(t)
\end{equation}
where
\begin{eqnarray*}
\mathfrak{a}(t) &= &\int_{K_\e \times \hat \CC_\e} \left( \nabla_X
{\mathcal T}_{\, \mu_\e + t \delta , \, \Phi_\e } (\hat w ) \bar
\chi_\e \right) \nabla_X \psi + \ve {\mathcal T}_{\, \mu_\e + t
\delta , \, \Phi_\e } (\hat w ) \bar \chi_\e \psi -
\left({\mathcal T}_{\, \mu_\e + t \delta , \, \Phi_\e } (\hat w
) \bar \chi_\e \right)^p \psi
\\
&+& \int_{K_\e \times \hat \CC_\e} \Xi_{ij} (\e z, X) \partial_i
\left( {\mathcal T}_{\, \mu_\e + t \delta , \, \Phi_\e } (\hat w
)\bar \chi_\e \right)
\partial_j \psi,
\end{eqnarray*}
$$
\mathfrak{b}(t)=\int_{K_\e \times \hat \CC_\e} \partial_{\bar a} (
{\mathcal T}_{\, \mu_\e + t \delta , \, \Phi_\e } (\hat w )\bar
\chi_\e )
\partial_{\bar a} \psi - \int_{K_\e \times \hat \CC_\e} \partial_{\bar a} ( {\mathcal T}_{\bar
\mu_\e  , \, \Phi_\e } (\hat w ) \bar \chi_\e ) \partial_{\bar a}
\psi
$$
and
$$
\mathfrak{c}(t)= \int_{K_\e \times \hat \CC_\e} B({\mathcal T}_{\bar
\mu_\e + t \delta , \, \Phi_\e } (\hat w ) \bar \chi_\e , \psi) -
\int_{K_\e \times \hat \CC_\e} B({\mathcal T}_{\bar  \mu_\e  , \bar
\Phi_\e } (\hat w ) \bar \chi_\e , \psi).
$$
\smallskip
We now compute $\mathfrak{a} (t) $ with $\psi = {\delta \over \bar
\mu_\e} {\mathcal T}_{\, \mu_\e + t \delta , \, \Phi_\e } (Z_0 )
\bar \chi_\e $. Performing the change of variables $\bar X = (\bar
\mu_\e + t \delta )\bar \xi + \, \Phi_\e$, $X_N = ( \, \mu_\e +t
\delta )\xi_N$ in the integral $\mathfrak{a} (t)$ and using
\equ{expdeterminante} together with the definition of $\bar
\chi_\ve$ in \equ{chibar}, we get
\begin{eqnarray}\label{I1}
\mathfrak{a} (t) &=&\left[ -\int {\delta \over \, \mu_\e} [ \nabla
\hat w \nabla Z_0 +\e (\, \mu_\e + t \delta )^2 \hat w Z_0
- \hat
w^p Z_0] (1+ \e (\, \mu_\e + t \delta ) \xi_N H_{\alpha \alpha} +
\e^2 O(|\xi|^2)) \right.\nonumber\\
 &+ & \left.    -\int {\delta \over \, \mu_\e} [-2\e (\, \mu_\e +t
\delta ) \xi_N H_{ij} + \e^2 O(|\xi|^2) ] \partial_i \hat w
\partial_j Z_0 \right] \times\\
&\times&\big(1+ O(t)\big) \bigg(1+ O(\ve) + O(\ve^{\gamma
(N-4)})\bigg).\nonumber
\end{eqnarray}
Thus we see immediately from \equ{I1} that
\begin{eqnarray*}
t^{-1} \left[ \mathfrak{a}(t)-\mathfrak{a}(0) \right] &=& \left[ 2\e
\int \delta^2 \hat w Z_0 - \int \e  {\delta^2 \over \, \mu_\e} [
\nabla \hat w \nabla Z_0 +\e \, \mu_\e^2 \hat w Z_0 - \hat w^p Z_0
] \xi_N H_{\alpha
\alpha}  \right.\\
\\
&+&\left. 2\e \int {\delta^2 \over \, \mu_\e} \xi_N H_{ij}
\partial_i \hat w \partial_j  Z_0 \right] \big(1+ O(t)\big) \bigg(1+ O(\ve)
+O(\ve^{\gamma (N-4)})\bigg).
\end{eqnarray*}
Integrating by parts in the $\xi$ variables and using the fact that
$\hat {\mathcal C}_\ve \to \R^N_+$ as $\ve \to 0$ one can write
\begin{eqnarray*}
&&t^{-1} \left[ \mathfrak{a}(t)-\mathfrak{a}(0) \right] = \left[ 2\e
\int \delta^2 \hat w Z_0 - \int \e  {\delta^2 \over \, \mu_\e} [
-\Delta \hat w  +\e \, \mu_\e^2 \hat w  - \hat w^p  ] Z_0 \xi_N
H_{\alpha \alpha}  \right.
\\
&+&\left. \int \e {\delta^2 \over \, \mu_\e}\hat w Z_0 H_{\alpha
\alpha} - 2\e \int {\delta^2 \over \, \mu_\e} \xi_N H_{ij}
\partial_{ij}  \hat w  Z_0 \right] \big(1+ O(t)\big)\bigg(1+ O(\ve) +O(\ve^{\gamma
(N-4)}) \bigg).
\end{eqnarray*}
Now using the fact that $ \|  -\Delta \hat w  +\e \, \mu_\e^2 \hat
w - \hat w^p \|_{\e , N-2} \leq C \e^3$ we get
\begin{eqnarray*}
t^{-1} \left[ \mathfrak{a}(t)-\mathfrak{a}(0) \right] &= & \bigg[
2\e \int \delta^2 \hat w Z_0 +\e [ \int   {\delta^2 \over \mu_0} (
\int_{\R^N_+} H_{\alpha \alpha} \hat w  Z_0 \xi_N -  2 \int_{\R^N_+}
H_{i j } \partial_{ij} \hat w  Z_0 )]\\
&& +\e^2 Q(\delta) \bigg] \,\big(1+ O(t)\big) \bigg(1+
O(\ve)+O(\ve^{\gamma (N-4)})\bigg).
\end{eqnarray*}

Since $N> 6$, we can choose $\gamma = 1-\sigma$, $\sigma >0$ so that
$\gamma (N-4)
>2$. Thanks to the definition of $\mu_0$ given in \equ{mu}, we
conclude that
\begin{equation}
\label{I-II} \mathfrak{a}(t)-\mathfrak{a}(0)= t \e^{-k} \left[ - B\e
\int_K \delta^2  + O(\e^2 ) Q(\delta ) \right] \big(1+ O(t)\big)
\big(1+ O(\ve)\big)
\end{equation}
where
$$-B= (\int_{\R^N_+} w Z_0 ) <0 \quad \hbox{ and }\quad Q(\delta)=\int_K \kappa(y) \delta^2
$$
 for some smooth and uniformly bounded (as $\e \to 0$) function $\kappa$ defined on $K$.

Observe that $\partial_\mu {\mathcal T}_{\mu , \Phi} (\hat w ) =
-{1\over \mu} {\mathcal T}_{\mu , \Phi} ( Z_0 ) (1+ \ve R_0 (z, \xi)
)$, where $R_0$ is a smooth function of the variables $(z, \xi)$,
uniformly bounded in $z$ and satisfying
$$
| R_0(y, \xi) | \leq {C \vartheta(y) \over (1+ |\xi|^{N-2})}
$$
for some positive constant $C$ independent of $\e$, and some generic
function $\vartheta(y)$ defined on $K$, smooth and uniformly bounded
as $\e \to 0$. Hence, recalling the definition of the function
$\mathfrak{b}$ above, a Taylor expansion gives
$$
\mathfrak{b}(t) = -t \int_{K_\e \times \hat \CC_\e} |\partial_{\bar
a} ({\delta \over \, \mu_\e} {\mathcal T}_{\, \mu_\e , \bar
\Phi_\e} ( Z_0 ) \bar \chi_\e  )|^2  (1+ O(t))  .
$$
Observe now that \begin{eqnarray*}
\partial_{\bar a} \left( {\delta \over \, \mu_\e} {\mathcal T}_{\, \mu_\e , \, \Phi_\e} (Z_0 ) \bar \chi_\e \right)
&=& (\partial_{\bar a} \delta ) {1 \over \, \mu_\e} {\mathcal
T}_{\, \mu_\e , \, \Phi_\e} (Z_0 )  \bar \chi_\e  + \delta
\partial_{\bar a} ( {1 \over \, \mu_\e} {\mathcal T}_{\, \mu_\e
, \, \Phi_\e} (Z_0 ) \bar \chi_\e  )
\\
&=& \e ( \partial_a \delta ) {1 \over \, \mu_\e} {\mathcal
T}_{\, \mu_\e , \, \Phi_\e} (Z_0 )\bar \chi_\e   + \e \delta (
\partial_a \, \mu_\e ) \partial_{\, \mu_\e} ({1 \over \bar
\mu_\e} {\mathcal T}_{\, \mu_\e , \, \Phi_\e} (Z_0 ) \bar
\chi_\e )\\
 &+& \e \delta ( \partial_a \, \Phi_\e ) \partial_{\bar
\Phi_\e}  ({1 \over \, \mu_\e} {\mathcal T}_{\, \mu_\e , \bar
\Phi_\e} (Z_0 ) \bar \chi_\e ).
\end{eqnarray*}
Since $ \int \left( {1 \over \, \mu_\e} {\mathcal T}_{\, \mu_\e
, \, \Phi_\e} (Z_0 ) \bar \chi_\e  \right)^2 \, dX= A (1+ o(\e)
)$, we conclude that
\begin{equation}
\label{III1} \mathfrak{b}(t) = - t \e^{-k} \left[ A_\e \e^2 \int_K
|\partial_a ( \delta (1+ o(\e^2 ) \beta_1^\e (y)  )) |^2 \right]
\end{equation}
where $A_\e \in \R$, $\lim_{\e \to 0} A_\e = A= \int_{\R^N_+} Z_0^2
$ and $\beta_1^\e$ is an explicit smooth function in $K$, which is
uniformly bounded as $\e \to 0$.   Finally we observe that the last
term $\mathfrak{c}(t)$ defined above is of lower order, and can be
absorbed in the terms described in \equ{I-II} and \equ{III1}.

The expansion \equ{leQ0} clearly holds from
\equ{tere1}-\equ{tere2}-\equ{I-II} and \equ{III1}.

\medskip

\noindent {\textbf{Step 2:}} {\it Proof of \equ{leQj}}. To get the
expansion in \equ{leQj} we argue in the same spirit as before. Let
$d $ be the vector field along $K$ defined by $d (\ve z) = (d^1 (\ve
z ) , \ldots , d^{N-1} (\ve z) )$. For any $t$ small and $t\not= 0$,
we have (see \equ{defFunctionalF})
 \begin{eqnarray*} &&\left[ DF
({\mathcal T}_{\, \mu_\ve , \, \Phi_\ve + t d} (\hat w) \bar
\chi_\e) - DF ({\mathcal T}_{\, \mu_\ve , \, \Phi_\ve } (\hat w)
\bar \chi_\e )\right] [\varphi] = \\
&& t D^2 F ({\mathcal T}_{\, \mu_\ve , \, \Phi_\ve } (\hat w)
\bar \chi_\e ) \left[\sum_l {d^l \over \, \mu_\ve} {\mathcal
T}_{\, \mu_\ve , \, \Phi_\ve} (Z_l ) \bar \chi_\e \right]
[\varphi] \,\big(1+ O(t)\big) \big(1+ O(\ve)\big)
\end{eqnarray*}
for any function $\varphi \in H^1_\ve$. In particular, choosing
$\varphi = {d^j \over \, \mu_\ve} {\mathcal T}_{\, \mu_\ve ,
\, \Phi_\ve} (Z_j) \bar \chi_\e$ we get (using the fact that
$\int_{\R^N_+} Z_j Z_l = C_0 \delta_{jl}$) that
\begin{eqnarray}
&&\left[ DF ({\mathcal T}_{\, \mu_\ve , \, \Phi_\ve + t d} (\hat
w) \bar \chi_\e) - DF ({\mathcal T}_{\, \mu_\ve , \, \Phi_\ve }
(\hat w) \bar \chi_\e )\right] [{d^j \over \, \mu_\ve} {\mathcal
T}_{\, \mu_\ve , \, \Phi_\ve} (Z_j)\bar \chi_\e]
\label{leQj1}=\nonumber\\
&& 2 t E ({d^j \over \, \mu_\ve} {\mathcal T}_{\, \mu_\ve ,
\Phi_\ve} (Z_j)\bar \chi_\e)
  (1+ O(t)) (1+ O(\ve)).
\end{eqnarray}
On the other hand, as in the previous step,  we write
\begin{eqnarray}
&&\label{leQj2} \left[ DF ({\mathcal T}_{\, \mu_\ve , \, \Phi_\ve
+ t d} (\hat w) \bar \chi_\e) - DF ({\mathcal T}_{\, \mu_\ve ,
\, \Phi_\ve } (\hat w) \bar \chi_\e )\right] [{d^j \over \bar
\mu_\ve} {\mathcal T}_{\, \mu_\ve , \, \Phi_\ve} (Z_j) \bar
\chi_\e]=\nonumber\\&& \mathfrak{a}_2(t) - \mathfrak{a}_2(0) + \mathfrak{b}_2(t)
+ \mathfrak{c}_2(t)
\end{eqnarray}
where we have set, for $\psi = {d^j \over \, \mu_\ve} {\mathcal
T}_{\, \mu_\ve , \, \Phi_\ve} (Z_j) \bar \chi_\e$,
\begin{eqnarray*}
\mathfrak{a}_2(t)& =& \int \left( \nabla_X {\mathcal T}_{\, \mu_\e
, \, \Phi_\e +t d } (\hat w )  \bar \chi_\e \right) \nabla_X \psi
+ \ve {\mathcal T}_{\, \mu_\e  , \, \Phi_\e +t d } (\hat w )
\bar \chi_\e \psi - \left({\mathcal T}_{\, \mu_\e  , \, \Phi_\e
+t d}
(\hat w \bar \chi_\e )\right)^p \psi\\
&+& \int \Xi_{ij} (\e z, X)
\partial_i \left( {\mathcal T}_{\, \mu_\e  , \, \Phi_\e +t d }
(\hat w ) \bar \chi_\e \right) \partial_j \psi,
\end{eqnarray*}
$$
\mathfrak{ b}_2(t) =\int \partial_{\bar a} ( {\mathcal T}_{\bar
\mu_\e  , \, \Phi_\e +t d} (\hat w ) \bar \chi_\e) \partial_{\bar
a} ( {\mathcal T}_{\, \mu_\e  , \, \Phi_\e +t d } (\hat w ) \bar
\chi_\e) - \int \partial_{\bar a} ( {\mathcal T}_{\, \mu_\e  ,
\, \Phi_\e } (\hat w ) \bar \chi_\e) \partial_{\bar a} ( {\mathcal
T}_{\, \mu_\e , \, \Phi_\e } (\hat w ) \bar \chi_\e)
$$
and
$$
\mathfrak{c}_2 (t)= \int B({\mathcal T}_{\bar  \mu_\e  , \bar
\Phi_\e +t d } (\hat w ) \bar \chi_\e , \psi) - \int B({\mathcal
T}_{\bar  \mu_\e , \, \Phi_\e } (\hat w ) \bar \chi_\e , \psi).
$$
We now compute $\mathfrak{a}_2 (t) $ with $\psi = {\delta \over \bar
\mu_\e} {\mathcal T}_{\, \mu_\e  , \, \Phi_\e } (Z_j ) $.
Defining the tensor ${\mathcal R}_{ml}$ by
$${\mathcal R}_{ml}=\bigg( (\tilde g^\e)^{ab}\,R_{mabl}-\G_a^c(E_m)
\G_c^a(E_l) \bigg),$$ performing the change of variables $\bar X =
(\, \mu_\e + t \delta )\bar \xi + \, \Phi_\e$, $X_N = ( \bar
\mu_\e +t \delta )\xi_N$ in the integral $\mathfrak{a}_2 (t)$, using
\equ{expdeterminante}, \equ{energydensity} and recalling the
definition of the cut-off function  $\bar \chi_\e$, we get
\begin{eqnarray*}
&&\mathfrak{a}_2(t) = \bigg\{ \int {d^j \over \, \mu_\ve} \bigg[
\nabla \hat w \nabla Z_j +\ve \, \mu_\ve^2 \hat w Z_j - p \hat w^p
Z_j\bigg] \times \bigg[ 1- \ve \, \mu_\ve \xi_N H_{\alpha \alpha}
\\
&&  +\ve^2 ({R_{miil} \over 6} + {{\mathcal R}_{ml} \over 2}) (\ve
\, \mu_\ve \xi_m +\, \Phi_{\ve m} + t d^m )  (\ve \bar \mu_\ve \xi_l
+\, \Phi_{\ve l} + t d^l ) + O(\ve^3 |\xi|^3 ) \bigg]
\\
&+& \int {d^j \over \, \mu_\ve} \bigg[  2 \ve \xi_N H_{ir} -{\ve^2
\over 3} R_{imlr} (\, \mu_\ve \xi_m +\, \Phi_{\ve m} + t d^m
) (\, \mu_\ve \xi_l +\, \Phi_{\ve l} + t d^l ) \\
&&\qquad \qquad+ O(\ve^3 |\xi|^3 ) \bigg] \partial_i \hat w
\partial_r Z_j \bigg\} \times
\\&&\times\bigg(1+ O(t)\bigg) \,\bigg(1+ O(\ve ) + O(\e^{\gamma
(N-3)})\bigg).
\end{eqnarray*}
Thus we immediately get (using the fact that $\gamma (N-3) >1$)
\begin{eqnarray*}
&&t^{-1} [\mathfrak{a}_2(t) - \mathfrak{a}_2(0) ] = \ve^2 \left\{
\int {d^j \over \, \mu_\ve} [ \nabla \hat w \nabla Z_j +\ve \bar
\mu_\ve^2 \hat w Z_j - p \hat w^p Z_j ] \right.\times
\\
&&\times ({R_{mijl} \over 6} + {{\mathcal R}_{lm} \over 2} ) [( \bar
\mu_\ve \xi_m +\, \Phi_{\ve m} ) d^l +  (\, \mu_\ve \xi_l +\bar
\Phi_{\ve l}) d^m ]
\\
&&\left. - \int {d^j \over \, \mu_\ve} {R_{ilmr} \over 3} [(\bar
\mu_\ve \xi_m  +\, \Phi_{\ve m} ) d^l +  ( \, \mu_\ve \xi_l
+\, \Phi_{\ve l}) d^m ]
\partial _i \hat w \partial_r Z_j \right\}  (1+ O(\ve )) (1+ O(t)).
\end{eqnarray*}
Integration by parts in the $\xi$ variables and using the fact that
$\hat {\mathcal C}_\ve \to \R^N_+$ as $\ve \to 0$, we get
\begin{eqnarray*}
&&t^{-1} [\mathfrak{a}_2(t) - \mathfrak{a}_2(0) ] = \ve^2 \left\{
\int{d^j \over \, \mu_\ve} [-\Delta \hat w +\ve \, \mu_\ve \hat
w - p \hat w ] Z_j\right.\times\\
&&\left. \times ({R_{mijl} \over 6} + {{\mathcal R}_{lm} \over 2} )
[( \, \mu_\ve \xi_m +\, \Phi_{\ve m} ) d^l + (\, \mu_\ve \xi_l
+\, \Phi_{\ve l}) d^m ]\right.
\\
&&\left. - \int ({R_{mijl} \over 6} + {{\mathcal R}_{lm} \over 2} )
d^j  [  \partial_l \hat w d^m +  \partial_m \hat w d^l ] Z_j + \int
d^j [{R_{ilrr} \over 3}  d^l + {R_{irmr} \over 3} d^m ] Z_i  Z_j
\right\}\times \\
&&\times \big(1+ O(\ve )\big) \big(1+ O(t)\big).
 \end{eqnarray*}
Now using the fact that $ \| -\Delta \hat w  +\e \, \mu_\e^2 \hat
w - \hat w^p \|_{\e , N-2} \leq C \e^3$ and that $R_{ilrr} = 0$, we
deduce that
\begin{eqnarray}\label{leQj3}
t^{-1} [\mathfrak{a}_2(t) - \mathfrak{a}_2(0) ] &= &\ve^2 \left\{ -
C \int_{K_e} ({R_{miij} \over 3} + {{\mathcal R}_{mj} \over 2} ) d^j
d^m   +C \int_{K_\e}  {R_{jrrm} \over 3} d^m  d^j \right\} \times\nonumber\\
&\times& \big(1+ o(\ve )\big) \,\big(1+ O(t)\big)
\\&=& \ve^{-k} \ve^2 \left[
- C \int {{\mathcal R}_{mj} \over 2}  d^j d^m + O(\e) Q(d) \right]
\big(1+ O(t)\big)\nonumber
\end{eqnarray}
where here we have set
$$
C= \int_{\R^N_+} Z_1^2 \qquad \hbox{ and } \quad Q(d):=\int_K \pi(y)
d^i d^j
$$
for some smooth and uniformly bounded (as $\e \to 0$) function
$\pi(y)$.  To estimate the term $\mathfrak{b}_2$ above we argue as
in \equ{III1}, we get that
\begin{equation}
\label{leQj4} t^{-1} \mathfrak{b}_2(t) = -\ve^{-k} \left[ \e^2 C_\e
\int_K |\partial_a (d^j (1+ \beta_2^\e (y) o(\e^2) ))|^2  \right]
(1+ O(t)).
\end{equation}
Finally we observe that the last term $\mathfrak{c}_2(t)$ is of
lower order, and can be absorbed in the terms described in
\equ{leQj3} and \equ{leQj4}. We get the expansion \equ{leQj} from
\equ{leQj1}-\equ{leQj2}-\equ{leQj3} and \equ{leQj4}.

\medskip
\noindent{ \textbf{Step 3}:} {\it Proof of \equ{leQ}}. To get the
expansion in \equ{leQ}, we compute
\begin{equation}
\label{tere3} E({e\over \, \mu_\e} {\mathcal T}_{\, \mu_\e ,
\, \Phi_\e} (Z)) = I + II + III
\end{equation}
where
\begin{eqnarray*}
I&=& \int_{K_\e \times \hat \CC_\e} {\delta^2 \over  \mu_\e^2}
\left(\frac12 ( |\nabla_X {\mathcal T}_{ \mu_\e , \Phi_\e} (Z) |^2
+\e {\mathcal T}_{\mu_\e ,  \Phi_\e} (Z)^2 - p V_\e^{p-1} {\mathcal
T}_{ \mu_\e ,  \Phi_\e} (Z)^2 ) \right) \sqrt{\det g^\e} dz dX
\\
&+&\int_{K_\e \times \hat \CC_\e} {\delta^2 \over \, \mu_\e^2}
\frac12\,\Xi_{ij} (\e z , X) \,\partial_{i} {\mathcal T}_{\bar
\mu_\e , \, \Phi_\e} (Z) \partial_{j}{\mathcal T}_{\, \mu_\e , \,
\Phi_\e} (Z) \,\sqrt{\det g^\e} \, dz \, dX,
\\[3mm]
II&=& \frac12 \int_{K_\e \times \hat \CC_\e} \pa_{\bar a} \left(
{e\over \, \mu_\e} {\mathcal T}_{\, \mu_\e , \, \Phi_\e} (Z) \right)
\,\pa_{\bar a} \left( {e\over \, \mu_\e} {\mathcal T}_{\, \mu_\e ,
\, \Phi_\e} (Z) \right)\,\sqrt{\det g^\e} \, dz
\, dX\\[3mm]
\hbox{and}\\
III&=& \int_{K_\e \times \hat \CC_\e} B({e\over \, \mu_\e}
{\mathcal T}_{\, \mu_\e , \, \Phi_\e} (Z), {e\over \, \mu_\e}
{\mathcal T}_{\, \mu_\e , \, \Phi_\e} (Z))\,\sqrt{\det(g^\e)} \,
dz \, dX.
\end{eqnarray*}
Using the change of variables $\bar X = \, \mu_\e \bar \xi + \bar
\Phi_\e$, $X_N = \, \mu_\e \xi_N$ in $I$, we can write
$$
I= \int {1\over 2} {\delta^2 \over \, \mu_\e^2} \bigg[ |\nabla
Z|^2 - p \hat w^{p-1} Z^2 +\e \, \mu_\e^2 Z^2 \bigg] \bigg(1+ \e
O(e^{-|\xi|})\bigg).
$$
Then, recalling the definition of $\la_0$ in \equ{lambda0}, we get
\begin{equation} \label{tere4}
I=\e^{-k} \left[  -{\la_0 \over 2} D \int_K e^2 + \e Q(e) \right]
\end{equation}
where we have set
$$D= \int_{\R^N_+} Z^2(\xi)\,d\xi \quad \hbox{and }\qquad Q(e):= \int_K \tau(y)
e^2\,dy,$$ for some  smooth and uniformly bounded, as $\ve \to 0$,
function $\tau$. On the other hand, using a direct computation and
arguing as in \equ{III1}, we get
\begin{equation} \label{ttere4}
II= {D_\e \over 2} \int_{K_\e} |\partial_{\bar a} e  + e^{-\la_0
\ve^{-\gamma}} \beta_3^\e (\e z) e |^2 =
 \ve^{-k}
 \bigg[ {D_\e \over 2} \ve^2 \int_{K} |\partial_{ a} (e (1+ e^{-\la'\ve^{-\gamma}} \beta_3^\e (y) )) |^2 \bigg]
\end{equation}
where $\beta_3^\e$ is an explicit smooth function on $K$, which is
uniformly bounded as $\e \to 0$, while $\lambda'$ is a positive real
number. Finally we observe that the last term $III$ is of lower
order, and can be absorbed in the terms described in \equ{tere4} and
\equ{ttere4}. This concludes the proof of \equ{leQ}.
\end{proof}

\medskip

\noindent We have now the elements to prove Theorem \ref{teo4.1}.

\begin{proof}[\rm\textbf{Proof of Theorem \ref{teo4.1}}] Given the result in Lemma \ref{lemma1}, we can write
\begin{eqnarray*}
{\mathcal M} (\phi^\bot , \delta , d, e) &=& E(\phi ) - E(\phi^\bot
) -
 E({\delta \over \, \mu_\e} {\mathcal T}_{\, \mu_\e ,\, \Phi_\e } ( Z_0 ) \bar \chi_\e)
 - \sum_{j=1}^{N-1} E({d_j \over \, \mu_\e} {\mathcal T}_{\, \mu_\e ,\, \Phi_\e } ( Z_j ) \bar
 \chi_\e)\\
&-& E({e \over \, \mu_\e} {\mathcal T}_{\, \mu_\e ,\, \Phi_\e
} ( Z ) \bar \chi_\e).
\end{eqnarray*}
Thus it is clear that the term ${\mathcal M}$ recollects all the
mixed terms in the expansion of $E(\phi)$. Indeed, if we define
\begin{eqnarray*}
m(f,g)& =& \int_{K_\e \times \hat \CC_\e} \left(\nabla_X f \nabla_X
g
 +\e fg  - p V_\e^{p-1} fg    \right) \sqrt{\det(g^\e)} \, dz \, dX
\\
&+&\int_{K_\e \times \hat \CC_\e} \,\Xi_{ij} (\e z ,
X)\,\partial_{i}f\partial_{j}g\,\sqrt{\det(g^\e)} \, dz \, dX\\
&+& \int_{K_\e \times \hat \CC_\e} \pa_{\bar a}f\,\pa_{\bar
a}g\,\sqrt{\det(g^\e)} \, dz \, dX+ \int_{K_\e \times \hat \CC_\e}
 B(f,g)\,\sqrt{\det(g^\e)} \, dz \, dX
\end{eqnarray*}
for $f$ and $g$ in $H^1_\e$, then
\begin{eqnarray}\label{auxi}
{\mathcal M}(\phi^\bot , \delta , d , e) &=& m (\phi^\bot , {\delta
\over \, \mu_\e} {\mathcal T}_{\, \mu_\e ,\, \Phi_\e } ( Z_0 )
\bar \chi_\e ) +\sum_j  m (\phi^\bot , {d_j \over \, \mu_\e}
{\mathcal T}_{\, \mu_\e ,\, \Phi_\e } ( Z_j ) \bar \chi_\e
)\nonumber
\\
&+& m (\phi^\bot , {e \over \, \mu_\e} {\mathcal T}_{\, \mu_\e
,\, \Phi_\e } ( Z ) \bar \chi_\e ) + \sum_j m ({\delta \over \bar
\mu_\e} {\mathcal T}_{\, \mu_\e ,\, \Phi_\e } ( Z_0 ) \bar
\chi_\e , {d^j \over \, \mu_\e} {\mathcal T}_{\, \mu_\e ,\bar
\Phi_\e } ( Z_j ) \bar \chi_\e )\nonumber
\\
&+&\sum_{i\not= j} m ({d^j \over \, \mu_\e} {\mathcal T}_{\bar
\mu_\e ,\, \Phi_\e } ( Z_j ) \bar \chi_\e , {d_i \over \bar
\mu_\e} {\mathcal T}_{\, \mu_\e ,\, \Phi_\e } ( Z_i ) \bar
\chi_\e )
\\
&+& m ({\delta \over \, \mu_\e} {\mathcal T}_{\, \mu_\e ,\bar
\Phi_\e } ( Z_0 ) \bar \chi_\e , {e \over \, \mu_\e} {\mathcal
T}_{\, \mu_\e ,\, \Phi_\e } ( Z ) \bar \chi_\e )\nonumber\\
& +& \sum_j m ({d^j \over \, \mu_\e} {\mathcal T}_{\, \mu_\e ,\,
\Phi_\e } ( Z_j ) \bar \chi_\e , {e \over \, \mu_\e} {\mathcal
T}_{\, \mu_\e ,\, \Phi_\e } ( Z ) \bar \chi_\e ).\nonumber
\end{eqnarray}
One can see clearly that ${\mathcal M}$ is homogeneous of degree $2$
and that its first derivatives with respect to its variables is a
linear operator  in $(\phi^\bot , \delta, d , e)$. We will then show
the validity of estimate \equ{estMM}. In a very similar way one
shows the validity of \equ{lips}. To prove \equ{estMM}, we should
treat each one of the above terms. Since the computations are very
similar, we will limit ourselves to treat the term
$$m:= m ({\delta
\over \, \mu_\e} {\mathcal T}_{\, \mu_\e ,\, \Phi_\e } ( Z_0 )
\bar \chi_\e , {d^j \over \, \mu_\e} {\mathcal T}_{\, \mu_\e
,\, \Phi_\e } ( Z_j ) \bar \chi_\e ).
$$
This term can be written as
\begin{equation}
m= \sum_{i=1}^5 m_i
\end{equation}
where
$$
m_1 = \int_{K_\e \times \hat \CC_\e} \left(\nabla_X f \nabla_X g
  - p V_\e^{p-1} fg    \right) \sqrt{\det(g^\e)} \, dz \, dX
$$
$$
m_2 =\int_{K_\e \times \hat \CC_\e}
 \e fg  \, \sqrt{\det(g^\e)} \, dz \, dX, \quad
m_3 = \int_{K_\e \times \hat \CC_\e} \,\Xi_{ij} (\e z ,
X)\,\partial_{i}f\partial_{j}g\,\sqrt{\det(g^\e)} \, dz \, dX
$$
$$
m_4 = \int_{K_\e \times \hat \CC_\e} \pa_{\bar a}f\,\pa_{\bar a}g\,
\sqrt{\det(g^\e)} \, dz \, dX \quad \hbox{and} \quad  m_5 =
\int_{K_\e \times \hat \CC_\e} B(f,g)\,\sqrt{\det(g^\e)} \, dz \, dX
$$
with $f= {\delta \over \, \mu_\e} {\mathcal T}_{\, \mu_\e ,\bar
\Phi_\e } ( Z_0 ) \bar \chi_\e$ and $g=  {d^j \over \, \mu_\e}
{\mathcal T}_{\, \mu_\e ,\, \Phi_\e } ( Z_j ) \bar \chi_\e$.
Using the fact that the function $Z_0$ solves
$$\Delta Z_0+ p w_0^{p-1} Z_0 = 0 \quad \hbox{in}\quad \R^N,
$$
with  $\int_{\R^N_+}
\partial_{\xi_N} Z_0 Z_j = 0$ and integrating by parts in the $X$
variable (recalling the expansion of $\sqrt{{\mbox {det}} g^\e}$),
one gets
\begin{eqnarray*}
m_1 &=& \left\{ \int {\delta d^j \over \, \mu_\e^2} [-\Delta
{\mathcal T}_{\, \mu_\e , \, \Phi_\e } (Z_0) - p V_\e^{p-1}
{\mathcal T}_{\, \mu_\e , \, \Phi_\e} (Z_0 ) ] \bar \chi_\e^2
{\mathcal T}_{\, \mu_\e , \, \Phi_\e} (Z_j ) \sqrt{{\mbox {det}}
g^\e} \right.
\\
&+& \left. \int {\delta d^j \over \, \mu_\e^2} \partial_{\xi_N}
\left( {\mathcal T}_{\, \mu_\e , \, \Phi_\e} (Z_0) \bar \chi_\e
\right) {\mathcal T}_{\, \mu_\e , \, \Phi_\e} (Z_j )
\,\frac{1}{\, \mu_\e}\,(\e \,tr(H)+O(\e^2))\,\bar \chi_\e \right\}
(1+ o(1) )
\end{eqnarray*}
where $o(1) \to 0$ as $\e \to 0$. Thus, a H\"older inequality yields
$$
| m_1 | \leq C \e^{-k}  \e^{\gamma (N-2)} \| \delta \|_{{\mathcal
L}^2 (K)} \| d^j \|_{{\mathcal L}^2 (K)}.
$$
On the other hand, using the orthogonality condition $\int_{\R^N_+}
Z_0 Z_j =0$, we get
$$
|m_2 | \leq C \e \e^{-k} (\int_{|\xi|>\e^{-\gamma}} Z_0 Z_j )   \|
\delta \|_{{\mathcal L}^2 (K)} \| d^j \|_{{\mathcal L}^2 (K)} \leq C
\e^{-k} \e^{1+\gamma (N-3) } \| \delta \|_{{\mathcal L}^2 (K)} \|
d^j \|_{{\mathcal L}^2 (K)}.
$$
Now, since $\int_{\R^N_+} \xi_N \partial_i Z_0 \partial_l Z_j = 0$,
for any $i, j, l=1, \ldots , N-1$, one gets
\begin{eqnarray*}
|m_3| &\leq &C \e \e^{-k} \left( \int_{|\xi|>\e^{-\gamma}} \xi_N
\partial_i Z_0 \partial_l Z_j \right) \| \delta \|_{{\mathcal L}^2
(K)} \| d^j \|_{{\mathcal L}^2 (K)} \\
&\leq& C \e^{-k} \e^{1+\gamma (N-2)} \| \delta \|_{{\mathcal L}^2
(K)} \| d^j \|_{{\mathcal L}^2 (K)}.
\end{eqnarray*}
 A direct computation on the term $m_4$ gives
\begin{eqnarray*}
|m_4 | &\leq& C \e^{-k} \left\{ \e^2  (\int_{|\xi|>\e^{-\gamma}} Z_0
Z_j ) \| \partial_a \delta \|_{{\mathcal L}^2 (K)} \| \partial_a d^j
\|_{{\mathcal L}^2 (K)} \right.
\\
&+& \e (\int_{|\xi|>\e^{-\gamma}} Z_0 Z_j ) ( \|  \delta
\|_{{\mathcal L}^2 (K)} \| \partial_a d^j \|_{{\mathcal L}^2 (K)}  +
\|  \partial_a \delta \|_{{\mathcal L}^2 (K)} \| d^j \|_{{\mathcal
L}^2 (K)} )
\\
&+& \left.  (\int_{|\xi|>\e^{-\gamma}} Z_0 Z_j ) \|  \delta
\|_{{\mathcal L}^2 (K)} \|  d^j \|_{{\mathcal L}^2 (K)} \right\}
\\
&\leq& C \e^{-k} \e^{\gamma (N-3)} [ \| \delta \|_{{\mathcal H}^1
(K)}^2 + \| d^j \|_{{\mathcal H}^1 (K)}^2 ].
\end{eqnarray*}
Since $|m_5| \leq C \sum_{j=1}^4 |m_j|$ we conclude that
$$
|m| \leq C \e^{-k} \e^{\gamma (N-3)} [ \| \delta \|_{{\mathcal H}^1
(K)}^2 + \| d^j \|_{{\mathcal H}^1 (K)}^2 ].
$$
Each one of the terms appearing in \equ{auxi} can be estimated to
finally get the validity of \equ{estMM}. This conclude the proof of
Theorem \ref{teo4.1}.

\end{proof}

\

\noindent {\bf Acknowledgments}. This work has been supported by grants Fondecyt 1110181, 1100164 and 1120151,
and Fondo Basal CMM.  We would like to thank the anonymous referees for a careful reading of the paper, which meant
an important improvement of the original version of the paper.

\


\begin{thebibliography}{9999}


\bibitem{am0} Adimurthi, Mancini, G.   {\em The Neumann problem for elliptic equations with critical nonlinearity},
A tribute in honour of G. Prodi. Scuola Norm. Sup. Pisa 9--25
(1991).



  \bibitem{am1} Adimurthi, Mancini, G.  {\em Geometry and topology of the boundary in the critical Neumann
problem.} J. Reine Angew. Math. 456, 1--18 (1994).


  \bibitem{amy}  Adimurthi,  Mancini G.,  Yadava S.L., {\em The role of the mean curvature in semilinear Neumann
problem involving critical exponent,} Comm. Partial Differential
Equations 20 (3-4) (1995) 591--631.


 \bibitem{apy} Adimurthi, Pacella, F., Yadava, S.L.   {\em Interaction between the geometry of the boundary and
positive solutions of a semilinear Neumann problem with critical
nonlinearity.} J. Funct. Anal. 113, 318--350 (1993).

\bibitem{apy1}
Adimurthi, Pacella F.,  Yadava  S. L., {\em Characterization of
concentration points and L1- estimates for solutions of a semilinear
Neumann problem involving the critical Sobolev exponent}, Diff.
Integ. Equ. 8 (1995), 42--68.


\bibitem{ao1} Ao W.,  Musso M., Wei J. {\em On Spikes Concentrating on Line-Segments to a Semilinear Neumann Problem}.  J. Differential Equations 251 (2011), no. 4-5, 881--901.



\bibitem{ao2} Ao W., Musso M., Wei J. {\em Triple Junction Solutions for a Singularly Perturbed Neumann Problem}.  SIAM J. Math. Anal. 43 (2011), no. 6, 2519--2541.

\bibitem{b}
Byeon, J. {\em Singularly perturbed nonlinear Neumann problems with
a general nonlinearity.}
 J. Differential Equations 244 (2008), no. 10, 2473--2497.


\bibitem{cgs}
Caffarelli, L., Gidas, B., Spruck, J. {\em Asymptotic symmetry and
local behavior of semilinear elliptic equations with critical
Sobolev growth.} Comm. Pure Appl. Math. 42, 271--297 (1989).


\bibitem{cao}
Cao, D., Kupper, T. {\em On the existence of multipeaked solutions
to a semilinear Neumann problem.} Duke Math. J. 97 (1999), no. 2,
261--300.


\bibitem{chavel}
I. Chavel, Riemannian Geometry, a Modern Introduction, Cambridge
Tracts in Math., vol. 108, Cambridge Univ. Press, Cambridge, 1993.


\bibitem{dy}
Dancer, E. N., Yan, S. {\em Multipeak solutions for a singularly
perturbed Neumann problem.} Pacific J. Math. 189, 241--262 (1999).



\bibitem{df0}
 del Pino, M., Felmer, P. {\em Spike-layered solutions of singularly perturbed elliptic problems in a degenerate setting.}
  Indiana Univ. Math. J. 48  no. 3, 883--898  (1999).


%




\bibitem{dfw}
del Pino, M., Felmer, P., Wei, J. {\em On the role of mean curvature
in some singularly perturbed Neumann problems.} SIAM J. Math. Anal.
31, 63--79 (2000).




\bibitem{dkwy}
del Pino, M.,  Kowalczyk, M.,  Wei J.,  Yang J.,  {\em Interface
foliations near minimal submanifolds in Riemannian manifolds with
positive Ricci curvature}. Geom. Functional Analysis No. 20 (2010)
918--957.




\bibitem{dmp}
del Pino, M., Musso, M., Pistoia, A. {\em Supercritical boundary
bubbling in a semilinear Neumann problem.} Ann. Inst. H. Poincare
Anal. Non-Linearie 22 (1), 45--82 (2005).

\bibitem{dmpa}
del Pino, M., Musso M., Pacard F., {\em Bubbling along geodesics
near the second critical exponent.} J. Eur. Math. Soc. (JEMS) 12
(2010), no. 6, 1553--1605.

\bibitem{gg}
Ghoussoub, N., Gui, C. {\em Multi-peak solutions for a semilinear
Neumann problem involving the critical Sobolev exponent.} Math. Z.
229, 443--474 (1998).


\bibitem{ggz}
 Ghoussoub, N., Gui, C., Zhu, M. {\em On a singularly perturbed Neumann problem with the
critical exponent.} Comm. Partial Differential Equations 26,
1929--1946 (2001).

\bibitem{gm}
Gierer, A., Meinhardt, H. {\em A theory of biological pattern
formation.}  Kybernetik (Berlin) 12, 30--39 (1972).




\bibitem{gp}
 Grossi, M., Pistoia, A., Wei, J. {\em Existence of multipeak solutions for a semilinear elliptic
problem via nonsmooth critical point theory. } Calc. Var. Partial
Differential Equations 11, 143--175 (2000).


\bibitem{g}
Gui, C. {\em Multi-peak solutions for a semilinear Neumann problem.}
Duke Math. J. 84, 739--769 (1996).









\bibitem{gl}
Gui, C., Lin, C.-S. {\em Estimates for boundary-bubbling solutions
to an elliptic Neumann problem.} J. Reine Angew. Math. 546, 201--235
(2002).




\bibitem{gw}
Gui, C., Wei, J. {\em Multiple interior peak solutions for some
singularly perturbed Neumann problems.} J. Differential Equations
158, 1--27 (1999).





\bibitem{l}
Li, Y.Y. {\em On a singularly perturbed equation with Neumann
boundary condition.} Comm. Partial Differential Equations 23,
487--545 (1998).

\bibitem{lin}
Lin, C.-S. {\em Locating the peaks of solutions via the maximum
principle, I. Neumann problem.} Comm. Pure Appl. Math. 54, 1065--1095
(2001).




\bibitem{lnt}

Lin, C.-S., Ni,W.-M., Takagi, I.{\em  Large amplitude stationary
solutions to a chemotaxis system.} J. Differential Equations 72,
1--27 (1988).

\bibitem{lww}
 Lin, C.-S. Wang, L., Wei, J. {\em Bubble accumulations in an elliptic Neumann problem with critical Sobolev exponent.} Calc. Var.
  Partial Differential Equations 30 (2007), no. 2, 153--182.

\bibitem{lnw}
Lin, F.-H. , Ni, W.-M., Wei, J. {\em On the number of interior peak
solutions for a singularly perturbed Neumann problem.}
 Comm. Pure Appl. Math. 60 (2007), no. 2, 252--281.




\bibitem{msw}
 Maier-Paape, S., Schmitt, K.,Wang, Z.Q. {\em On Neumann problems for semilinear elliptic equations
with critical nonlinearity existence and symmetry of multi-peaked
solutions.} Comm. Partial Differential Equations 22, 1493--1527
(1997).


\bibitem{mmah}
Mahmoudi F.,   Malchiodi A., {\em Concentration on minimal
submanifolds for a singularly perturbed Neumann problem.} Adv. Math.
209 (2007), no. 2, 460--525.


\bibitem{mmp} Mahmoudi, F.,  Mazzeo, R., Pacard, F. {\em Constant mean
curvature hypersurfaces condensing along a submanifold.}
Geom. funct. anal. Vol.  {16} (2006) 924--958.

\bibitem{mm1}
 Malchiodi A.,  Montenegro M., {\em Boundary concentration phenomena
for a singularly perturbed elliptic problem}, Comm. Pure Appl. Math,
15 (2002), 1507--1568.

\bibitem{mm2}
Malchiodi A., Montenegro M., {\em
Multidimensional Boundary-layers for a singularly perturbed Neumann
problem,} Duke Math. J. 124:1 (2004), 105--143.


\bibitem{m}
 Malchiodi A., {\em Concentration at curves for a singularly perturbed Neumann problem
in three-dimensional domains.} Geom. Funct. Anal. 15 (2005), no. 6,
1162--1222.


\bibitem{ni0}
 Ni, W.-M. {\em Diffusion, cross-diffusion, and their spike-layer steady states.}
  Notices Amer. Math. Soc. 45 (1998), no. 1, 9--18.

\bibitem{ni}
 Ni, W.-M. {\em Qualitative properties of solutions to elliptic problems.} In Stationary Partial Differential
Equations. Handbook Differential Equations, vol. I, pp. 157--233.
North-Holland, Amsterdam (2004).


\bibitem{npt}
Ni, W.-M., Pan, X.-B., Takagi, I. {\em Singular behavior of
least-energy solutions of a semi-linear Neumann problem involving
critical Sobolev exponents.} Duke Math. J. 67, 1--20 (1992).

\bibitem{nt1}
 Ni, W.-M., Takagi, I. {\em
On the shape of least-energy solutions to a semi-linear problem
Neumann problem.} Comm. Pure Appl. Math. 44, 819--851 (1991).

\bibitem{nt2}
 Ni, W.-M., Takagi, I. {\em Locating the peaks of least-energy solutions to a semi-linear Neumann
problem.} Duke Math. J. 70, 247--281 (1993).



\bibitem{Po}
S. Pohozaev, {\em Eigenfunctions of the equation $\Delta u +\lambda
f (u) =0$}, Soviet. Math. Dokl. { 6}, (1965), 1408--1411.



\bibitem{rey0}
Rey, O.  {\em Boundary effect for an elliptic Neumann problem with
critical nonlinearity, } Comm. Part. Diff. Equ. 22 (1997),
1055--1139.


\bibitem{rey1}
Rey, O.  {\em An elliptic Neumann problem with critical nonlinearity
in three dimensional domains.} Comm. Contemp. Math. 1, 405--449
(1999).

\bibitem{rey2}
Rey, O. {\em The question of interior blow-up points for an elliptic
Neumann problem: the critical case.} J. Math. Pures Appl. 81,
655--696 (2002).



\bibitem{rey}
Rey, O. {\em The role of the Green's function in a nonlinear
elliptic equation involving the critical Sobolev exponent.} J.
Funct. Anal. 89 (1990), no. 1, 1--52.

\bibitem{reywei}
 Rey, O., Wei, J. {\em Arbitrary number of positive solutions for an elliptic problem with critical
nonlinearity.} J. Eur. Math. Soc. (JEMS) 7, 449--476 (2005).

\bibitem{Schoen-Yau} R. Schoen and S.T. Yau, Lectures on Differential Geometry, International Press (1994).

\bibitem{wang}
Wang, X.-J. {\em Neumann problem of semilinear elliptic equations
involving critical Sobolev exponents.} J. Differential Equations 93,
283--310 (1991).



\bibitem{zqwang1}
Wang, Z.Q. {\em High energy and multi-peaked solutions for a
nonlinear Neumann problem with critical exponent.} Proc. Roy. Soc.
Edimburgh 125A, 1003--1029 (1995).


\bibitem{wwy}
Wang, L., Wei, J., Yan, S. {\em  A Neumann problem with critical
exponent in nonconvex domains and Lin-Ni's conjecture.}
 Trans. Amer. Math. Soc. 362 (2010), no. 9, 4581--4615.



\bibitem{wei}
Wei J., {\em On the boundary spike layer solutions of a singularly
perturbed semilinear Neumann problem,} J. Differential Equations 134
(1997) 104--133.



\bibitem{weixu}
 Wei, J., Xu, X. {\em Uniqueness and a priori estimates for some nonlinear elliptic Neumann equations
in R3.} Pacific J. Math. 221, 159--165 (2005).











\end{thebibliography}
\end{document}